\documentclass[12pt]{article}
\usepackage[utf8]{inputenc}
\usepackage[spanish,es-lcroman]{babel}
\usepackage{graphicx,amsmath,amsfonts,amsthm}
\usepackage[T1]{fontenc}
\usepackage{listings}
\usepackage{mathrsfs}
\usepackage{thmtools}
\usepackage{amssymb,enumitem}
\usepackage{comment}
\usepackage[all]{xy}
\usepackage{tikz}
\usepackage{tikz-cd}
\usetikzlibrary{patterns,decorations.pathreplacing,babel,decorations.markings,shapes,positioning,intersections,quotes,arrows}
\usepackage{faktor}
\usepackage{xfrac}
\usepackage{afterpage}
\usepackage{hyperref}
\usepackage{xspace}

\numberwithin{equation}{section}

\newtheorem{theorem}{Teorema}[section]
\newtheorem{lemma}[theorem]{Lema}
\newtheorem{proposition}[theorem]{Proposición}
\newtheorem{corollary}[theorem]{Corolario}

\theoremstyle{definition}
\newtheorem{definition}[theorem]{Definición}
\newtheorem{definitions}[theorem]{Definiciones}
\newtheorem{example}[theorem]{Ejemplo}
\newtheorem{examples}[theorem]{Ejemplos}
\newtheorem{observation}[theorem]{Observación}
\newtheorem{Observation}[theorem]{Observación Importante}
\newtheorem{notation}[theorem]{Notación}
\newtheorem{remark}[theorem]{Aclaración}
\newtheorem{comentario}[theorem]{Comentario}

\swapnumbers
\newtheorem{id_vertical}[theorem]{Identidades para la composici\'on vertical}

\theoremstyle{remark}
\newtheorem{duality}[theorem]{Dualidad}

\newcommand \catH {\mathcal{H}o(\mathscr{C})}
\newcommand \Ho {\mathcal{H}o_{fc}(\mathscr{C})}

\newcommand*\widebar[1]{%
   \hbox{%
     \vbox{%
       \hrule height 0.5pt %
       \kern0.5ex
       \hbox{%
         \kern-0.1em
         \ensuremath{#1}%
         \kern-0.1em%
       }%
     }%
   }%
}

\newdir{ >}{{}*!/-6pt/@{>}}

\begin{document}

\begin{center}

\textbf{\Large La 2-localizaci\'on de una categor\'ia de modelos}

\vspace{2ex}

Tesis de Licenciatura\footnote{Director Eduardo J. Dubuc, Departamento de Matematica, F.C.E. y
  N., Universidad de Buenos Aires.} 

\vspace{1ex}

Jaqueline Girabel

\end{center}

\vspace{1ex}

\begin{center} \textbf{Abstract} \end{center}

\vspace{-1.5ex}

In [\emph{Homotopical Algebra}, Springer LNM 43] Quillen introduces the notion of a \emph{model category}: a category $\mathcal{C}$ provided with three distinguished classes of maps $\{\mathcal{W},\, \mathcal{F},\, co\mathcal{F}\}$ (weak equivalences, fibrations, cofibrations), and gives a construction of the localization 
$\mathcal{C}[\mathcal{W}^{-1}]$ as the quotient of $\mathcal{C}$ by the congruence relation determined by the homotopies on the sets of arrows $\mathcal{C}(X,\,Y)$.

We develop here the 2-categorical localization, in which the 2-cells of this 2-localization are given by homotopies, and one can get the Quillen's localization when applying the connected components functor $\pi_0$ on the hom-categories of the 2-localization. Our proof is not just a generalization of the well-known Quillen's one. We work with definitions of cylinders
and homotopies introduced in [M.E. Descotte, E.J. Dubuc, M. Szyld, \emph{Model bicategories and their homotopy bicategories}, 	arXiv:1805.07749 (2018)] considering only a single family of arrows $\Sigma$. When $\Sigma$ is the class $\mathcal{W}$ of weak equivalences of a model category, we get the Quillen's results.

\vspace{2ex}

\begin{center} \textbf{Resumen} \end{center}

\vspace{-1.5ex}

Quillen en [\emph{Homotopical Algebra}, Springer LNM 43] presenta el concepto de \emph{categor\'ia de modelos}: una categor\'ia $\mathcal{C}$ provista de tres clases de flechas $\{\mathcal{W},\, \mathcal{F},\, co\mathcal{F}\}$ (equivalencias debiles, fibraciones, cofibraciones), y construye la localizaci\'on $\mathcal{C}[\mathcal{W}^{-1}]$ como el cociente de $\mathcal{C}$ por la congruencia determinada por las homotop\'ias en los conjuntos de morfismos $\mathcal{C}(X,\,Y)$.

Aquí  desarrollamos la localizaci\'on 2-categ\'orica, en la cual las homotopías determinan las 2-celdas de esta 2-localizaci\'on, y es posible obtener la localizaci\'on de Quillen tomando el funtor $\pi_0$ de componentes conexas en las categor\'ias de morfismos de la 2-localizaci\'on. Nuestra demostraci\'on no es una simple generalizaci\'on de la conocida demostraci\'on de Quillen. Se utilizan nuevas definiciones de cilindro y de homotop\'ia introducidas en [M.E. Descotte, E.J. Dubuc, M. Szyld, \emph{Model bicategories and their homotopy bicategories}, 	arXiv:1805.07749 (2018)] considerando una \'unica familia de morfismos $\Sigma$. Cuando $\Sigma$ es la clase $\mathcal{W}$ de equivalencias d\'ebiles de una categor\'ia de modelos, se obtienen los resultados de Quillen.

\section*{Introducci\'on}

\vspace{3mm}

El concepto de \emph{categor\'ia de modelos} fue introducido por Quillen previendo una de sus consecuencias m\'as significativas, y particularmente atractiva: la teor\'ia de homotop\'ia abstracta asociada a una estructura de modelos dada. Se define la \emph{categor\'ia homot\'opica} como la localizaci\'on de la categor\'ia original con respecto a una clase distinguida de morfismos que no necesariamente son inversibles, pero `` se asemejan '' a los isomorfismos .
Con el desarrollo de esta teor\'ia, Quillen provee una maquinaria ampliamente utilizada en diversos contextos.

\bigskip
En nuestra definici\'on de localizaci\'on, la categor\'ia con dicha propiedad universal est\'a determinada salvo equivalencia de categor\'ias, a diferencia de las definiciones de Gabriel y Zisman, y de Quillen, para quienes la localizaci\'on queda determinada en un sentido m\'as fuerte: salvo isomorfismo de categor\'ias.

\bigskip
Desarrollamos este trabajo con el objetivo de construir una \emph{2-categor\'ia homot\'opica} para una categor\'ia de modelos, provista de un 2-funtor con la propiedad universal de la 2-localizaci\'on de dicha categor\'ia con respecto a la clase de \emph{equivalencias d\'ebiles}. Presentamos, as\'i, una versi\'on 2-dimensional an\'aloga a la teor\'ia de homotop\'ia elaborada por Quillen para una categor\'ia de modelos. Aplicando el funtor $\pi_o$ de componentes conexas en las categor\'ias de flechas de la 2-localizaci\'on se obtienen los resultados de Quillen.

\bigskip
A su vez, entendemos esta exposici\'on como un caso particular de la versi\'on 2-dimensional original \cite{e.d.}, en la que Descotte, Dubuc y Szyld introducen el concepto de \emph{bicategor\'ia de modelos} y construyen la \emph{bicategor\'ia homot\'opica} asociada. Consideramos una
categoría de modelos como una bicategor\'ia de modelos trivial en su
estructura 2-dimensional. En el desarrollo de este caso particular se producen demostraciones m\'as simples que no son una mera adaptaci\'on de aquellas demostraciones correspondientes al caso general.

\bigskip
Dada una categor\'ia de modelos $\mathscr{C}$, las homotop\'ias de Quillen pueden interpretarse como las 2-celdas de la 2-categor\'ia buscada, en la que los objetos y las flechas son como en $\mathscr{C}$. Pero, aunque las homotop\'ias de Quillen se componen bajo ciertas condiciones, no determinan directamente una estructura de 2-categor\'ia.

\bigskip
Generalizando el concepto de cilindro de Quillen y, luego, el concepto de homotop\'ia de Quillen, una apropiada relaci\'on de equivalencia entre estas homotop\'ias nos permite definir una 2-categor\'ia en la que, efectivamente, objetos y morfismos son como en $\mathscr{C}$, y las \emph{clases de equivalencia de homotop\'ias} se componen vertical y horizontalmente, verificando todos los axiomas requeridos. De hecho, estos cilindros que generalizan los cilindros de Quillen, y con los que trabajamos en todo el desarrollo de la tesis, se definen en una categor\'ia $\mathscr{C}$ con una \'unica clase distinguida de flechas, $\Sigma$, conteniendo a todas las identidades, y la construcci\'on de esta 2-categor\'ia es independiente del contexto de categor\'ias de modelos. Nos referimos a ella con el nombre de \emph{2-categor\'ia homot\'opica de} $\mathscr{C}$, respecto de la clase $\Sigma$, y es denotada $\catH$. El 2-funtor $i: \mathscr{C} \longrightarrow \catH$, dado por la inclusi\'on, tiene la propiedad universal correspondiente a la 2-localizaci\'on de $\mathscr{C}$, pero no es precisamente la 2-localizaci\'on debido a que los morfismos de la clase $\Sigma$ no necesariamente son equivalencias en $\catH$.

\bigskip
Cuando $\mathscr{C}$ es una categor\'ia de modelos y $\Sigma$ es la clase de equivalencias d\'ebiles, que denotamos $\mathcal{W}$, estudiamos las condiciones bajo las cuales el 2-funtor $i: \mathscr{C} \longrightarrow \catH$ manda equivalencias d\'ebiles en equivalencias. Restringiendo $i$ a la subcategor\'ia $\mathscr{C}_{fc}$ de objetos fibrantes-cofibrantes, obtenemos una sub-2-categor\'ia $\Ho \subseteq \catH$, y podemos demostrar que la inclusi\'on $i: \mathscr{C}_{fc} \hookrightarrow \Ho$ es la 2-localizaci\'on de $\mathscr{C}_{fc}$ respecto de $\mathcal{W}$. Adem\'as, como una aplicaci\'on de esta construcci\'on, tomando el 2-funtor $\pi_0: \Ho \longrightarrow \pi(\Ho)$ de componentes conexas se tiene una categor\'ia $\pi(\Ho)$ isomorfa a la categor\'ia homot\'opica de Quillen de $\mathscr{C}_{fc}$, denotada $\pi\mathscr{C}_{fc}$ en \cite{Qui}.

\bigskip
Las conclusiones ya obtenidas para la subcategor\'ia  $\mathscr{C}_{fc}$ nos permiten ahora producir resultados concernientes a la categor\'ia $\mathscr{C}$ mediante un reemplazo fibrante-cofibrante. Podemos percibir, as\'i como en la exposici\'on de Quillen, el rol fundamental que tienen las clases $\mathcal{F}$, de \emph{fibraciones}, y $co\mathcal{F}$ , de \emph{cofibraciones} que hacen posible la construcci\'on de la 2-localizaci\'on con respecto a una tercera clase, $\mathcal{W}$.
Al momento de definir un 2-funtor con la propiedad universal de la 2-localizaci\'on de $\mathscr{C},$ asumimos que las factorizaciones de las que disponemos en una categor\'ia de modelos son funtoriales. De esta forma, denotando $\mathscr{C}_f$ y $\mathscr{C}_c$ a las subcategor\'ias de objetos fibrantes y cofibrantes de $\mathscr{C}$, respectivamente, los reemplazos fibrante $R: \mathscr{C} \longrightarrow \mathscr{C}_f$ y cofibrante $Q:\mathscr{C} \longrightarrow \mathscr{C}_c$ resultan funtores, m\'as que simples asignaciones, como lo son en el caso tratado por Quillen, que no asume funtorialidad en las factorizaciones. Consideramos, entonces, el 2-funtor $\xymatrix{\mathscr{C} \ar@/^2pc/[rrr]^q \ar[r]^Q & \mathscr{C}_c \ar[r]^R & \mathscr{C}_{fc}\ar[r]^(.4)i & \Ho}$, y demostramos que este tiene la propiedad universal esperada.

\bigskip
Aplicando el funtor de componentes conexas a la 2-localizaci\'on de $\mathscr{C}$, lo que obtenemos en este caso es una categor\'ia equivalente a la categor\'ia homot\'opica de $\mathscr{C}$ de Quillen, pero no isomorfa a ella. De todas maneras, esto se condice con nuestra definici\'on de localizaci\'on de categor\'ias, determinada salvo equivalencias.

\bigskip
Destacamos, por \'ultimo, que las homotop\'ias con las que trabajamos son siempre homotop\'ias a izquierda, y no necesitamos de las homotop\'ias a derecha para generar los resultados esperados. Esto se debe a la funtorialidad de los reemplazos fibrante y cofibrante. Imitando la construcci\'on de $\Ho$ considerando s\'olo homotop\'ias a derecha se tiene, en principio, otra versi\'on de la 2-localizaci\'on de $\mathscr{C}$, pero demostramos que coincide con la 2-categor\'ia obtenida con las homotop\'ias a izquierda. Estas conclusiones son una consecuencia de un resultado esencial que nos permite asegurar, adem\'as, que las categor\'ias de flechas $Hom(X, Y)$ de la 2-localizaci\'on de una categor\'ia de modelos son localmente pequeñas.

\newpage

\tableofcontents

\newpage

\pagenumbering{arabic}

\section{Preliminares}
En esta secci\'on incluiremos las definiciones, las notaciones y los resultados b\'asicos que ultilizaremos a lo largo de este trabajo. Luego de recordar los conceptos de equivalencia de categor\'ias y de localizaci\'on, adaptaremos estas nociones al contexto m\'as general de 2-categor\'ias.

\subsection{Equivalencias de categor\'ias y localizaci\'on}

 \begin{definition} \label{equi_cat}
     Sean $\mathscr{X}$ e $\mathscr{Y}$ dos categor\'ias. Un funtor $F: \mathscr{X} \longrightarrow \mathscr{Y}$ es una \emph{equivalencia de categor\'ias} si existen un funtor $G: \mathscr{Y} \longrightarrow \mathscr{X}$ e isomorfismos naturales $FG \simeq Id$ y $GF \simeq Id$. En tal caso, $G$ tambi\'en es una equivalencia de categor\'ias y decimos que $\mathscr{X}$ e $\mathscr{Y}$ son categor\'ias equivalentes.
     
     Decimos que $F$ es un \emph{isomorfismo de categor\'ias} si las transformaciones naturales $FG \simeq Id$ y $GF \simeq Id$ son las identidades.
 \end{definition}

 \begin{definitions}
     Sea $F: \mathscr{X} \longrightarrow \mathscr{Y}$ un funtor. 
     
     \vspace{2mm}
     1) Decimos que $F$ es
     \renewcommand{\labelenumi}{\roman{enumi}.}
     \begin{enumerate}
        \item \emph{pleno} si para todo $X, Y$ en $\mathscr{X}$, $F: \mathscr{X}[X, Y] \longrightarrow \mathscr{Y}[FX, FY]$ es suryectivo,
         \item \emph{fiel} si para todo $X, Y$ en $\mathscr{X}$, $F: \mathscr{X}[X, Y] \longrightarrow \mathscr{Y}[FX, FY]$ es inyectivo,
         \item \emph{plenamente fiel} si a la vez pleno y fiel.
     \end{enumerate}
     
     \vspace{3mm}
     
     2) Decimos que $F$ es \emph{esencialmente suryectivo} si todo objeto de $\mathscr{Y}$ es isomorfo a uno de la forma $FX$ para $X$ en $\mathscr{X}$.
 \end{definitions}
 
Una caracterizaci\'on importante de las equivalencias de categor\'ias es la siguiente (\cite{Mac} Ch.IV $\S$4):
 
 \begin{proposition}
     Consideramos un funtor $F: \mathscr{X} \longrightarrow \mathscr{Y}$. Son equivalentes:
     \renewcommand{\labelenumi}{\roman{enumi}.}
     \begin{enumerate}
         \item $F$ es una equivalencia de categor\'ias;
         \item $F$ es plenamente fiel y esencialmente suryectivo.
     \end{enumerate}
 \end{proposition}

\vspace{2mm}
Recordemos la construcci\'on de los anillos de fracciones: dado un anillo conmutativo con unidad $A$ y un subconjunto multiplicativo $S \subset A$ tal que $0 \notin S$, localizar $A$ con respecto a $S$ es construir de alg\'un modo un anillo $A_S$ que contenga a los inversos multiplicativos de los elementos en $S$: se tiene un morfismo de anillos $\lambda: A \longrightarrow A_S$ tal que $\lambda(s)$ es inversible para todo $s \in S$, y cualquier otro morfismo de anillos $\phi : A \longrightarrow B$ con esta propiedad se extiende a $A_S$ en forma \'unica.

\begin{center}
    $\xymatrix{ \vspace{2mm} \ar@{}[d]|{\exists \phi(a)^{-1} \hspace{1mm} \forall a \in S} & A \ar@{^{(}->}[r]^{\lambda} \ar[dr]_(.4){\phi} & A_S \ar@{-->}[d]^{\exists ! \bar{\phi} \text{ / } \bar{\phi}\lambda= \phi }\\
    && B \\
    }$
\end{center}

Notar que la misma definici\'on tiene sentido sin requerir que $S$ sea multiplicativo, y cuando $0 \in S$ se tiene que $A_S$ es el anillo trivial.

Como un caso particular disponemos del \'algebra de g\'ermenes de funciones $\mathscr{D}_p(M)$, donde $M$ es una variedad diferencial y $p \in M$, que guarda la informaci\'on local alrededor de $p$. Este es un anillo local cuyo \'unico ideal maximal $\widebar{m}_p$ consiste de los g\'ermenes de funciones que se anulan en $p$. Resulta que $\mathscr{D}_p(M)$ es la localizaci\'on del anillo $\mathscr{C}^\infty (M)$ con respecto al complemento del ideal maximal $m_p=\{ f \in \mathscr{C}^\infty (M) / f(p)=0 \}$. Notar que si $f(p)\neq 0$, entonces $f$ no se anula en todo un entorno de $p.$
\\

Dada una categor\'ia $\mathscr{C}$ y una clase $\Sigma$ de flechas en $\mathscr{C}$, la idea de la localizaci\'on consiste en encontrar una nueva categor\'ia $\mathscr{C}[\Sigma^{-1}]$ que aproxime o extienda a la categor\'ia original adjuntando los inversos de los morfismos en $\Sigma$, en el sentido de asegurar la existencia de un funtor $q: \mathscr{C} \longrightarrow \mathscr{C}[\Sigma^{-1}]$ que mande los elementos de $\Sigma$ en isomorfismos, de manera tal que el par $(\mathscr{C}[\Sigma^{-1}], q)$ sea universal con esta propiedad.

\vspace{1mm}
\begin{definition} \label{localization}
    Sean $\mathscr{C}$ una categor\'ia y $\Sigma$ una subclase de morfismos de $\mathscr{C}$. La \emph{localizaci\'on} de $\mathscr{C}$ con respecto a $\Sigma$ es una categor\'ia $\mathscr{C}[\Sigma^{-1}]$ junto con un funtor $q: \mathscr{C} \longrightarrow \mathscr{C}[\Sigma^{-1}]$ tales que:
    \begin{enumerate}
        \item para todo $s \in \Sigma$, $q(s)$ es un isomorfismo;
        \item para toda categor\'ia $\mathscr{D}$, $q$ induce una equivalencia en las categor\'ias de funtores dada por la precomposici\'on
        \begin{center}
            $\xymatrix{
            Hom(\mathscr{C}[\Sigma^{-1}], \mathscr{D}) \ar[r]^{q^*} & Hom_+(\mathscr{C}, \mathscr{D}),}$
        \end{center}
        donde los elementos de $Hom_+(\mathscr{C}, \mathscr{D})$ son aquellos funtores que mandan la clase $\Sigma$ en la clase $Isos(\mathscr{D})$.
    \end{enumerate}
\end{definition}

\begin{observation}
    Si $\mathscr{C}[\Sigma^{-1}]$ existe, es \'unica salvo equivalencia de categor\'ias.
     Gabriel y Zisman (\cite{GZ}), quienes introducen este concepto, y muchos autores despu\'es, requieren que $q^*$ en la definici\'on \ref{localization} sea un \emph{isomorfismo} de categor\'ias.
\end{observation}

\begin{comentario}
    Una construcci\'on formal permite ver que $\mathscr{C}[\Sigma^{-1}]$ siempre existe, pero sus hom-sets no necesariamente son conjuntos, a\'un cuando la categor\'ia $\mathscr{C}$ s\'i es localmente pequeña. En el contexto de categor\'ias de modelos, la construcci\'on de Quillen (\cite{Qui}) muestra que la localizaci\'on de una categor\'ia de modelos localmente pequeña tambi\'en es localmente pequeña.
    
    Cuando la clase $\Sigma$ satisface ciertas condiciones, an\'alogas a las propiedades de un subconjunto multiplicativo en un anillo no conmutativo, $\mathscr{C}[\Sigma^{-1}]$ puede describirse en t\'erminos de ``fracciones''; es decir, se construye una \emph{categor\'ia de fracciones} que satisface la definici\'on de localizaci\'on. Dicha descripci\'on de la localizaci\'on $\mathscr{C}[\Sigma^{-1}]$ es \'util y necesaria en la pr\'actica para demostrar una serie de propiedades, pero el problema conjunt\'istico en los hom-sets que mencionamos anteriormente a\'un persiste con esta definici\'on expl\'icita.

\end{comentario}

\subsection{2-Categor\'ias}

El concepto de 2-categor\'ia se entiende como una generalizaci\'on de las categor\'ias, admitiendo, adem\'as de objetos y morfismos, una estructura adicional dada por las \emph{2-celdas} (o \emph{2-morfismos}), que pueden componerse de dos maneras distintas respetando cierta condicion de compatibilidad. Un ejemplo que clarifica esta noci\'on de 2-celdas es el de las transformaciones naturales, que proveen a la categor\'ia $\mathbf{Cat}$ de una estructura de 2-categor\'ia, cuyos objetos son las categor\'ias (pequeñas) y los morfismos son los funtores.

\begin{definition}
    Una 2-categor\'ia $\mathscr{C}$ consiste de:
    \begin{enumerate}
        \item Una familia de objetos $X$, $Y$, $Z$, ...
        \item Para cada par de objetos $X$, $Y$ en $\mathscr{C}$, una categor\'ia $\mathscr{C}[X, Y]$. Los objetos de esta categor\'ia son flechas $f: X \longrightarrow Y$ mientras que los morfismos son 2-celdas $\alpha: f \Longrightarrow g$ entre flechas de $X$ a $Y$.
        La composici\'on en esta categor\'ia es la \emph{composici\'on vertical}, denotada por $\circ$.
        La identidad de una flecha $f$ es la 2-celda denotada ``$Id_f $''.
        
        Para cada objeto $X$ se tiene un morfismo distinguido $id_X \in \mathscr{C}[X, X]$.
        \item Dados objetos $X$, $Y$, $Z$, un funtor $\mathscr{C}[Y, Z] \times \mathscr{C}[X, Y] \longrightarrow \mathscr{C}[X, Z]$ que define una \emph{composici\'on horizontal} asociativa (tanto en los morfismos como en las 2-celdas), y es denotada por $*$.
        Las identidades para esta composici\'on son las flechas $id_X$ y las 2-celdas $Id_{id_X}$, para todo $X$.
        Adem\'as, para cada par de morfismos $f:X \longrightarrow Y$ y $g: Y \longrightarrow Z$, dicha composici\'on horizontal satisface $Id_g*Id_f= Id_{g*f}$.
        \item Un axioma de compatibilidad entre las composiciones horizontal y vertical, conocido como ``Interchange law'':
        
        \vspace{3mm}
        Dada una configuraci\'on en $\mathscr{C}$
        \begin{center}
             \begin{tikzcd}
                  X \arrow[shift left=20pt, "\hspace{6mm}f"]{rr}[name=LUU, below]{}  \arrow["\hspace{6mm} g"']{rr}[name=LUD]{}
                  \arrow[swap]{rr}[name=LDU]{}
                  \arrow[shift right=20pt, "\hspace{6mm}l"']{rr}[name=LDD]{}
                  \arrow[Rightarrow,to path=(LUU) -- (LUD)\tikztonodes]{rr}{\hspace{-6mm} \alpha}
                  \arrow[Rightarrow,to path=(LDU) -- (LDD)\tikztonodes]{rr}{\hspace{-6mm} \beta}
                  &&\hspace{1mm} Y
                  \arrow[shift left=20pt, "\hspace{6mm}f'"]{rr}[name=RUU, below]{}
                  \arrow["\hspace{6mm}g'"']{rr}[name=RUD]{}    
                  \arrow[swap]{rr}[name=RDU]{}
                  \arrow[shift right=20pt, "\hspace{6mm}l'"']{rr}[name=RDD]{}
                  \arrow[Rightarrow,to path=(RUU) -- (RUD)\tikztonodes]{r}{\hspace{-6mm} \alpha '}
                  \arrow[Rightarrow,to path=(RDU) -- (RDD)\tikztonodes]{r}{\hspace{-6mm}\beta '}
                  &&\hspace{1mm} Z
            \end{tikzcd}
        \end{center}
       se verifica $(\beta ' * \beta)\circ(\alpha ' * \alpha)=(\beta ' \circ \alpha')*(\beta \circ \alpha).$ 
    \end{enumerate}
\end{definition}

\vspace{2mm}
\begin{notation} \label{notation}
    En general, las identidades de la composici\'on vertical se denotar\'an simplemente como flechas ``$f$ '' en lugar de ``$Id_f$''. Adem\'as, omitiremos la notaci\'on $*$ cuando las 2-celdas se compongan horizontalmente con los morfismos. De manera que, si por ejemplo $\alpha$ e $Id_f$ son componibles en el sentido $Id_f * \alpha = f*\alpha$, muchas veces escribiremos $f\alpha$.
\end{notation}

\begin{observation} \label{horizontal_comp}
    En presencia de composici\'on vertical, para determinar una composici\'on horizontal de 2-celdas compatible basta definir una composici\'on horizontal entre 2-celdas y morfismos (interpretando a los morfismos como 2-celdas identidades), a la que algunos autores llaman \emph{wishkering}:
    
    Para cada \begin{tikzcd}
    X \arrow["l"]{r} 
    & Y \arrow[shift left=10pt, "f"]{r}[name=LUU, below]{}
    \arrow[shift right=10pt, "g"']{r}[name=LDD]{}
    \arrow[Rightarrow,to path=(LUU) -- (LDD)\tikztonodes]{r}{\alpha}
    & \hspace{1mm} Z \arrow["r"]{r}
    & W,
    \end{tikzcd}
    supongamos que tenemos definidas 2-celdas $r\alpha$ y $\alpha l$. Si se verifican los axiomas
    \begin{enumerate}
        \item Para cada $\xymatrix{X \ar[r]^f & Y\ar[r]^g & Z}$, se tiene $Id_gf=gId_f=Id_{gf}$; 
        \item Para cada \begin{tikzcd}
                  X \arrow["l"]{r}
                  & Y \arrow[shift left=20pt, "\hspace{6mm}f"]{r}[name=LUU, below]{}  \arrow["\hspace{6mm} g"']{r}[name=LUD]{}
                  \arrow[swap]{r}[name=LDU]{}
                  \arrow[shift right=20pt, "\hspace{6mm}k"']{r}[name=LDD]{}
                  \arrow[Rightarrow,to path=(LUU) -- (LUD)\tikztonodes]{r}{\hspace{-6mm} \alpha}
                  \arrow[Rightarrow,to path=(LDU) -- (LDD)\tikztonodes]{r}{\hspace{-6mm} \beta}
                  & Z \arrow["r"]{r}
                  & W
            \end{tikzcd}, vale $(\beta l) \circ (\alpha l)=(\beta \circ \alpha)l$ y $(r\alpha) \circ (r\beta)=r(\beta \circ \alpha);$
        \item Para cada
            \begin{tikzcd}
                X \arrow[shift left=10pt, "f"]{r}[name=LUU, below]{}
                \arrow[shift right=10pt, "g"']{r}[name=LDD]{}
                \arrow[Rightarrow,to path=(LUU) -- (LDD)\tikztonodes]{r}{\alpha}
                & \hspace{1mm} Y \arrow[shift left=10pt, "f'"]{r}[name=RUU, below]{}
                \arrow[shift right=10pt, "g'"']{r}[name=RDD]{}
                \arrow[Rightarrow,to path=(RUU) -- (RDD)\tikztonodes]{r}{\alpha'}
                & Z,
            \end{tikzcd}
    $(g'\alpha) \circ (\alpha' f) = (\alpha' g) \circ (f'\alpha)$
    \end{enumerate}
    entonces cualquier composici\'on horizontal $\alpha' * \alpha$ como en el \'ultimo axioma queda definida por $(g'\alpha) \circ (\alpha' f) = (\alpha' g) \circ (f'\alpha).$

\end{observation}

    \vspace{2mm}
    El ejemplo protot\'ipico de 2-categor\'ia es el de $\mathbf{Cat}$, en donde las 2-celdas son las transformaciones naturales.
    La composici\'on vertical de transformaciones naturales $\tau: F \Longrightarrow G$, $\sigma: G \Longrightarrow R$, donde $F, G, R$ son funtores con iguales dominio y codomio, se define como $(\sigma \circ \tau)_X = \sigma_X \tau_X$. Por otro lado,
    si tenemos
    \begin{tikzcd}
     \mathscr{C} \arrow[shift left=10pt, "F"]{r}[name=LUU, below]{}
    \arrow[shift right=10pt, "G"']{r}[name=LDD]{}
    \arrow[Rightarrow,to path=(LUU) -- (LDD)\tikztonodes]{r}{\alpha}
    & \hspace{1mm} \mathscr{D} \arrow[shift left=10pt, "R"]{r}[name=RUU, below]{}
    \arrow[shift right=10pt, "S"']{r}[name=RDD]{}
    \arrow[Rightarrow,to path=(RUU) -- (RDD)\tikztonodes]{r}{\beta}
    & \mathscr{E},
    \end{tikzcd}
    la composici\'on horizontal es definida como la tranformaci\'on natural cuyas componentes son $(\beta \circ \alpha)_X= \beta_{GX}R\alpha_X$ para cada $X$ en $\mathscr{C}$. Notamos que de la conmutatividad del diagrama
    \begin{center}
        $\xymatrix{
            RFX \ar[r]^{\beta_{FX}} \ar[d]_{F\alpha_X} & SFX \ar[d]^{S\alpha_X}\\
            RGX \ar[r]_{\beta_{GX}} & SGX\\
        }$
    \end{center}
    tambi\'en $(\beta \circ \alpha)_X= S\alpha_X \beta_{FX}$.

    \vspace{2mm}
\begin{definitions} \label{def_pseudo}
    Un \emph{2-funtor} $F:\mathscr{C} \longrightarrow \mathscr{D}$ entre 2-categor\'ias manda objetos de $\mathscr{C}$ en objetos de $\mathscr{D}$, morfismos de $\mathscr{C}$ en morfismos de $\mathscr{D}$ y 2-celdas de $\mathscr{C}$ en 2-celdas de $\mathscr{D}$, preservando todas las estructuras: composiciones vertical y horizontal e identidades.
    
    \bigskip
    Si $G$ es otro 2-funtor entre las mismas 2-categor\'ias, una \emph{transformaci\'on 2-natural} $\eta: F \Longrightarrow G$ consiste de una familia de flechas $\eta_X:FX \longrightarrow GX$ en $\mathscr{D}$, con $X$ variando en $\mathscr{C}$, de forma tal que para cada 2-celda $\alpha: f \Longrightarrow g$ se verifique la ecuaci\'on $\eta_YF(\alpha) = G(\alpha)\eta_X$:
    
    \begin{center}
        $\xymatrix{
                    FX \ar[rr]^{\eta_X} \ar@<-10pt>[dd]_{Fg} \ar@<10pt>[dd]^{Ff} \ar@{}[dd]|{\underset{\Longleftarrow}{F(\alpha)}} && GX \ar@<-10pt>[dd]_{Gg} \ar@<10pt>[dd]^{Gf} \ar@{}[dd]|{\underset{\Longleftarrow}{G(\alpha)}}\\
                    \\
                    FY \ar[rr]_{\eta_Y} && GY}$
    \end{center}

    Cuando $\alpha$ es una identidad, la igualdad anterior es la conmutatividad del diagrama correspondiente a la naturalidad de las transformaciones naturales que ya conoc\'iamos.
    
    M\'as generalmente, decimos que $\eta: F \Longrightarrow G$ es una \emph{transformaci\'on pseudonatural} entre 2-funtores si asigna a cada objeto $X$ en $\mathscr{C}$ una flecha $\eta_X: FX \longrightarrow GX$ en $\mathscr{D}$ y a cada flecha $f: X \longrightarrow Y$ en $\mathscr{C}$ una 2-celda inversible $\eta_f: Gf\eta_X \longrightarrow \eta_YFf$ en $\mathscr{D}$
    
    \begin{center}
        $\xymatrix{
            FX \ar[r]^{\eta_X} \ar[d]_{Ff} & GX \ar[d]^{Gf} \ar@{}[dl]|{\rotatebox[origin=c]{45}{$\Leftarrow$} \eta_f}\\
            GY \ar[r]_{\eta_Y} & GY}$
    \end{center}
    satisfaciendo los siguientes axiomas:
    \begin{enumerate}
        \item Para cada $X$ en $\mathscr{C}$, $\eta_{id_X}=Id_{\eta_X}$;
        \item Dadas $\xymatrix{ X \ar[r]^{f} & Y \ar[r]^{g} & Z }$ en $\mathscr{C}$, se tiene que $\eta_g F(f)\circ G(g) \eta_f=\eta_{gf}$, de acuerdo a \ref{notation}.
        \begin{center}
        $\xymatrix{
            FX \ar[r]^{\eta_X} \ar[d]_{Ff} & GX \ar[d]^{Gf} \ar@{}[dl]|{\rotatebox[origin=c]{45}{$\Leftarrow$} \eta_f}\\
            FY \ar[r]_{\eta_Y} \ar[d]_{Fg} & GY \ar[d]^{Gg}\ar@{}[dl]|{\rotatebox[origin=c]{45}{$\Leftarrow$} \eta_g} & =\hspace{3mm} \\
            FZ \ar[r]_{\eta_Z} & GZ\\}$
            $\xymatrix{
            FX \ar[r]^{\eta_X} \ar[dd]_{Fgf} & GX \ar[dd]^{Ggf} \ar@{}[ddl]|{\rotatebox[origin=c]{45}{$\Leftarrow$} \eta_{gf}}\\
            \\
            FZ \ar[r]_{\eta_Z} & GZ}$
        \end{center}
        \item Para cada 2-celda $\alpha: f \Longrightarrow g : X \longrightarrow Y$ en $\mathscr{C}$, vale la ecuaci\'on $\eta_g \circ G(\alpha)\eta_X=\eta_YF(\alpha) \circ \eta_f:$
            \begin{center}
                 $\xymatrix{
                    FX \ar[rr]^{\eta_X} \ar[dd]_{Fg} && GX \ar@<-10pt>[dd]_{Gf} \ar@<10pt>[dd]^{Gg \hspace{5mm} = \hspace{5mm}} \ar@{}[dd]|{\underset{\Longrightarrow}{G(\alpha)}} \ar@{}[ddll]|{\rotatebox[origin=c]{45}{$\Leftarrow$} \eta_{g}} \\
                    \\
                    FY \ar[rr]_{\eta_Y} && GY}$
                    $\xymatrix{
                    FX \ar[rr]^{\eta_X} \ar@<-10pt>[dd]_{Fg} \ar@<10pt>[dd]^{Ff} \ar@{}[dd]|{\underset{\Longleftarrow}{F(\alpha)}} && GX \ar[dd]^{Gf} \ar@{}[ddll]|{\rotatebox[origin=c]{45}{$\Leftarrow$} \eta_f} \\
                    \\
                    FY \ar[rr]_{\eta_Y} && GY}$
            \end{center}

    \end{enumerate}
    Notamos que una transformaci\'on 2-natural es una transformaci\'on pseudonatural tal que $\eta_f$ es la identidad para cada $f$, en cuyo caso las primeras dos condiciones son triviales y la tercera es el axioma de 2-naturalidad que mencionamos anteriormente.
    
    \bigskip
    Supongamos ahora que $\tau,\sigma: F \Longrightarrow G$ son transformaciones 2-naturales entre 2-funtores $F, G: \mathscr{C} \Longrightarrow \mathscr{D}$. Una \emph{modificaci\'on} $\mu : \tau \longrightarrow \sigma$ asigna a cada objeto $X$ de $\mathscr{C}$ una 2-celda $\mu_X : \tau_X \Longrightarrow \sigma_X$ en $\mathscr{D}$, de manera tal que para todo par de flechas $f, g: X \longrightarrow Y$ y para toda 2-celda $\alpha : f\Longrightarrow g$ en $\mathscr{C}$ se verifica la igualdad $\mu_YF(\alpha) = G(\alpha) \mu_X$:
    
\begin{center}
    \begin{tikzcd}
     FX \arrow[shift left=13pt, "Ff"]{r}[name=LUU, below]{}
    \arrow[shift right=13pt, "F g"']{r}[name=LDD]{}
    \arrow[Rightarrow,to path=(LUU) -- (LDD)\tikztonodes]{r}{F\alpha}
    & \hspace{1mm} FY \arrow[shift left=13pt, "\tau_Y"]{r}[name=RUU, below]{}
    \arrow[shift right=13pt, "\sigma_Y"']{r}[name=RDD]{}
    \arrow[Rightarrow,to path=(RUU) -- (RDD)\tikztonodes]{r}{\mu_Y}
    & GY
    \end{tikzcd}
    $=$
    \begin{tikzcd}
     FX \arrow[shift left=13pt, "\tau_X"]{r}[name=LUU, below]{}
    \arrow[shift right=13pt, "\sigma_X"']{r}[name=LDD]{}
    \arrow[Rightarrow,to path=(LUU) -- (LDD)\tikztonodes]{r}{\mu_X}
    & \hspace{1mm} GX \arrow[shift left=13pt, "Gf"]{r}[name=RUU, below]{}
    \arrow[shift right=13pt, "Gg"']{r}[name=RDD]{}
    \arrow[Rightarrow,to path=(RUU) -- (RDD)\tikztonodes]{r}{G\alpha}
    & GY.
    \end{tikzcd}
\end{center}

    Si $\tau$ y $\sigma$ son pseudonaturales, una modificaci\'on $\mu: \tau \longrightarrow \sigma$ ser\'a una asignaci\'on como antes satisfaciendo la igualdad $ (\mu_YF(\alpha))\circ \tau_f = \sigma_g \circ (G(\alpha) \mu_X)$ para cada $\alpha: f \Longrightarrow g$ en $\mathscr{C}$.
    
\end{definitions}

\vspace{3mm}
\begin{observation} \label{strict_&_pseudo}
Las transformaciones 2-naturales (pseudonaturales) y las modificaciones se componen verticalmente y horizontalmente (ver, por ejemplo, \cite{Bak} I.2.4, p. 25).

Si $\mathscr{C}$ y $\mathscr{D}$ son 2-categor\'ias, entonces $Hom(\mathscr{C}, \mathscr{D})$ es la 2-categor\'ia de 2-funtores, que puede ser interpretada de dos maneras a partir de las definiciones anteriores:

En el sentido d\'ebil, denotaremos $Hom_p(\mathscr{C}, \mathscr{D})$ a la 2-categor\'ia en la que los objetos son los 2-funtores de $\mathscr{C}$ en $\mathscr{D}$, las flechas son las transformaciones pseudonaturales y las 2-celdas son las modificaciones.

En el sentido estricto, $Hom_s(\mathscr{C}, \mathscr{D})$ ser\'a la 2-categor\'ia cuyos objetos son los 2-funtores de $\mathscr{C}$ en $\mathscr{D}$, las flechas son las transformaciones 2-naturales y las 2-celdas son las modificaciones.

Notamos que se tiene un 2-funtor $ Hom_s(\mathscr{C}) \rightarrow Hom_p(\mathscr{C})$ que es fiel pero no es pleno. Y para cualquier par de 2-funtores $F$, $G$, el funtor $$Hom_s(\mathscr{C})[F, G] \rightarrow Hom_p(\mathscr{C})[F, G]$$ s\'i es plenamente fiel.
\end{observation}

\subsection{Equivalencias de 2-categor\'ias y 2-localizaci\'on}

Decimos que una flecha $f: X \longrightarrow Y$ es una \emph{equivalencia} si existe otra flecha $g: Y \longrightarrow X$ que es interpretada como una ``inversa'' de $f$ en el sentido m\'as general posible, seg\'un el contexto, y que recibe el nombre de \emph{cuasi-inversa} de $f$:

\bigskip
Si $f$ es un morfismo en una categor\'ia, entonces es una equivalencia si tiene una cuasi-inversa $g$ que es una verdadera inversa (es decir, $f$ es un isomorfismo).

\bigskip
Si $f$ es un morfismo en una 2-categor\'ia, entonces es una equivalencia si tiene una cuasi-inversa $g:Y \longrightarrow X$ en el sentido de que existen 2-celdas inversibles $id_X \Longrightarrow gf$, $fg \Longrightarrow id_Y $. Cuando estamos en la 2-categor\'ia $\mathbf{Cat}$, estas son precisamente las equivalencias de categor\'ias consideradas en la definici\'on \ref{equi_cat}.

\begin{observation}
    La cuasi-inversa $g$ est\'a determinada salvo una 2-celda inversible.
    
    Siempre pueden elegirse una tal $g$ y 2-celdas inversibles $\alpha: id_X \Longrightarrow fg$ y $\beta: gf \Longrightarrow id_Y$ de forma tal que se verifiquen las \emph{ecuaciones triangulares}:
    \begin{center}
            $\xymatrix{
                & fgf \ar[dr]^{f\beta} \\
                f \ar[ur]^{\alpha f} \ar@{=}[rr] && f 
            }$
            \hspace{2mm} y \hspace{2mm}
            $\xymatrix{
                & gfg \ar[dr]^{g\alpha} \\
                g \ar[ur]^{\beta g} \ar@{=}[rr] && g. 
            }$
        \end{center}
    
\end{observation}

\bigskip
\begin{observation} \label{s*_equivalencia}
Si $s: X \longrightarrow Y$ en una 2-categor\'ia $\mathscr{C}$ es una equivalencia, entonces induce equivalencias en las categor\'ias de morfismos:

Para cada objeto $Z$ en $\mathscr{C}$, el funtor $s^*: \mathscr{C}[Y, Z] \longrightarrow \mathscr{C}[X, Z]$ dado por la precomposici\'on $f \longmapsto fs$ es una equivalencia de categor\'ias. En efecto, se tienen $t: Y \longrightarrow X$ y 2-celdas inversibles $\alpha :ts \Longrightarrow id_X$, $\beta: st \Longrightarrow id_Y$ que nos permiten definir isomorfismos naturales $\eta: Id_{\mathscr{C}[Y, Z]} \Longrightarrow (st)^*=t^*s^*$ y $\theta: Id_{\mathscr{C}[X, Z]} \Longrightarrow (ts)^*=s^*t^*$. Para cada $f$ en  $\mathscr{C}[Y, Z]$, tenemos una flecha $\eta_f:=f\beta : fst \Longrightarrow f$
(una 2-celda en $\mathscr{C}$) que es un isomorfismo por ser una composici\'on $Id_f*\beta$ de isomorfismos. Adem\'as, $\eta$ es natural en $f$ y esto es una consecuencia de la compatibilidad de las composiciones vertical y horizontal en $\mathscr{C}$. An\'alogamente, $\theta_g$ es un isomorfismo natural en $g \in Ob(\mathscr{C}[X, Z])$, y por lo tanto $t^*$ es una cuasi-inversa para $s^*$.

Del mismo modo puede verse que $s_*: \mathscr{C}[Z, X] \longrightarrow \mathscr{C}[Z, Y]$
es una equivalencia de categor\'ias.
\end{observation}

M\'as a\'un, se tiene la siguiente
\begin{proposition}
    Dada $s:X \longrightarrow Y$ en una 2-categor\'ia $\mathscr{C}$, entonces $s$ es una equivalencia si y s\'olo si $s^*: \mathscr{C}[Y, Z] \longrightarrow \mathscr{C}[X, Z]$ es una equivalencia para todo $Z$ en $\mathscr{C}$ si y s\'olo si $s_*: \mathscr{C}[Z, X] \longrightarrow \mathscr{C}[Z, Y]$ es una equivalencia para todo $Z$ en $\mathscr{C}.$
\end{proposition}

\vspace{4mm}
En una 3-categor\'ia, una flecha $f: X \longrightarrow Y$ ser\'a una equivalencia si tiene una cuasi-inversa $g:Y \longrightarrow X$ en el sentido de que existen 2-celdas $id_X \Longrightarrow gf$, $fg \Longrightarrow id_Y$ que son equivalencias en las 2-categor\'ias de flechas $Hom(X, X)$ y $Hom(Y, Y)$.

Notamos que $2$-$\mathbf{Cat}$ es una 3-categor\'ia. De esta manera, el concepto de equivalencia entre categor\'ias (es decir, una equivalencia en $\mathbf{Cat}$) se extiende a 2-categor\'ias, usualmente con el nombre de \emph{pseudoequivalencia} de 2-categor\'ias:

\begin{definitions} \label{equiv_2cat}
    Sean $\mathscr{C}, \mathscr{D}$ 2-categor\'ias, $F, G: \mathscr{C} \longrightarrow \mathscr{D}$ 2-funtores.
    \begin{enumerate}
    \item  Una transformaci\'on 2-natural $F \Longrightarrow G$ es una equivalencia si lo es en la 2-categor\'ia de 2-funtores $Hom_s(\mathscr{C}, \mathscr{D})$.
    
    \item Una transformaci\'on pseudonatural $F \Longrightarrow G$ es una equivalencia si lo es en la 2-categor\'ia $Hom_p(\mathscr{C}, \mathscr{D})$.
    \end{enumerate}
\end{definitions}

    \vspace{2mm}
    \begin{definition}
    Decimos que un 2-funtor $F: \mathscr{C} \longrightarrow \mathscr{D}$ es una \emph{pseudoequivalencia de 2-categor\'ias} si existen $G: \mathscr{D} \longrightarrow \mathscr{C}$ y tranformaciones pseudonaturales $\eta: Id_{\mathscr{C}} \Longrightarrow GF$, $\theta: FG \Longrightarrow Id_{\mathscr{D}}$ que son equivalencias en $Hom_p(\mathscr{C}, \mathscr{C})$ y en $Hom_p(\mathscr{D}, \mathscr{D})$, respectivamente.
    
    El 2-funtor $F$ ser\'a una pseudoequivalencia de 2-categor\'ias en el sentido estricto cuando $\eta$ y $\theta$ sean equivalencias en $Hom_s(\mathscr{C}, \mathscr{C})$ y en $Hom_s(\mathscr{D}, \mathscr{D})$, respectivamente.
    \end{definition}

\bigskip
Dada $\eta :\xymatrix{ F \ar@{=>}[r]& G }$, una transformaci\'on pseudonatural entre 2-funtores $F, G \in Hom(\mathscr{C}, \mathscr{D})$, tal que cada componente $\eta_X$ es una equivalencia con cuasi-inversa $\theta_X$ en la 2-categor\'ia $\mathscr{D}$, entonces es posible darle una estructura pseudonatural a la familia $\{\theta_X\}_X $ para definir as\'i una cuasi-inversa $\theta$ de $\eta$ en la 2-categor\'ia $Hom_p(\mathscr{C},  \mathscr{D}).$ Este hecho es utilizado frecuentemente en la literatura pero no hemos podido encontrar una demostraci\'on del mismo, es por ello que aquí incluimos una.

\begin{proposition} \label{prop_infinita}
   Sea $\eta: F \Longrightarrow G : \mathscr{C} \longrightarrow \mathscr{D}$ una transformaci\'on pseudonatural entre 2-funtores. Entonces $\eta$ es una equivalencia en $Hom_p(\mathscr{C}, \mathscr{D})$ si y s\'olo si cada componente $\eta_X$ es una equivalencia en la 2-categor\'ia $\mathscr{D}$.
\end{proposition}


\begin{proof}
    Si $\eta$ es una equivalencia con cuasi-inversa $\theta$, es claro que $\theta_X$ es una cuasi inversa para $\eta_X$, para cada $X$ en $\mathscr{C}.$
    
    \bigskip
    Supongamos ahora que cada componente $\eta_X$ es una equivalencia en $\mathscr{D}$ con cuasi-inversa $\theta_X$, y sean $\alpha_X: \theta_X\eta_X \Rightarrow id_{FX}$, $\beta_X: \eta_X \theta_X \Rightarrow id_{GX}$ inversibles. Podemos tomar $\alpha$ y $\beta$ satisfaciendo las identidades triangulares
    \begin{center}
            $\xymatrix{
                & \theta\eta\theta \ar@{=>}[dr]^{\theta\beta} \\
                \theta \ar@{=>}[ur]^{\alpha^{-1}\theta} \ar@{=}[rr] && \theta
            }$
            \hspace{5mm} y \hspace{5mm}
            $\xymatrix{
                & \eta\theta\eta \ar@{=>}[dr]^{\alpha\eta} \\
                \eta \ar@{=>}[ur]^{\beta^{-1}\eta} \ar@{=}[rr] && \eta. 
            }$
    \end{center}
    
    Queremos definir una transformaci\'on pseudonatural $\theta$ cuasi-inversa de $\eta$.
    
    Dado que $\eta$ es pseudonatural, para cada $f: X \longrightarrow Y$ en $\mathscr{C}$ existe una 2-celda inversible $\eta_f: Gf\eta_X \Longrightarrow \eta_YFf$ en $\mathscr{D}$
    \begin{center}
        $\xymatrix{
            FX \ar[rr]^{\eta_X} \ar[d]_{Ff} && GX \ar@{-->}@/_1.5pc/[ll]_{\theta_X}  \ar[d]^{Gf} \ar@{}[dll]|{\rotatebox[origin=c]{45}{$\Leftarrow$} \eta_f}\\
            FY \ar[rr]_{\eta_Y} && GY. \ar@{-->}@/^1.5pc/[ll]^{\theta_Y}\\
        }$
    \end{center}
    
    Definimos $\theta_f :Ff\theta_X \Longrightarrow \theta_YGf$ como la composici\'on
    \begin{center}
        $\xymatrix{
            Ff\theta_X \ar@{=>}[rr]^(.4){\alpha_Y^{-1}Ff\theta_X} && \theta_Y\eta_YFf\theta_X \ar@{=>}[rr]^{\theta_Y\eta_f^{-1}\theta_X} && \theta_YGf\eta_X\theta_X \ar@{=>}[rr]^{\theta_YGg\beta_X} && \theta_YGf. \\
        }$
    \end{center}
    
    Por ser una composici\'on de isomorfismos, $$\theta_f= (\theta_YGf\beta_X)\circ (\theta_Y\eta_f^{-1}\theta_X) \circ (\alpha_Y^{-1}Ff\theta_X)$$ tambi\'en es un isomorfismo.
    
    \vspace{3mm}
    Teniendo en cuenta los axiomas de la definici\'on de tranformaci\'on pseudonatural en \ref{def_pseudo}, veamos que $\theta$ satisface cada uno de estos.
    
    \vspace{4mm}
    \begin{enumerate}
        \item Se verifica $\theta_{id_X}=Id_{\theta_X}.$ En efecto, como $\eta_{id_X}=Id_{\eta_X}$ y se cumplen las identidades triangulares para $\alpha$ y $\beta$, entonces
        \begin{equation*}
        \begin{split}
            \theta_{id_X} & = (\theta_Xid_{GX}\beta_X)\circ (\theta_XId_{\eta_X}\theta_X) \circ (\alpha_X^{-1}Id_{FX}\theta_X)\\
                          \\
                          & = (\theta_X\beta_X)\circ (\theta_X\eta_X\theta_X)\circ (\alpha_X^{-1}\theta_X)\\
                          \\
                          & = (\theta_X\beta_X)\circ (\alpha_X^{-1}\theta_X) = Id_{\theta_X}.
        \end{split}
        \end{equation*}
        
        \vspace{4mm}
        \item Dadas $\xymatrix{ X \ar[r]^{f} & Y \ar[r] ^{g} & Z }$, se tiene el segundo axioma de pseudonaturalidad $(\theta_gGf)\circ (Fg\theta_f)=\theta_{gf}:$ 
        
        \vspace{1mm}
        
        Como $\eta$ es pseudonatural, entonces $\eta_{gf}^{-1}=(Gg\eta_f^{-1})\circ (\eta_g^{-1}Ff)$ y escribimos
        
            \begin{equation} \label{axioma_2_bis}
                \begin{split}
                    \theta_{gf} & = (\theta_ZGgGf\beta_X)\circ (\theta_Z\eta_{gf}^{-1}\theta_X)\circ (\alpha_Z^{-1}FgFf\theta_X) \\
                                \\
                                & = (\theta_ZGgGf\beta_X)\circ (\theta_ZGg\eta_{f}^{-1}\theta_X)\circ (\theta_Z\eta_{g}^{-1}Ff\theta_X)\circ (\alpha_Z^{-1}FgFf\theta_X)\\
                \end{split}
            \end{equation}
            y
            \begin{equation} \label{axioma_2}
                \begin{split}
                   (\theta_gGf)\circ (Fg\theta_f) & =(\theta_ZGg\beta_YGf)\circ [\text{ }(\theta_Z\eta_g^{-1}\theta_YGf) \circ (\alpha_Z^{-1}Fg\theta_YGf)\text{ }]\circ \\ & \hspace{5mm} [\text{ }(Fg\theta_YGf\beta_X) \circ (Fg\theta_Y\eta_f^{-1}\theta_X)\text{ }] \circ (Fg\alpha_Y^{-1}Ff\theta_X)\\
                    &\hspace{1mm}\\
                     & =(\theta_ZGg\beta_YGf) \circ [\text{ }(\theta_Z\eta_g^{-1}) \circ (\alpha_Z^{-1}Fg)\text{ }]\theta_YGf \circ \\ & \hspace{5mm} Fg\theta_Y[\text{ }(Gf\beta_X) \circ (\eta_f^{-1}\theta_X)\text{ }] \circ (Fg\alpha_Y^{-1}Ff\theta_X)\\
                     \\
                     & \hspace{-20mm} =(\theta_ZGg\beta_YGf)\circ [\text{ }(\theta_Z\eta_g^{-1}) \circ (\alpha_Z^{-1}Fg)\text{ }]\theta_Y[\text{ }(Gf\beta_X) \circ (\eta_f^{-1}\theta_X)\text{ }] \\ & \hspace{55mm} \circ (Fg\alpha_Y^{-1}Ff\theta_X).
                \end{split}
            \end{equation}        

        \vspace{5mm}
        El diagrama que sigue representa el lado derecho de la ecuaci\'on (\ref{axioma_2}). 
        
         \begin{center}
            \begin{tikzcd}
                GX \arrow[rrrr,"Ff\theta_X"{name=L, description}]{} &&&& FY \arrow[rr, "Fg"'{name=R, description}]{} && FZ\\
                GX \arrow[rrrr, "\theta_Y\eta_YFf\theta_X"{name=LL, description}]{} &&&& FY \arrow[rr, "Fg"{name=RR, description}]{} && FZ \arrow[Rightarrow, from=L, to=LL, "\alpha_Y^{-1}Ff\theta_X"]{} \arrow[Rightarrow, from=R, to=RR, "Id"]{} \\
                GX \arrow[rr, "\eta_YFf\theta_X"{name=LLL, description}]{} && GY \arrow[rr, "\theta_Y"{name=CCC, description}]{} \arrow[Rightarrow, from=LL, "Id"]{} && FY \arrow[rr, "Fg"{name=RRR, description}]{} \arrow[Rightarrow, from=RR, to=RRR]{} && FZ \\
                GX \arrow[rr, "Gf"{name=LLLL, description}]{} && GY \arrow[rr, "\theta_Y"{name=CCCC, description}]{} \arrow[rrrr, bend right=5, white, ""{name=A,below}]{} \arrow[Rightarrow, from=LLL, to=LLLL, "(Gf\beta_X) \circ (\eta_f^{-1}\theta_X)"]{} && FY \arrow[rr, "\theta_ZGg\eta_Y"{name=RRRR, description}]{}  \arrow[Rightarrow, from=CCC, to=CCCC, "Id"]{}  && FZ \arrow[Rightarrow, from=RRR, to=RRRR, "(\theta_Z\eta_g^{-1}) \circ (\alpha_Z^{-1}Fg)"]{}\\
                GX \arrow[rr, "Gf"{name=LLLLL, description}]{} && GY \arrow[rrrr, "\theta_ZGg\eta_Y\theta_Y"{name=RRRRR, description}]{} \arrow[Rightarrow, from=LLLL, to=LLLLL, "Id"]{} &&&& FZ  \arrow[Rightarrow, from=A, to=RRRRR, "Id"]{} \\
                GX \arrow[rr, "Gf"{name=LLLLLL, description}]{} && GY \arrow[rrrr, "\theta_ZGg"{name=RRRRRR, description}]{} \arrow[Rightarrow, from=LLLLL, to=LLLLLL, "Id"]{} &&&& FZ \arrow[Rightarrow, from=RRRRR, to=RRRRRR, "\theta_ZGg\beta_Y"]{} \\
            \end{tikzcd}
        \end{center}
        
        Como consecuencia de la compatibilidad entre las composiciones horizontal y vertical, puede reescribirse de la siguiente manera:
        \begin{center}
            \begin{tikzcd}
                GX \arrow[rrrr,"Ff\theta_X"{name=L, description}]{} &&&& FY \arrow[rr, "Fg"'{name=R, description}]{} && FZ\\
                GX \arrow[rrrr, "Ff\theta_X"{name=LL, description}]{} &&&& FY \arrow[rr, "\theta_ZGg\eta_Y"{name=RR, description}]{} && FZ \arrow[Rightarrow, from=L, to=LL, "Id"]{} \arrow[Rightarrow, from=R, to=RR, "(\theta_Z\eta_g^{-1}) \circ (\alpha_Z^{-1}Fg)"]{} \\
                GX \arrow[rrrr, "\theta_Y\eta_YFf\theta_X"{name=LLL, description}]{} \arrow[rrrrrr, white, bend right=4, ""{name=A, below}] &&&& FY \arrow[rr, "\theta_ZGg\eta_Y"{name=RRR, description}]{} \arrow[Rightarrow, from=LL, to=LLL, "\alpha_Y^{-1}Ff\theta_X"]{} && FZ \arrow[Rightarrow, from=RR, to=RRR, "Id"]{}\\
                GX \arrow[rr, "\eta_YFf\theta_X"{name=LLLL, description}]{} \arrow[rrrrrr, white, bend left=4, ""{name=B, below}]{} && GY \arrow[rrrr, "\theta_ZGg\eta_Y\theta_Y"{name=RRRR, description}]{} &&&& FZ \arrow[Rightarrow, from=A, to=B, "Id"]{} \\
                GX \arrow[rr, "\eta_YFf\theta_X"{name=LLLLL, description}]{} && GY \arrow[rrrr, "\theta_ZGg"{name=RRRRR, description}]{} \arrow[Rightarrow, from=LLLL, to=LLLLL, "Id"]{} &&&& FZ \arrow[Rightarrow, from=RRRR, to=RRRRR, "\theta_ZGg\beta_Y"]{}\\
                GX \arrow[rr, "Gf"{name=LLLLLL, description}]{} && GY \arrow[rrrr, "\theta_ZGg"{name=RRRRRR, description}]{}\arrow[Rightarrow, from=LLLLL, to=LLLLLL, "(Gf\beta_X)\circ (\eta_f^{-1}\theta_X)"]{} &&&& FZ. \arrow[Rightarrow, from=RRRRR, to=RRRRRR, "Id"]{}\\
            \end{tikzcd}
        \end{center}
        
        Componiendo las 2-celdas en el diagrama anterior, horizontalmente y luego verticalmente, y usando las identidades triangulares ya mencionadas, nos queda la expresi\'on de la ecuaci\'on (\ref{axioma_2_bis}).
        
        \item Dada $\phi : f \Longrightarrow g: X \longrightarrow Y$ una 2-celda en $\mathscr{C}$, veamos que $$\theta_g \circ G(\phi)\eta_X = \eta_Y F(\phi) \circ \eta_f$$.
        
        Tenemos
        \begin{equation} \label{axioma_3}
            \begin{split}
                \theta_g \circ F(\phi) \theta_X & = [(\theta_YGg\beta_Y)\circ (\theta_Y\eta_g^{-1}\theta_X)] \circ [(\alpha_Y^{-1}Fg\theta_X) \circ (F(\phi) \theta_X)]\\
                \\
                & = [(\theta_YGg\beta_Y)\circ (\theta_Y\eta_g^{-1}\theta_X)] \circ [(\theta_Y\eta_YF(\phi)\theta_X)\circ (\alpha_Y^{-1}Ff\theta_X)]\\
            \end{split}
        \end{equation}
        y
        \begin{equation*}
            \begin{split}
                \theta_YG(\phi) \circ \theta_f & = (\theta_YGg\beta_X) \circ (\theta_YG(\phi) \eta_X \theta_X) \circ (\theta_Y \eta_f^{-1}\theta_X) \circ (\alpha_Y^{-1}Ff\theta_X)\\
                \\
                & = (\theta_YGg\beta_X) \circ \theta_Y[(G(\phi) \eta_X) \circ ( \eta_f^{-1})]\theta_X \circ (\alpha_Y^{-1}Ff\theta_X).\\
            \end{split}
        \end{equation*}
        
        Como el axioma se verifica para $\eta$, entonces $$\eta_g^{-1} \circ \eta_YF(\phi) = G(\phi) \eta_X \circ \eta_f^{-1}.$$ Reemplazando esta expresi\'on en (\ref{axioma_3}) obtenemos la igualdad que queriamos probar.
    \end{enumerate}

    \vspace{5mm}
    Lo que queda de la demostraci\'on consiste en exhibir modificaciones inversibles $\mu : \eta \circ \theta \longrightarrow Id_G : G \Longrightarrow G$ y $\lambda: \theta \circ \eta \longrightarrow Id_F : F \Longrightarrow F$.
    
    Definimos $\mu_X=\beta_X : \eta_X\theta_X \Longrightarrow id_{GX}$, para cada $X$. Dada $f: X \longrightarrow Y$, queremos ver que
    $$(Id_G)_g \circ (Gf\mu_X)=\mu_YGf \circ (\eta \circ \theta)_f.$$
    
    Como $(Id_G)_g=g$, $\mu_Y=\beta_Y$, $\mu_X=\beta_X$ y $(\eta \circ \theta)_f=(\eta_Y\theta_f)\circ(eta_f\theta_X),$ entonces la igualdad requerida es $$(Gf\beta_X)=\beta_YGf \circ (\eta_Y\theta_f)\circ(\eta_f\theta_X).$$
    
    Por definici\'on de $\theta_f$ y usando la ecuaci\'on $\eta_Y\alpha_Y^{-1}Ff=\beta_Y^{-1}\eta_YFf$ se tiene
    \begin{equation*}
        \begin{split}
            (\beta_YGf) \circ (\eta_Y\theta_f)\circ(\eta_f\theta_X) & = \eta_Y[(\theta_YGf\beta_X)\circ (\theta_Y\eta_f^{-1}\theta_X) \circ (\alpha_Y^{-1}Ff\theta_X)]\circ (\eta_f\theta_X)\\
            \\
            & = (\beta_YGf)\circ (\eta_Y\theta_YGf\beta_X)\circ(\eta_Y\theta_Y\eta_f^{-1}\theta_X) \\
            & \hspace{50mm} \circ (\eta_Y\alpha_Y^{-1}Ff\theta_X)\circ(\eta_f\theta_X)\\
            \\
            & = (\beta_YGf)\circ (\eta_Y\theta_YGf\beta_X)\circ (\beta_Y^{-1}Gf\eta_X\theta_X)\\
            \\
            & = (\beta_YGf)\circ (\beta_Y^{-1}Gf\beta_X)\\
            \\
            & = Gf\beta_X. \\
        \end{split}
    \end{equation*}

    Por lo tanto, $\mu$ es una 2-celda inversible en $Hom_p(F, F)$.
    
    \bigskip
    Definiendo $\lambda_X=\alpha_X$ para cada $X$, podemos demostrar con una cuenta similar que $\lambda$ es una modificaci\'on inversible.
    
\end{proof}


\begin{observation}
    La proposici\'on anterior no vale para transformaciones 2-naturales. Es decir, una transformaci\'on 2-natural que es una equivalencia punto a punto no necesariamente es una equivalencia en $Hom_s(\mathscr{C}, \mathscr{D})$.
\end{observation}

\bigskip
\begin{definition}
    Sean $\mathscr{C}$ una 2-categor\'ia y $\Sigma$ una subclase de morfismos. La \emph{2-localizaci\'on} de $\mathscr{C}$ con respecto a $\Sigma$ es una 2-categor\'ia $\mathscr{C}[\Sigma^{-1}]$ junto con un 2-funtor $q: \mathscr{C} \longrightarrow \mathscr{C}[\Sigma^{-1}]$ tales que
    \begin{enumerate}
        \item $q(s)$ es una equivalencia para todo $s \in \Sigma$;
        \item para toda 2-categor\'ia $\mathscr{D}$, $q$ induce una pseudoequivalencia de 2-categor\'ias dada por la precomposici\'on
        \begin{center}
            $q^*: Hom_p(\mathscr{C}[\Sigma^{-1}], \mathscr{D}) \longrightarrow Hom_{p_+}(\mathscr{C}, \mathcal{D})$,
        \end{center}
        donde $Hom_{p_+}(\mathscr{C}, \mathcal{D})$ consiste de los 2-funtores que mandan los elementos de $\Sigma$ en equivalencias.
    \end{enumerate}
\end{definition}

\begin{observation}
    La 2-categor\'ia $\mathscr{C}[\Sigma^{-1}]$ queda caracterizada salvo pseudoequivalencias.
\end{observation}

\begin{observation}
    La definici\'on de 2-localizaci\'on que hemos dado es una adaptaci\'on al contexto de 2-categor\'ias de la definici\'on de localizaci\'on de bicategorias (ver \cite{Pro}).
    
    Decimos que el 2-funtor $q$ es la 2-localizaci\'on en el sentido \emph{estricto} si el 2-funtor inducido $q^*$ es una pseudoequivalencia de 2-categor\'ias en el sentido estricto, de acuerdo a la definici\'on \ref{equiv_2cat}.
\end{observation}

\newpage
\section{Categor\'ias de Modelos}

Las categor\'ias de modelos fueron introducidas por Quillen (\cite{Qui}) como un escenario en el cual es posible desarrollar una teor\'ia de homotop\'ia abstrayendo ciertas propiedades que encontramos en el contexto particular de espacios topol\'ogicos y que son comunes a muchos ejemplos conocidos. Una categor\'ia de modelos consiste de tres clases distinguidas de flechas $\mathcal{F}$ (fibraciones), $co\mathcal{F}$ (cofibraciones) y $\mathcal{W}$ (equivalencias d\'ebiles) cumpliendo ciertos axiomas que codifican sus propiedades. Si bien en la teor\'ia de homotop\'ia asociada a una categor\'ia de modelos lo que se quiere es invertir formalmente los elementos de la clase $\mathcal{W}$, tanto fibraciones como cofibraciones son esenciales a la hora de hacer posible una teor\'ia que va m\'as all\'a de una simple localizaci\'on.

La definici\'on que usaremos nosotros es m\'as fuerte que la definicion original y es la que Quillen introduce con el nombre de \emph{categor\'ia de modelos cerrada}, en la que cualesquiera dos de las tres clases distinguidas de morfismos determinan la tercera.

En general, no es f\'acil demostrar que una categor\'ia admite una estructura de modelos, por lo que s\'olo mencionaremos algunos de los ejemplos m\'as usuales, haciendo especial \'enfasis en la categor\'ia $\mathcal{T}op$ de espacios topol\'ogicos. 

\subsection{Axiomas y definiciones}

\begin{definition} \label{lifting}
    Sea $\mathscr{X}$ una categoría. Decimos que un morfismo $f$ en $\mathscr{X}$ tiene la \emph{propiedad de levantamiento a izquierda} con respecto a un morfismo $g$ si todo problema de la forma
    \begin{center}
        $\xymatrix{
        \cdot \ar[r] \ar[d]_{f} & \cdot \ar[d]^{g}\\
        \cdot \ar[r] \ar@{-->}[ur]^{h} & \cdot\\
        }$
    \end{center}
    
    tiene una solución $h$, no necesariamente única, que hace conmutar ambos triángulos. Equivalentemente, decimos que $g$ tiene la \emph{propiedad de levantamiento a derecha} con respecto a $f$.
\end{definition}

\begin{definition} \label{retract}
    Dadas $f:X \longrightarrow Y$, $g: X' \longrightarrow Y'$ en una categoría $\mathscr{X}$, entonces \emph{f es retracto de g} si existe un diagrama conmutativo:
    \begin{center}
        $\xymatrix{
        X \ar[r] \ar@/^1pc/[rr]^{id_X} \ar[d]_{f} & X' \ar[r] \ar[d]_{g} & X \ar[d]^{f} \\
        Y \ar[r] \ar@/_1pc/[rr]_{id_Y} & Y' \ar[r] & Y \\
        }$
    \end{center}
\end{definition}

\bigskip
La siguiente es una definici\'on tratada por Goerss y Jardine en \cite{JG}, e introducida por Quillen (\cite{Qui}) de manera equivalente.
\begin{definition} \label{model_cat}
    Una \emph{categor\'ia de modelos} es una categor\'ia $\mathscr{C}$ provista de tres clases de morfismos $\mathcal{F}$, $co\mathcal{F}$ y $\mathcal{W}$, que llamamos, respectivamente, Fibraciones, Cofibraciones y Equivalencias Débiles, satisfaciendo los siguientes axiomas.
    \renewcommand{\labelenumi}{M\arabic{enumi}.}
    \renewcommand{\labelenumii}{\roman{enumii}.}
    \begin{enumerate}
        \item $\mathscr{C}$ tiene l\'imites finitos y col\'imites finitos.
        \item Si una cofibración es además una equivalencia débil, entonces tiene la propiedad de levantamiento a izquierda con respecto a  cualquier fibración.
        
        Si una fibración es además una equivalencia débil, entonces tiene la propiedad de levantamiento a derecha con respecto a  cualquier cofibración.
        \item Si $f$ es retracto de $g$ y $g$ es una fibración, una cofibración o una equivalencia débil, entonces $f$ también lo es. Además, las tres clases son cerradas por composición y contienen todas las identidades.
        \item Todo morfismo $f$ en $\mathscr{C}$ puede ser factorizado de dos maneras:
            \begin{enumerate}
                \item $f=pi$, donde $p$ es una fibración e $i$ es una cofibración y también es una equivalencia débil;
                \item $f=pi$, donde $p$ es una fibración y tambi\'en es una equivalencia débil e $i$ es una cofibración.
            \end{enumerate}
        \item Sean \begin{tikzcd}
                        X \arrow[r, "f"] & Y \arrow[r, "g"] & Z
                    \end{tikzcd}
         en $\mathscr{C}$. Si dos de los tres morfismos $f$, $g$ y $gf$ son equivalencias débiles, entonces los tres lo son.
    \end{enumerate}
    
\end{definition}

\bigskip
\begin{duality}
Un objeto $X \in \mathscr{C}$ pensado en la categor\'ia dual formal se denota $\widebar{X} \in \mathscr{C}^{op}$. Los morfismos no cambian la notaci\'on y se tiene que una flecha $\xymatrix{\widebar{X} \ar[r]^{f} & \widebar{Y}}$ en $\mathscr{C}^{op}$ es una flecha $\xymatrix{Y \ar[r]^{f} & X}$ en $\mathscr{C}.$

Los axiomas de la definici\'on \ref{model_cat} son auto-duales. Dada una categoría de modelos $\mathscr{C}$, la categoría opuesta también admite una estructura de modelos, donde un morfismo $f: \widebar{Y} \longrightarrow \widebar{X}$ en $\mathscr{C}^{op}$ es
\begin{enumerate}
    \item una \emph{equivalencia débil} si $f: X \longrightarrow Y$ lo es en $\mathscr{C}$.
    
    \item una \emph{cofibración} si $f$ es una fibración en $\mathscr{C}$.
    
    \item una \emph{fibración} si $f$ es una cofibración en $\mathscr{C}$.
\end{enumerate}
\end{duality}

\bigskip

Observamos que, por el axioma M1, en una categoría de modelos siempre disponemos de un objeto inicial y de un objeto terminal, denotados $0$ y $1$, respectivamente.

\begin{definition}
   Un objeto $X$ en una categoría de modelos $\mathscr{C}$ es \emph{fibrante} si $X \longrightarrow 1$ es una fibración, y es \emph{cofibrante} si $0 \longrightarrow X$ es una cofibración. 
\end{definition}

\begin{definition}
    Una (co)fibración es \emph{trivial} si además es una equivalencia débil.
    Usaremos la siguiente notación:

    \renewcommand{\labelenumi}{\roman{enumi}.}
    \begin{enumerate}
        \item 
            $\xymatrix@1{ \cdot \ar[r]|\circ & \cdot }$ (equivalencias débiles)
        \item 
            $\xymatrix@1{ \cdot \ar@{->>}[r] & \cdot }$ (fibraciones) y $\xymatrix@1{ \cdot \ar@{->>}[r]|\circ & \cdot }$ (fibraciones triviales)
        \item 
            $\xymatrix@1{ \cdot \hspace{2mm} \ar@{>->}[r] & \cdot }$ (cofibraciones) y $\xymatrix@1{ \cdot \hspace{2mm} \ar@{>->}[r]|\circ & \cdot }$ (cofibraciones triviales)
    \end{enumerate}
\end{definition}

\subsection{Ejemplos}
\subsubsection{Espacios Topol\'ogicos}

\begin{definition}
    Una función continua $f: X \longrightarrow Y$ entre espacios topológicos es una \emph{equivalencia homotópica débil} si induce un isomorfismo de grupos de homotop\'ia $f_*: \pi_n(X, x_0) \longrightarrow \pi_n (Y, f(x_0))$ para todo $n \geq 0$, para todo $x_0 \in X$.
\end{definition}

\begin{definition}
    Una función continua $p: X \longrightarrow Y$ es una \emph{fibración de Serre} si tiene la propiedad de levantamiento a derecha con respecto a las inclusiones $\xymatrix{D^n\hspace{1mm} \ar@{>->}[r] & D^n \times [0, 1]}, n\geq 0$.
    \begin{center}
        $\xymatrix{
            D^n \ar[r] \ar@{ >->}[d] & X \ar[d]^{p} \\
            D^n\times [0, 1] \ar@{-->}[ur] \ar[r] & Y
        }$
    \end{center}
\end{definition}

\vspace{2mm}
La categoría $\mathcal{T}op$ de espacios topológicos admite una estructura de categor\'ia de modelos, donde las equivalencias débiles $\mathcal{W}$ son las equivalencias homotópicas débiles, las fibraciones $\mathcal{F}$ son las fibraciones de Serre y las cofibraciones $\mathcal{C}o\mathcal{F}$ son aquellas funciones continuas con la propiedad de levantamiento a izquierda respecto de toda función en $\mathcal{F} \cap \mathcal{W}$.

Algunas caracterizaciones de las clases $\mathcal{F}$ y $\mathcal{C}o\mathcal{F}$ resultan útiles a la hora de demostrar los axiomas:
\begin{enumerate}
    \item las fibraciones triviales son aquellas funciones continuas que tienen la propiedad de levantamiento a derecha respecto de las inclusiones $\xymatrix{S^{n-1} \hspace{1mm} \ar@{>->}[r] & D^n}$, $n \geq 0$;
    \item una funci\'on continua es una cofibraci\'on trivial si y s\'olo si es una cofibraci\'on y un retracto por deformaci\'on fuerte.
\end{enumerate}

Con estas caracterizaciones se puede ver que cualquier espacio topol\'ogico es un objeto fibrante y la clase de objetos cofibrantes incluye a la familia de CW-complejos.

La categoría $\mathcal{T}op$ tiene todos los límites y colímites pequeños. Además, por la definición de equivalencia débil, es evidente que el axioma M5 se verifica. En cuanto a M3, por funtorialidad de $\pi_n$ y la conmutatividad del diagrama de la definición \ref{retract} puede 
verse que la clase $\mathcal{W}$ es cerrada por retractos. Tanto para $\mathcal{F}$ como para $co\mathcal{F}$ el argumento es muy similar y est\'a basado en el siguiente resultado categ\'orico (junto con su versi\'on dual):

\emph{Si $i$ es retracto de $j$ y $j$ tiene la propiedad de levantamiento a derecha con respecto a $f$, entonces $i$ tiene la propiedad de levantamiento a derecha con respecto a $f$.}
\begin{center}
$\xymatrix{
\cdot \ar[r] \ar@/^1pc/[rr]^{id} \ar[d]_(.4){i} & \cdot \ar[r] \ar[d]_(.4){j} & \cdot \ar[r] \ar[d]_(.4){i} & \cdot \ar[d]^(.4){f} \\
\cdot \ar[r] \ar@/_1pc/[rr]_{id} &\cdot \ar@{-->}[urr] \ar[r] & \cdot \ar@{-->}[ur] \ar[r] & \cdot \\
}$
\end{center}

En relación al axioma de levantamiento M2, sólo hay que ver la segunda parte, ya que la primera es inmediata de la definición de cofibración. Dada una función $f: X \longrightarrow Y$ en $co\mathcal{F} \cap \mathcal{W}$, consideramos una factorización $f=pi$ donde $p$ es una fibración e $i$ es una cofibración trivial y, luego, las tres funciones están en $\mathcal{W}$. Como $f$ es cofibración y $p$ una fibración trivial, existe un levantamiento $d$ en el diagrama conmutativo
$\xymatrix{ \cdot \ar[r]^{i} \ar[d]_{f} & \cdot \ar[d]^{p} \\
            \cdot \ar@{-->}[ur]^{d} \ar@{=}[r] & \cdot}$
y, por lo tanto, $f$ es retracto de $i$. Dado que la clase de morfismos con la propiedad de levantamiento es cerrada por retractos, $f$ tiene dicha propiedad.

Notamos que, partiendo de las caracterizaciones de fibraciones trivales y cofibraciones triviales que mencionamos antes, estas demostraciones de los axiomas s\'olo usan argumentos puramente categ\'oricos.

Resta la mayor dificultad, que está en probar la existencia de las factorizaciones de M4, y para ello se emplean ciertas propiedades de los espacios topol\'ogicos y las funciones continuas. La demostración de este axioma suele basarse en un argumento introducido por Quillen en \cite{Qui} con el nombre \emph{"small object argument"}, que es una manera de producir factorizaciones cuando las fibraciones están caracterizadas por tener la propiedad de levantamiento a derecha respecto de cierto conjunto de morfismos.

Podemos encontrar estas demostraciones con todo detalle en \cite{Dwy} y \cite{Hov}.

\subsubsection{Complejos de cadenas de m\'odulos sobre un anillo}

    Sean $A$ un anillo con unidad y $Mod_A$ la categoría de $A$-módulos a izquierda.
    
    Se define la categoría $Ch_A$ de complejos de cadenas (graduados positivamente) sobre $Mod_A$ como aquella en la que cada objeto $M_{\bullet}$ consiste de una familia $\{C_k\}_{k\geq 0}$ de $A$-módulos junto con morfismos de borde $\partial : M_k \longrightarrow M_{k-1}$ para todo $k\geq 1$ tales que $\partial^2=0$; y una flecha $f: M_{\bullet} \longrightarrow N_{\bullet}$ en $Ch_A$ es una colección de morfismos de $A$-módulos $f_k : M_k \longrightarrow N_k$ tales que $f_{k-1}\partial = \partial f_k$.
    
    La categoría $Ch_A$ resulta una categoría de modelos, donde un morfismo $f:M_{\bullet} \longrightarrow N_{\bullet}$ es
    \begin{enumerate}
        \item una \emph{equivalencia débil} si $f$ induce isomorfismos $f_k : H_k(M) \longrightarrow H_k (N)$ para todo $k\geq 0$, siendo $H_k(\cdot)$ el $k$- ésimo grupo de homología;
        \item una \emph{cofibración} si $f_k$ es un monomorfismo y adem\'as su co-núcleo es un $A$-módulo proyectivo, para cada $k\geq 0$;
        \item una \emph{fibración} si $f_k$ es un epimorfismo, $k\geq 0$.
    \end{enumerate}
    
    Como en el caso de espacios topol\'ogicos con los funtores $\pi_n$, los axiomas M3 y M5 son una consecuencia de la funtorialidad de la homolog\'ia $H_n: Ch_A \longrightarrow Mod_A$. La categor\'ia $Ch_A$ tiene todos los l\'imites y col\'imites pequeños, que se calculan gradualmente. En cuanto a los axiomas de factorizaci\'on y levantamiento de morfismos, estos se obtienen de algunas construcciones que requieren de las siguientes caracterizaciones:
    \begin{enumerate}
        \item Sea $f: M_{\bullet} \longrightarrow N_{\bullet}$ un morfismo en $Ch_A$.
        
        Consideramos el pullback
        $\xymatrix{
        Z_{n-1}(M)\underset{Z_{n-1}(N)}{\times}N_n \ar@{-->}[r] \ar@{-->}[d] & N_n, \ar[d] \\
         Z_{n-1}(M) \ar[r] & Z_{n-1}(N)\\
        }$
        donde $Z_n(\cdot)$ denota el $n$-\'esimo ciclo del complejo.
      
        Son equivalentes:
        \renewcommand{\labelenumi}{\roman{enumi}.}
        \begin{enumerate}
            \item f es una fibraci\'on trivial;
            \item el morfismo inducido $M_n \longrightarrow Z_{n-1}(M)\underset{Z_{n-1}(N)}{\times}N_n$ es un epimorfismo para todo $n \geq 0$.
        \end{enumerate}
        
        \item Para cada $n>0$, denotamos $D_{\bullet}^n$ al complejo tal que $D_n^n=D_{n-1}^n=A$, $D_k^n=0$ para todo $k \neq n, n-1$ y $\partial : D_n^n \longrightarrow D_{n-1}^n$ es la identidad.
        
        Un morfismo $f: M_{\bullet} \longrightarrow N_{\bullet}$ es una fibraci\'on si y s\'olo si tiene la propiedad de levantamiento a derecha respecto de $0 \longrightarrow D_{\bullet}^n$ para todo $n > 0$:
        \begin{center}
            $\xymatrix{
                0 \ar[r] \ar[d] & M_k \ar[d]^{f_k} \\
                D^n_k \ar[r] \ar@{-->}[ur] & M_k.\\
            }$
        \end{center}
        
    \end{enumerate}
    
    Las demostraciones de estas caracterizaciones y de los axiomas se encuentran en \cite{Riehl}. En otra versi\'on \cite{Dwy}, contamos con una demostraci\'on del axioma de factorizaci\'on basada en el argumento del objeto pequeño.

\subsubsection{Conjuntos simpliciales}

Sea $\mathcal{SS}et$ la categor\'ia de conjuntos simpliciales $X : \Delta^{op} \longrightarrow Set$ y transformaciones naturales.
Denotamos $d_i: X_n \longrightarrow X_{n-1}$ (para cada $n$) a la \emph{i-\'esima cara} de $X$ y $s_i : X_n \longrightarrow X_{n+1}$ a la \emph{i-\'esima degeneraci\'on} de $X$.
Un $n$-simplex de $X$ es \emph{no degenerado} si no est\'a en la imagen de ning\'un $s_i$.

Recordemos algunos ejemplos de conjuntos simpliciales:

\newcommand{\colim}[2]{
    \underset{#1}{\underrightarrow{colim}}\;{#2}
}

\begin{enumerate}
    \item Para $n \in \mathbb{N}$, el funtor representable $\Delta^n := Hom_{\Delta}(-, [n]): \Delta^{op}\longrightarrow Set$ es el \emph{n- s\'implex est\'andar}.
    Por el Lema de Yoneda (\cite{Mac} Ch.III $\S$2), existe una biyecci\'on $Hom_{SSet}(\Delta^n, X) \simeq X_n$, natural en $X$.
    
    Si consideramos la categor\'ia $\Delta\! \downarrow\!  {X} $ (la categor\'ia de ``elementos'' de $X$), donde los objetos son las flechas $\Delta^n \longrightarrow X$ $(n\geq 0)$, entonces podemos escribir a $X$ como un col\'imite de funtores representables \footnote{Este es un caso particular de la caracterizaci\'on de todo funtor contravariante a valores en $Set$ (``prehaz'') como col\'imite de funtores representables (ver \cite{Mac} Ch.III $\S$7).}:
     \\
     $X \simeq \colim{\Delta^n \longrightarrow X}{\Delta^n}$
    \item El conjunto simplicial $\partial \Delta^n$ es el subconjunto de $\Delta^n$ formado por todos sus s\'implices menos el \'unico $n$-s\'implex no degenerado.
    Este es el \emph{borde} de $\Delta^n$.
    
    \item El funtor $\Lambda^n_k : \Delta^{op} \longrightarrow Set$ es el conjunto simplicial formado por la uni\'on de todas las caras de $\Delta^n$ menos la $k$-\'esima.
\end{enumerate}

\vspace{2mm}
Uno de los conceptos importantes a considerar en esta categor\'ia es el de la \emph{realizaci\'on geom\'etrica}, que nos permite interpretar a los conjuntos simpliciales (que se obtienen ``pegando'' s\'implices a trav\'es de sus bordes) como CW-complejos (que se obtienen pegando celdas a trav\'es de sus bordes), y viceversa.
Se define la realizaci\'on geom\'etrica de $\Delta^n$ como la c\'apsula convexa de los vectores de la base can\'onica de $\mathbb{R}^{n+1}$ dotada de la topolog\'ia subespacio (es decir, es el $n$-s\'implex topol\'ogico est\'andar). Esta definici\'on se extiende a cualquier conjunto simplicial $X$ tomando el col\'imite
$|X|:= \colim{\Delta^n \longrightarrow X}{|\Delta^n|}$. De hecho, $|X|$ es un CW-complejo con una $n$-celda por cada $n$-s\'implex no degenerado de $X$ (ver \cite{Mil}).
\\

La realizaci\'on geom\'etrica $|\cdot|: \mathcal{SS}et \longrightarrow \mathcal{T}op$ es funtorial y tiene un adjunto a izquierda, que es el \emph{funtor singular} $S : X \in \mathcal{T}op \longmapsto S(X) \in \mathcal{SS}et$ definido como $S(X)_n:= Hom (\Delta^n, X)$, $n\geq 0$.

\vspace{5mm}
La categor\'ia $\mathcal{SS}et$ es una categor\'ia de modelos, donde $f: X \longrightarrow Y$ es 

\begin{enumerate}
    \item una \emph{equivalencia d\'ebil} si su realizaci\'on geom\'etrica $|f|: |X| \longrightarrow |Y|$ es una equivalencia d\'ebil de espacios topol\'ogicos;
    \item una \emph{fibraci\'on} si tiene la propiedad de levantamiento a derecha respecto de las inclusiones $\Lambda^n_k \longrightarrow \Delta^n$, para todo $n\geq1$ y $0\leq k \leq n$ (fibraci\'on de Kan);
    \item una \emph{cofibraci\'on} si es un monomorfismo (es decir, $f_n: X_n \longrightarrow Y_n$ es una funci\'on inyectiva para todo $n$).
\end{enumerate}

Los objetos fibrantes en $\mathcal{SS}et$ son los complejos de Kan y cualquier objeto es cofibrante.

\begin{observation}
    En $\mathcal{SS}et$, así como en $\mathcal{T}op$, las fibraciones triviales son aquellos morfismos con la propiedad de levantamiento a derecha respecto a las inclusiones $\partial \Delta^n \hookrightarrow \Delta ^n $, $n \geq 1$,
    \begin{center}
        $\xymatrix{
        \partial\Delta^n \ar[r] \ar@{^{(}->}[d] & X \ar[d] \\
        \Delta^n \ar@{-->}[ur] \ar[r] & Y. \\
        }$
    \end{center}
    
    Como consecuencia de la adjunci\'on $S \dashv |\cdot|$, una flecha $p: E\longrightarrow B$ en $\mathcal{T}op$ es una fibraci\'on trivial si y s\'olo si $S(p)$ lo es en $\mathcal{SS}et$. En efecto, existe una correspondencia entre los diagramas que son de la forma
    \begin{center}
    $\xymatrix{
        |\partial\Delta^n| \ar[r] \ar@{^{(}->}[d] & E \ar[d]^{p} \\
        |\Delta^n| \ar@{-->}[ur] \ar[r] & B \\
    }$
    \hspace{4mm} y \hspace{4mm}
    $\xymatrix{
        \partial\Delta^n \ar[r] \ar@{^{(}->}[d] & S(E) \ar[d]^{S(p)} \\
        \Delta^n \ar@{-->}[ur] \ar[r] & S(B) \\
    }$
    \end{center}
    en $\mathcal{T}op$ y $\mathcal{SS}et$, respectivamente.
\end{observation}

\bigskip
\subsubsection{Categor\'ias pequeñas (Estructura de modelos de Thomason) }

La categor\'ia $\mathbf{Cat}$ de categor\'ias pequeñas puede verse como una categor\'ia de modelos de dos formas diferentes. Una de ellas es la estructura can\'onica, donde las equivalencias d\'ebiles son las equivalencias de categor\'ias. La otra es la estructura de modelos de Thomason, en la que un morfismo $F$ en $\mathbf{Cat}$ es una equivalencia d\'ebil si y s\'olo si tomando el nervio obtenemos una equivalencia d\'ebil $NF$ en $\mathcal{SS}et$.

Para explicitar un poco m\'as la estructura de Thomason necesitamos algunas definiciones previas:

\vspace{2mm}
Dada una categor\'ia pequeña $\mathcal{C}$, definimos $N(\mathscr{C})$, el \emph{nervio de $\mathscr{C}$}, como el conjunto simplicial cuyos v\'ertices son los objetos de $\mathscr{C}$ y, para $n\geq 1$, los $n$-s\'implices son las $n$-tuplas de morfismos componibles en $\mathscr{C}$. La $i$-\'esima cara $d_i: N(\mathscr{C})_n \longrightarrow N(\mathscr{C})_{n-1}$ es la funci\'on dada por $$(f_1, ..., f_{i-1}, f_i, f_{i+1}, ..., f_n) \mapsto (f_1, ..., f_{i-1}, f_{i+1}, ..., f_n),$$
mientras que la $i$-\'esima degeneraci\'on $s_i: N(\mathscr{C})_n \longrightarrow N(\mathscr{C})_{n+1}$ consiste en insertar la flecha identidad correspondiente en el lugar $i$-\'esimo de cada tupla $$(f_1, ..., f_i, f_{i+1}, ..., f_n) \mapsto (f_1, ..., f_{i}, id, f_{i+1}, ..., f_n).$$ 
Por ejemplo, tomando la categor\'ia $\mathscr{C}$ con objetos $x, y, z$ y dos morfismos componibles $f: x\rightarrow y$, $g:y \rightarrow z$, entonces $N(\mathscr{C})=\Delta^2$.

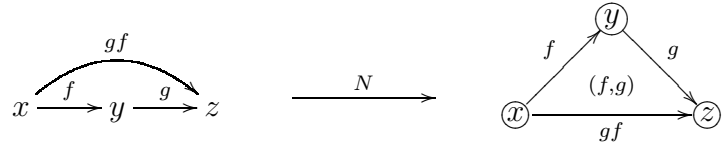
\begin{center}
    $\xymatrix{
    \\
    x \ar@/^1.5pc/[rr]^{gf} \ar[r]^{f} & y \ar[r]^{g} & z}$
    $\xymatrix{
    \\
    \hspace{5mm} \ar[rr]^{N} && \hspace{5mm}}$
    $\xymatrix{
    & *+[o][F-]{y} \ar[dr]^{g} \ar@{}[d]|(.7){(f, g)} \\
    *+[o][F-]{x} \ar[ur]^{f} \ar[rr]_{gf} && *+[o][F-]{z}\\
    }$
\end{center}
M\'as generalmente, pensando al conjunto ordenado $[n]=\{0, 1, ..., n\}$ como una categor\'ia, se tiene $N([n])=\Delta^n$.
\\

Dada una flecha $F : \mathscr{C} \longrightarrow \mathscr{D}$ en $\mathbf{Cat}$, $N(F): N(\mathscr{C}) \longrightarrow N(\mathscr{D})$ es la transformaci\'on natural cuyas componentes $N(F)_n :N(\mathscr{C})_n \longrightarrow N(\mathscr{D})_n$ son $(f_0, ..., f_{n-1}) \longmapsto (Ff_0, ..., Ff_{n-1})$. Obtenemos de esta manera una asignaci\'on funtorial $N: \mathbf{Cat} \longrightarrow \mathcal{SS}et$.

El nervio $N$ tiene un adjunto a izquierda, denotado $c: \mathcal{SS}et \longrightarrow \mathbf{Cat}$. Para cada conjunto simplicial $X$, $c(X)$ es una categor\'ia tomando como objetos a los v\'ertices de $X$, y los morfismos son generados libremente por los $1$-s\'implices de $X$ y cocientando por las relaciones $d_1x=d_0xd_2x$ para todo $2$-s\'implex $x$ (\cite{GZ} Ch.II $\S$4).
\\

La \emph{subdivisi\'on} del $n$-s\'implex est\'andar es un conjunto simplicial $Sd\Delta^n$ que se define como el nervio aplicado al poset de subconjuntos no vac\'ios de $[n]$. As\'i, por ejemplo, $Sd\Delta^2$ est\'a dado por los siguientes v\'ertices, $1$-s\'implices no degenerados y $2$-s\'implices no degenerados:

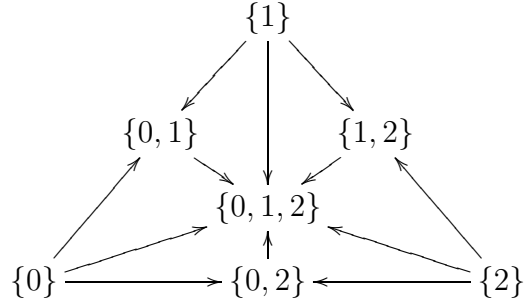
\begin{center}
$\xymatrix@C=.1pt@R=10pt{
            &&&&& \{1\}\ar[ddl] \ar[ddr] \ar[ddd] &&&&&\\
            \\
        &&&& \{0, 1\}\ar[dr] && \{1, 2\} \ar[dl] &&&& \\
                    &&&&& \{0,1,2\} &&&&&\\
\{0\}\ar[uurrrr]\ar[urrrrr]\ar[rrrrr] &&&&& \{0,2\}\ar[u] &&&&& \{2\}\ar[lllll]\ar[ulllll]\ar[uullll]}$
\end{center}

Esta definici\'on se extiende a cualquier conjunto simplicial $X$ tomando el col\'imite $SdX:= \colim{\Delta^n \longrightarrow X}{Sd\Delta^n}$.
La asignaci\'on $Sd:\mathcal{SS}et \longrightarrow \mathcal{SS}et$ tambi\'en es funtorial y tiene un adjunto a derecha $E\text{x}(X)_n:= Hom_{\mathcal{SS}et}(Sd\Delta^n, X)$. Luego, se tiene una adjunci\'on $(Sd)^2 \dashv (E\text{x})^2$, donde $(Sd)^2(X)=Sd(Sd(X))$ y 
$(E\text{x})^2(X)=E\text{x}(E\text{x}(X)).$

Esto induce un par de funtores adjuntos entre la categor\'ia $\mathcal{SS}et$ y la categor\'ia $\mathbf{Cat}$ dados por las composiciones $c(Sd)^2$ y $(E\text{x})^2N$

\begin{center}
    $\xymatrix@C=3pc{
        \mathcal{SS}et \ar@/^1pc/[r]^{(Sd)^2} \ar@{}[r]|{\perp} & \mathcal{SS}et \ar@/^1pc/[r]^{c} \ar@/^1pc/[l]^{(\text{Ex})^2} \ar@{}[r]|{\perp} & \mathbf{Cat}. \ar@/^1pc/[l]^{N}
    }$
\end{center}

\vspace{3mm}
Un morfismo $F$ en $\mathbf{Cat}$ es
\begin{enumerate}
    \item una equivalencia d\'ebil si $(E\text{x})^2NF$ es una equivalencia d\'ebil en $\mathcal{SS}et$,
    \item una fibraci\'on si $(E\text{x})^2NF$ es una fibraci\'on en $\mathcal{SS}et$,
    \item una cofibraci\'on si tiene la propiedad de levantamiento a izquierda respecto de las fibraciones triviales.
\end{enumerate}

La categor\'ia $\mathbf{Cat}$ es una categor\'ia de modelos con $\mathcal{W}$, $\mathcal{F}$ y $co\mathcal{F}$ como acabamos de definir.
Thomason (\cite{Tho}) demuestra que, efectivamente, se verifican los axiomas de categor\'ias de modelos, y toda la dificultad se concentra en M2 (la parte no trivial de M2 en este caso) y en M4.

Con esta estructura se tiene que

\begin{center}
\emph{$F$ es una equivalencia d\'ebil en $\mathbf{Cat}$ si y s\'olo si $NF$ lo es en $\mathcal{SS}et$},
\end{center}
y puede verse que el nervio $N$ induce una equivalencia entre las categor\'ias homot\'opicas de Quillen $\mathbf{Ho}(\mathbf{Cat})$ y $\mathbf{Ho}(\mathcal{SS}et)$.

\vspace{3mm}
Por otro lado, los funtores $\xymatrix@C=2.5pc{ \mathbf{Cat} \ar@<5pt>[r]^{(E\text{x})^2N} \ar@<-5pt>@{<-}[r]_{c(Sd)^2} & \mathcal{SS}et}$ satisfacen las hip\'otesis del teorema de equivalencia entre teor\'ias de homotop\'ia establecido por Quillen (\cite{Qui}, Ch.I $\S 4$.) y, luego, este par de funtores adjuntos tambi\'en induce una equivalencia entre la cl\'asica localizaci\'on de $\mathcal{SS}et$ y la localizaci\'on de $\mathbf{Cat}$ con respecto a las equivalencias d\'ebiles.

\vspace{5mm}

\subsection{Determinaci\'on}
En una categoría $\mathscr{C}$ la estructura de modelos queda determinada por dos de las tres clases distinguidas $\mathcal{F}$, $co \mathcal{F}$, $\mathcal{W}$. Esto será una consecuencia de los siguientes resultados.

\begin{lemma}
    Si $f=hg$ es un morfismo en una categoría $\mathscr{X}$ tal que $f$ tiene la propiedad de levantamiento a izquierda con respecto a $h$, entonces $f$ es retracto de $g$. Dualmente, si $f$ tiene la propiedad de levantamiento a derecha con respecto a $g$, entonces $f$ es retracto de $h$.
\end{lemma}

\begin{proof}
    Si escribimos $f: X \longrightarrow Y$, $g:X \longrightarrow Z$ y $h: Z \longrightarrow Y$, como $f$ tiene la propiedad de levantamiento con respecto a $h$ y $f=hg$, existe un morfismo $l$ que hace conmutar ambos triangulos en el diagrama
    \begin{center}
        $\xymatrix{
            X \ar[r]^{g} \ar[d]_{f} & Z \ar[d]^{h} \\
            Y \ar@{-->}[ur]^{l} \ar@{=}[r] & Y.\\
        }$
    \end{center}
    
    Por lo tanto, tenemos
    \begin{center}
        $\xymatrix{
            X \ar@{=}[r] \ar[d]_{f} & X \ar@{=}[r] \ar[d]_{g} & X \ar[d]^{f} \\
            Y \ar[r]_{l} & Z \ar[r]_{h} & Y,\\
        }$
    \end{center}
    donde $hl=id_{Y}$, y luego $f$ es retracto de $g$.
    
    Por dualidad, $f$ es retracto de $h$ si tiene la propiedad de levantamiento con respecto a $g$.
    
\end{proof}

\begin{proposition} \label{lifting_sii}
    Sea $\mathscr{C}$ una categoría de modelos.
    \renewcommand{\labelenumi}{\roman{enumi}.}
    \begin{enumerate}
        \item Un morfismo en $\mathscr{C}$ es una cofibración (cofibración trivial) si y sólo si tiene la propiedad de levantamiento a izquierda con respecto a toda fibración trivial (fibración).
        \item Un morfismo en $\mathscr{C}$ es una fibración (fibración trivial) si y sólo si tiene la propiedad de levantamiento a derecha con respecto a toda cofibración trivial (cofibración).
    \end{enumerate}
\end{proposition}

\begin{proof}
    Sea $f:X \longrightarrow Y$ en $\mathscr{C}$ con la propiedad de levantamiento a izquierda respecto a la clase $\mathcal{F}$ y sea $f=pi$ una factorización, con $i$ una cofibración trivial y $p$ una fibración. Por el lema previo, $f$ es retracto de $i$. Dado que las tres clases son cerradas por retractos, $f$ es también una cofibración trivial. Por otro lado, el axioma M2 garantiza que toda cofibración trivial tiene dicha propiedad de levantamiento.
    
    Como para el resto de los casos el argumento es muy similiar, podemos concluir la demostración.
    
\end{proof}

\begin{observation}
La proposición anterior nos dice que la clase $\mathcal{F}$ queda determinada por $(co\mathcal{F}, \mathcal{W})$ y la clase $co\mathcal{F}$ está determinada por $(\mathcal{F}, \mathcal{W})$.

Además, un morfismo $f$ es una equivalencia débil si y sólo si admite una factorización $f=pi$, donde $p$ es una fibración trivial e $i$ es una cofibración trivial: la implicación $(\Rightarrow)$ se debe a los axiomas M4 y M5, mientras que la recíproca es evidente ya que la clase $\mathcal{W}$ es cerrada por composición.
\end{observation}

\begin{proposition}
    Si $\mathscr{C}$ es una categoría de modelos, entonces las clase de cofibraciones y de cofibraciones triviales son ambas estables por pushouts. Las fibraciones y las fibraciones triviales son clases estables por pullbacks.
\end{proposition}

\begin{proof}
    Nuevamente, la segunda afirmación del enunciado se obtendrá de la primera por un argumento de dualidad.
    
    Sean $i: X \longrightarrow Y$ una cofibración y $f: X \longrightarrow Z$ cualquier morfismo en $\mathscr{C}$. Consideramos el pushout de $i$ y $f$
    \begin{center}
        $\xymatrix{
            X \ar[r]^{f} \ar@{ >->}[d]_{i} & Z \ar[d]^{j} \\
            Y \ar[r]_{g} & W \\
        }$,
    \end{center}
    y queremos ver que $j \in \mathcal{C}o\mathcal{F}$. Por la proposición anterior, es equivalente ver que $j$ tiene la propiedad de levantamiento a izquierda con respecto a toda fibración trivial. Sea, entonces, $p \in \mathcal{F} \cap \mathcal{W}$, junto con un diagrama conmutativo 
    \begin{center}
        $\xymatrix{
            Z \ar[r]^{f'} \ar[d]_{j} & X' \ar[d]^{p} \\
            W \ar[r]_{g'} & Y' \\
        }$.
    \end{center}
    
    Como $i$ tiene dicha propiedad de levantamiento, existe un morfismo diagonal $d: Y \longrightarrow Y'$ haciendo conmutar un nuevo diagrama:
    \begin{center}
        $\xymatrix{
           X \ar[r]^{f}\ar[d]_{i} & Z \ar[r]^{f'} & X' \ar[d]^{p} \\
            Y\ar[r]_{g} \ar@{-->}[urr]^{d} & W \ar[r]_{g'} & Y' \\
        }$
    \end{center}
    
    De la propiedad universal del pushout se obtiene una flecha $h: W \longrightarrow X'$ tal que $hg=d$ y $hp=f'$. Además, de la unicidad del morfismo que sale de $W$ como consecuencia de la misma propiedad universal se deduce que $ph=g'$.
    
    Para el caso en el que $i$ es una cofibración trivial la demostración es análoga: la diferencia está en usar la correspondiente versión de la proposición \ref{lifting_sii}.
\end{proof}

\begin{remark}
    Si bien $\mathcal{F}$ ($\mathcal{F}\cap \mathcal{W}$) y $co\mathscr{F}$  ($co\mathcal{F}\cap \mathcal{W}$) son clases cerradas por pullbacks y pushouts, respectivamente, \emph{no} es cierto que la clase $\mathcal{W}$ de equivalencias débiles tenga alguna de estas propiedades.
    
    Los siguientes ejemplos, propuestos por Jonathan Barmak, muestran que las equivalencias d\'ebiles no son estables por pullbacks ni por pushouts en la categor\'ia de modelos de espacios topol\'ogicos.
    
    Sea $I=[0, 1].$ Tomamos el pullback del diagrama 
    $\xymatrix{
        & \{*\}, \ar[d] \\
        S^0 \ar@{^{(}->}[r]^i & I \\
    }$ \\
    donde $i$ es la inclusi\'on de $S^0=\{0, 1\}$ en el intervalo, y la funci\'on $\{*\} \rightarrow I$ es una equivalencia homot\'opica (en particular, equivalencia d\'ebil). Si esta \'ultima manda al punto en el $0$ o el $1$, el pullback es
    $$\xymatrix{
       \{*\}\ar[r] \ar[d] & \{*\} \ar[d] \\
        S^0 \ar@{^{(}->}[r]^i & I, \\
    }$$
    y si manda al punto en cualquier otro elemento de $I$ distinto de $0$ y $1$ el pullback es
    $$\xymatrix{
        \emptyset \ar[r] \ar[d] & \{*\} \ar[d] \\
        S^0 \ar@{^{(}->}[r]^i & I. \\
    }$$
    Cualquiera sea el caso, el pullback de $\{*\} \rightarrow I$ no es equivalencia d\'ebil. Esto muestra que la clase $\mathcal{W}$ no es cerrada por pullbacks en $\mathcal{T}op.$ 
    
    Por otro lado, sean $I \rightarrow \{*\}$ la funci\'on al punto y $exp: I \rightarrow D^2$ la funci\'on dada por $t \mapsto e^{2\pi it}.$ Tomando el pushout, obtenemos  
    $$\xymatrix{
        I \ar[r]^{exp} \ar[d] & D^2 \ar[d] \\
        \{*\} \ar[r]^i & \small{\faktor{D^2}{S^1}}. \\
    }$$
    La funci\'on del intervalo al punto es una equivalencia homot\'opica, pero el cociente $D^2 \longrightarrow \small{\faktor{D^2}{S^1}} \simeq S^2 $ no es equivalencia d\'ebil. As\'i, las equivalencias d\'ebiles no son estables por pushouts en $\mathcal{T}op.$
\end{remark}

\newpage
\section{La 2-categor\'ia homot\'opica \texorpdfstring{$\catH$}{}}

Quillen (\cite{Qui}) introduce las nociones de \emph{cilindro} y de \emph{homotop\'ia} desarrollando una teor\'ia asociada a una categor\'ia de modelos $\mathscr{C}$. Demuestra que las homotop\'ias definen una relaci\'on de equivalencia en los morfismos de la subcategor\'ia plena $\mathscr{C}_{fc}$ de objetos fibrantes-cofibrantes y, cocientando por esta relaci\'on de equivalencia, construye una categor\'ia $\pi\mathscr{C}_{fc}$ cuyos objetos son los objetos de $\mathscr{C}_{fc}$ y los morfismos son clases de flechas en $\mathscr{C}_{fc}$. El funtor $\mathscr{C}_{fc} \longrightarrow \pi\mathscr{C}_{fc}$ resulta equivalente a la localizaci\'on que se obtiene inviertiendo los elementos de la clase $\mathcal{W}$ de equivalencias d\'ebiles.

En esta secci\'on estudiaremos una versi\'on an\'aloga a la construcci\'on de la categor\'ia homot\'opica de Quillen. Fijada una \emph{categor\'ia de modelos} $\mathscr{C}$, queremos definir una 2-categor\'ia $\catH$ cuyos objetos y flechas sean como en $\mathscr{C}$ y cuyas 2-celdas est\'en dadas por homotop\'ias entre flechas, con el objetivo de obtener resultados relativos a la 2-localizaci\'on de $\mathscr{C}$ respecto de la clase $\mathcal{W}$. Este ser\'a un caso particular de la versi\'on 2-dimensional presentada por Descotte, Dubuc y Szyld (\cite{e.d.}) en una exposici\'on en la que introducen el concepto de \emph{bicategor\'ia de modelos} y desarrollan, en este contexto, una teor\'ia cuyo resultado principal es el teorema de la 2-localizaci\'on.

Daremos una definici\'on de cilindro m\'as general que la definici\'on de Quillen, y consecuentemente obtendremos una definici\'on de homotop\'ia tambi\'en m\'as general. Con $q$-cilindro y $q$-homotop\'ia nos referiremos a las nociones originales introducidas por Quillen, para distinguirlas de estas nuevas definiciones. Esto nos permitir\'a establecer una apropiada relaci\'on de equivalencia entre homotop\'ias, que depende s\'olo de la clase de equivalencias d\'ebiles $\mathcal{W}$, y adem\'as nos permitir\'a componer homotop\'ias de forma tal que esta composici\'on resulte compatible con dicha relaci\'on, y se verifiquen los axiomas de 2-categor\'ia.

Una vez definida la 2-categor\'ia $\catH$, la inclusi\'on $i: \mathscr{C} \longrightarrow \catH$ no ser\'a precisamente la 2-localizaci\'on de $\mathscr{C}$ ya que, en general, no manda equivalencias d\'ebiles en equivalencias, y es por esta raz\'on que necesitaremos restringir el funtor $i$ a la subcategor\'ia de objetos fibrantes-cofibrantes para quedarnos con una sub-2-categor\'ia $\Ho$ de $\catH$. Usando los axiomas de la definici\'on 3.3, podremos ver que $i: \mathscr{C}_{fc} \longrightarrow \Ho$ es la 2-localizaci\'on de $\mathscr{C}_{fc}$ respecto de $\mathcal{W}$.

Tomando el funtor $\pi_0$ de componentes conexas, obtendremos los resultados de Quillen como una consecuencia de esta construcci\'on.
\subsection{Homotop\'ias de Quillen}
\begin{definition}
    Un \emph{q-cilindro} $C=(W, d_0, d_1, s)$ para un objeto $X$ en $\mathscr{C}$ es una factorización de la codiagonal 
            \begin{center}
                $\xymatrix{
                    X\coprod X \ar@/^2.5pc/[rrrr] ^{\nabla_X} \ar[rr]^{\binom{d_0}{d_1}} && W \ar[rr]^{s} && X,
                }$
            \end{center}
    donde $\binom{d_0}{d_1}$ es una cofibración y $s$ es una equivalencia débil.
    
    Dualmente, un \emph{q-path object} $P=(V, \delta_0, \delta_1, \sigma)$ para un objeto $Y$ es una factorizaci\'on de la diagonal 
            \begin{center}
                $\xymatrix{
                    Y \ar@/^2.5pc/[rrrr] ^{\Delta_Y} \ar[rr]^{\sigma} &&V \ar[rr]^(.4){(\delta_0, \delta_1)}  && Y\times Y,
                }$
            \end{center}
    con $(\delta_0, \delta_1)$ una fibraci\'on y $\sigma$ una equivalencia d\'ebil.
\end{definition}

\begin{observation}
    Dado $X$ en $\mathscr{C}$, diremos que un $q$-cilindro $C$ es \emph{fibrante} si $s:W\longrightarrow X$ es una fibración.
    
    Por el axioma M4, existe al menos un $q$-cilindro fibrante para $X$, que se obtiene factorizando la codiagonal $\nabla_X$ de la siguiente forma
    \begin{center}
        $\xymatrix{
            X\coprod X \ar[rr]^{\nabla_{X}} \ar@{>->}[dr]_{\binom{d_0}{d_1}}&& X, \\
            & W \ar@{->>}[ur]_{s}|\circ}$
    \end{center}
    
\end{observation}

\begin{lemma} \label{lemma2_quillen}
    Si $X$ es cofibrante y $C=(W, d_0, d_1, s)$ es un $q$-cilindro para $X$, entonces $d_0$ y $d_1$ son cofibraciones triviales.
\end{lemma} 

    \begin{proof}
    Dado que $id_X$ y $s$ son equivalencias débiles, y que $sd_0=id_X$ y $sd_1=id_X$, por el axioma M5 obtenemos que $d_0$ y $d_1$ son equivalencias débiles.
    
    Por otro lado, tenemos el pushout
        \begin{center}
            $\xymatrix{
                0  \ar[r] \ar[d] & X \ar[d]^{i_0} \\
                X \ar[r]_{i_1}   &     X\coprod X
            }$
        \end{center}
    Como $0\longrightarrow X$ es una cofibración y la clase $\mathcal{C}o\mathcal{F}$ es cerrada por pushouts, se tiene que $i_0$ e $i_1$ son cofibraciones y, luego, $d_0=\binom{d_0}{d_1}\circ i_0$ y $d_1=\binom{d_0}{d_1}\circ i_1$ lo son también.
    \end{proof}

\begin{definition} \label{fibrante_hpy}
    Sean $f,g: X \longrightarrow Y$ en $\mathscr{C}$. Una \emph{q-homotopía a izquierda} \\ $H:\xymatrix{f \ar@2{~>}[r] & g}$ con $q$-cilindro $C=(W, d_0, d_1, s)$ (para $X$) es un morfismo $h:W\longrightarrow Y$ tal que $hd_0=f$ y $hd_1=g$.
   
    \vspace{2mm}
      \begin{center}
                $\xymatrix{
                    X\coprod X \hspace{1mm} \ar@{>->}[rr]^{\binom{d_0}{d_1}} \ar[dr]_{\nabla_X} && W  \ar[r]^{h} \ar[dl]^{s}|\circ  & Y, \\
                                                      &  X \\ 
                  }$
      
            \end{center}
    
    Diremos que $H$ es \emph{fibrante} si es una homotopía con cilindro fibrante.

    \vspace{2mm}
    An\'alogamente, una \emph{q-homotop\'ia a derecha} $K: \xymatrix{f \ar@2{~>}[r] & g}$ con $q$-path-object $P=(V, \delta_0, \delta_1, \sigma)$ (para $Y$) es un morfismo $k: X \longrightarrow V$ satisfaciendo $\delta_0k=f$ y $\delta_1k=g.$   
    
    \vspace{2mm}
    \begin{center}
            $\xymatrix{
                    X \ar[r]^{k} & V \ar@{->>}[rr]^{(\delta_0, \delta_1)} && Y \times Y\\
                    && Y \ar[ul]^{\sigma}| \circ \ar[ur]_{\Delta_Y}
            }$,
    \end{center}
\end{definition}

En adelante, trabajaremos s\'olo con homotop\'ias a izquierda.

\bigskip
\begin{lemma} \label{comp_v}
    
Sean $f, g, l : X \longrightarrow Y$, y sean $H:\xymatrix{f \ar@2{~>}[r] & g}$, $H':\xymatrix{g \ar@2{~>}[r] & l}$ dos q-homotopías con cilindros $C$ y $C'$, respectivamente. Si $X$ es cofibrante, entonces existe una q-homotopía $H'':\xymatrix{f \ar@2{~>}[r] & l}$ con cilindro $C''$, donde $C''$ se obtiene del pushout de $d_1$ y $d'_0$.
\end{lemma}

\begin{proof}

Como $X$ es cofibrante, $d_{0}, d_{1}, d'_{0}$ y $d'_{1}$ son cofibraciones triviales.\\
El pushout de $d_{1}$ y $d'_{0}$ nos permite definir un cilindro $C''$ como podemos ver en el siguiente diagrama

\vspace{3mm}

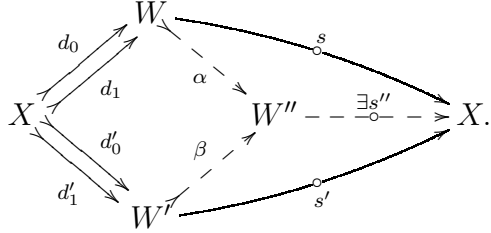
\begin{center}
    $\xymatrix{
            & W \ar@{>-->}[dr]_{\alpha}
                 \ar@/^/[drrr]^{s}|\circ\\
        X \hspace{2mm} \ar@{>->}@<3pt>[ur]^{d_0}
          \ar@{>->}@<-3pt>[ur]_{d_1}
          \ar@{>->}@<3pt>[dr]^{d'_0}
          \ar@{>->}@<-3pt>[dr]_{d'_1}  && W'' \ar@{-->}[rr]^{\exists s''}|\circ && X. \\
            & W' \ar@{>-->}[ur]^{\beta}
                 \ar@/_/[urrr]_{s'}|\circ \\}$

\end{center}

\vspace{3mm}

Dado que $sd_0 = id_X = s'd'_1$, por la propiedad universal de $W''$, exite $s''$ que factoriza tanto a $s$ como a $s'$. Además, como $d_1$ y $d'_0$ son cofibraciones triviales, por M3 $\alpha$ y $\beta$ también lo son. Luego, del axioma M5 y del hecho de que $s$ y $s'$ son equivalencias débiles, se deduce que $s''$ es una equivalencia débil.
Si escribimos $d_0''=\alpha d_0$ y $d_1''=d'_1\beta$, entonces se tiene que $\binom{d_0''}{d_1''}$ es una cofibraci\'on:

Consideramos los pushouts
\begin{center}
$\xymatrix{
    X\coprod X \ar[rr]^{\binom{d_0'}{d_1'}} \ar[dd]_{d_1+id_X} && W' \ar[dd]^{\beta} \\
    \\
    W \coprod X \ar[rr] _{\binom{\alpha}{d_1}} && W'' \\
}$
\hspace{3mm} y  \hspace{3mm}
$\xymatrix{
    X \ar[rr]^{d_0} \ar[dd]_{i_0} && W \ar[dd]^{i_0} \\
    \\
    X \coprod X \ar[rr] _{d_0 + id_X} && W\coprod X, \\
}$
\end{center}
y se tiene que $\binom{\alpha}{d_1}$ y $d_0+id_X$ son ambas cofibraciones, porque $\binom{d'_0}{d'_1}$ y $d_0$ lo son. Luego,
$$ \xymatrix{
X \coprod X \ar@/^2pc/[rrrr]^{\binom{d''_0}{d''_1}} \ar[rr]^{d_0 + id_X} && W \coprod X \ar[rr]^{\binom{\alpha}{d_1}} && W''.
} $$
es tambi\'en una cofibraci\'on.

As\'i, $C'' = (W'', d_0'', d_1'', s'')$ es, efectivamente, un $q$-cilindro para $X$.

\vspace{2mm}
Nuevamente, de la propiedad universal del pushout se tiene una $q$-homotop\'ia $H''$ de $f$ a $l$ con cilindro $C''$. En efecto, ampliamos el diagrama
anterior de la siguiente forma

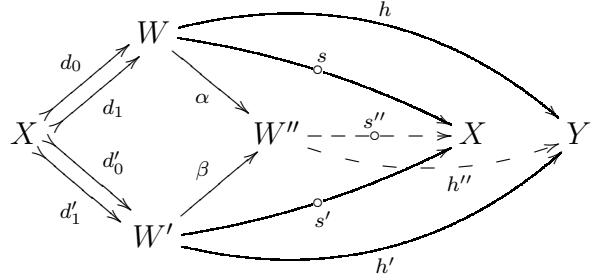
\begin{center}
    $\xymatrix{
            & W \ar[dr]_{\alpha}
                 \ar@/^/[drrr]^{s}|\circ \ar@/^2pc/[rrrrd]^{h}\\
        X \hspace{2mm} \ar@{>->}@<3pt>[ur]^{d_0}
          \ar@{>->}@<-3pt>[ur]_{d_1}
          \ar@{>->}@<3pt>[dr]^{d'_0}
          \ar@{>->}@<-3pt>[dr]_{d'_1}  && W'' \ar@{-->}[rr]^{s''}|\circ \ar@{-->}@/_1pc/[rrr]_(.6){h''} && X & Y\\
            & W' \ar[ur]^{\beta}
                 \ar@/_/[urrr]_{s'}|\circ \ar@/_2pc/[rrrru]_{h'}\\}$
\end{center}

\vspace{3mm}
donde la existencia del morfismo $h''$ se debe a que $hd_1=g=h'd'_0$, y obtenemos, por lo tanto, $H'':\xymatrix{f \ar@2{~>}[r] & l}$ dada por $H''= (C'', h'').$
\end{proof}

\vspace{3mm}
\begin{observation}
Para construir la categoría homotópica de una categoría de modelos $\mathscr{C}$, Quillen demuestra que las homotopías definen una relación de equivalencia entre los morfismos de la subcategoría plena de objetos cofibrantes, y el lema anterior se corresponde con la transitividad de esta relación entre flechas.
En nuestro caso, dadas $H:\xymatrix{f \ar@2{~>}[r] & g}$ y $H':\xymatrix{g \ar@2{~>}[r] & l}$, podemos pensar a la homotopía $H''$ de este resultado como una manera de definir la composición vertical $H\circ H'$.
Sin embargo, si bien las homotopías se componen cuando los objetos son cofibrantes, esta composición que acabamos de exhibir no determina ni una 2-categor\'ia ni una bicategor\'ia. Es necesario, por lo tanto, definir una adecuada relación de equivalencia entre homotopías, de manera que si tomamos como 2-celdas a las clases de equivalencia se verifiquen, entonces, los axiomas de 2-categorías. La categor\'ia as\'i obtenida tendr\'a todas sus 2-celdas inversibles (ver lema \ref{2_cells_inv}).
\end{observation}

\vspace{3mm}
\begin{comentario}
 En $\mathcal{T}op$, el pushout que permite definir la composici\'on vertical anterior es precisamente el cilindro que se obtiene pegando los primeros dos, definiendo as\'i una nueva homotop\'ia como consecuencia del lema del pegado:
\\

   $
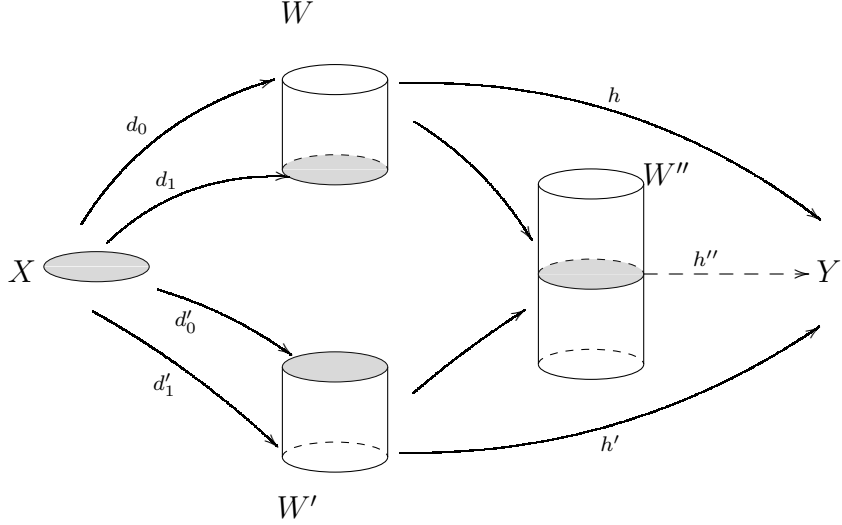
\xymatrix@C=1.5pc@R=1.7pc{
&& W\\
&&\\
&& \save []+<0.5cm, 0.7cm>*{\begin{tikzpicture}
        \draw (0,0) ellipse (0.7 and 0.2);
        \draw (-0.7,0) -- (-0.7,-1.2);
        \draw (-0.7,-1.2) arc (180:360:0.7 and 0.2);
        \draw [dashed] (-0.7,-1.2) arc (180:360:0.7 and -0.2);
        \draw (0.7,-1.2) -- (0.7,0);  
        \fill [gray,opacity=0.3] (-0.7,-1.2) arc (180:360:0.7 and 0.2) -- (-0.7,-1.2) arc (180:360:0.7 and -0.2);
    \end{tikzpicture}} \ar@<15pt>@/^0.6pc/[drrr] \ar@<15pt>@/^1.7pc/[drrrrrrr]^h \restore &&&&& \hspace{1mm} W'' \\
X\hspace{1mm} \begin{tikzpicture}
        \draw (0,0) ellipse (0.7 and 0.2);
        \fill [gray,opacity=0.3] (-0.7,0) arc (180:360:0.7 and 0.2) -- (-0.7,0) arc (180:360:0.7 and -0.2);
  \end{tikzpicture} \ar@/^1pc/[urr]^{d_1}
                  \ar@<9pt>@/^1pc/[uurr]^{d_0}
                  \ar@/^0.5pc/[drr]_{d'_0}
                  \ar@<-10pt>@/^0.3pc/[ddrr]_{d'_1} &&&&&
 \save []+<1cm, 0cm>*{\hspace{2mm} \begin{tikzpicture}
        \draw (0,0) ellipse (0.7 and 0.2);
        \draw (-0.7,0) -- (-0.7,-1.2);
        \draw (-0.7,-1.2) arc (180:360:0.7 and 0.2);
        \draw [dashed] (-0.7,-1.2) arc (180:360:0.7 and -0.2);
        \draw (0.7,-1.2) -- (0.7,0);
        \draw (-0.7,-1.2) -- (-0.7,-2.4);
        \draw (-0.7,-2.4) arc (180:360:0.7 and 0.2);
        \draw [dashed] (-0.7,-2.4) arc (180:360:0.7 and -0.2);
        \draw (0.7,-2.4) -- (0.7,-1.2);
        \fill [gray,opacity=0.3] (-0.7,-1.2) arc (180:360:0.7 and 0.2) -- (-0.7,-1.2) arc (180:360:0.7 and -0.2);
    \end{tikzpicture}} \ar@{-->}[rrrr]^{h''} \restore &&&& Y \\
&& \save []+<0.5cm, -0.7cm>* {\begin{tikzpicture}
        \draw (0,0) ellipse (0.7 and 0.2);
        \draw (-0.7,0) -- (-0.7,-1.2);
        \draw (-0.7,-1.2) arc (180:360:0.7 and 0.2);
        \draw [dashed] (-0.7,-1.2) arc (180:360:0.7 and -0.2);
        \draw (0.7,-1.2) -- (0.7,0);  
        \fill [gray,opacity=0.3] (0,0) ellipse (0.7 and 0.2);
    \end{tikzpicture}} \ar@<-15pt>@/^0.2pc/[urrr] \ar@<-15pt>@/_1.5pc/[urrrrrrr]_{h'} \restore \\                  
&&\\
&& W' \\}$

\end{comentario}

\bigskip
\subsection{Construcci\'on de \texorpdfstring{$\catH$}{}}
Recordemos que el problema que queremos resolver consiste en definir una 2-categor\'ia $\catH$ junto con un 2-funtor $i: \mathscr{C} \longrightarrow \catH$ que tenga la propiedad universal de la 2-localización:

\begin{equation} \label{diagrama}
    \xymatrix{
        \mathscr{C} \hspace{1mm} \ar[rr]^{i} \ar[dr]_{F} && \catH \ar@{-->}[dl]^(.5){\exists \widetilde{F}} \\
        & \mathscr{D}}
\end{equation}
para todo 2- funtor $F:\mathscr{C} \longrightarrow \mathscr{D}$ tal que $F(W) \subseteq Equiv(\mathscr{D})$.

Con el objetivo de establecer una relación entre homotopías para que sean parte de la estructura de 2-categoría de $\catH$, en adelante vamos a trabajar con una generalización de los cilindros y de las homotopías de Quillen que ya conocemos.
Estas nuevas definiciones pueden ser introducidas en el contexto de una categoría con una única clase de morfismos $\Sigma$ conteniendo a las identidades, y por el momento trabajaremos con estas nociones prescindiendo de la estructura de modelos de $\mathscr{C}$.

\begin{definition} \label{cilindros_homotopias}
     Sea $\Sigma$ una familia de morfismos en una categor\'ia $\mathscr{C}$ conteniendo a todas las identidades. Un \emph{cilindro} $C=(W, Z, d_0, d_1, s, x)$ para un objeto $X$ en $\mathscr{C}$ es una configuración
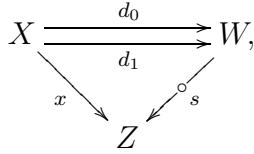
    \begin{center}
        $\xymatrix{
     X  \ar@<3pt>[rr]^{d_0}
        \ar@<-3pt>[rr]_{d_1}
        \ar[dr]_{x} && W, 
        \ar[dl]^{s}| \circ\\
        &  Z \\  }$
    \end{center}
    
    donde $s \in \Sigma$ y $sd_0=sd_1=x$.
    \bigskip
    
    Una $\emph{homotopía (a izquierda)}$ $H=(C, h)$ de $f$ a $g$ con cilindro \\ $C=(W, Z, d_0, d_1, s, x)$ es una flecha $h:W \longrightarrow Y$ cumpliendo $hd_0=f$ y $hd_1=g$. Notamos, como antes, $H:
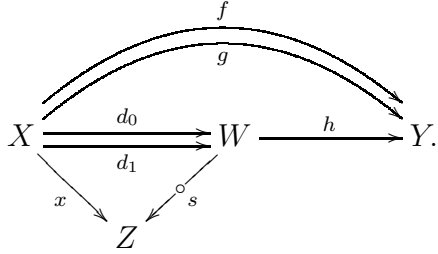
\xymatrix{f \ar@2{~>}[r] & g}$.

    \begin{center}
    $\xymatrix{
     X  \ar@/^3pc/@<6pt>[rrrr]^{f}
        \ar@/^3pc/[rrrr]_{g}
        \ar@<2pt>[rr]^{d_0}
        \ar@<-3pt>[rr]_{d_1}
        \ar[dr]_{x} && W 
        \ar[rr]^{h} \ar[dl]^{s}|\circ  && Y. \\
        &  Z \\  }$
\end{center}

\end{definition}

\bigskip
\subsubsection{Clases de equivalencia de homotop\'ias}
La relaci\'on de equivalencia que buscamos se desprender\'a de la siguiente observaci\'on, que adem\'as sugiere c\'omo debemos definir el 2-funtor $\widetilde{F}$ del diagrama \ref{diagrama} en las clases de homotop\'ia.

\begin{observation}

Sea $\mathscr{D}$ una 2-categoría y consideremos en $\mathscr{D}$ el siguiente diagrama conmutativo

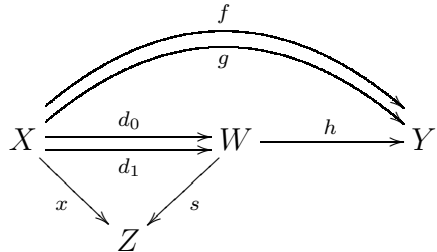
\begin{center}
    $\xymatrix{
     X  \ar@/^3pc/@<6pt>[rrrr]^{f}
        \ar@/^3pc/[rrrr]_{g}
        \ar@<2pt>[rr]^{d_0}
        \ar@<-3pt>[rr]_{d_1}
        \ar[dr]_{x} && W 
        \ar[rr]^{h} \ar[dl]^{s}  && Y \\
        &  Z \\  }$
\end{center}
donde el morfismo $s$ es una verdadera \emph{equivalencia} .

Para cada objeto $X'$ en $\mathscr{D}$, $s$ induce un funtor plenamente fiel

\vspace{3mm}

\begin{center}
    $\xymatrix{ \mathscr{D}[X', W] \ar[rr]^{s_{*}} && \mathscr{D}[X', Z]
    }$
\end{center}

\vspace{3mm}
y, tomando $X'=X$, existe entonces una única 2-celda $\widehat{C}: d_0 \Longrightarrow d_1$ tal que $s\widehat{C}=x$ (ver \ref{s*_equivalencia}):

\vspace{3mm}

\begin{center}
    \begin{tikzcd}
    X \arrow[shift left=10pt, "d_0"]{r}[name=LUU, below]{}
    \arrow[shift right=10pt, "d_1"']{r}[name=LDD]{}
    \arrow[Rightarrow,to path=(LUU) -- (LDD)\tikztonodes]{r}{\hat{C}}
    & W
    \arrow[ r, "s"]
    & Z
  \end{tikzcd}
  \hspace{2mm}$=$ \hspace{2mm}
  \begin{tikzcd}
    X \arrow[shift left=10pt, "sd_0=x"]{rr}[name=LUU, below]{}
    \arrow[shift right=10pt, "sd_1=x"']{rr}[name=LDD]{}
    \arrow[Rightarrow,to path=(LUU) -- (LDD)\tikztonodes]{r}{id_x}
    && Z
  \end{tikzcd}
\end{center}

\vspace{10mm}

Sean ahora $\xymatrix{ \mathscr{C} \ar[r]^{F} & \mathscr{D}}$ un 2-funtor que manda las flechas de la clase $\Sigma$ en equivalencias, y $H=(C, h)$ una homotop\'ia de $f$ a $g$ en $\mathscr{C}$ con cilindro $C=(W, Z, d_0, d_1, s, x)$.

Aplicando $F$ al diagrama en la definici\'on \ref{cilindros_homotopias} resulta

\vspace{3mm}

\begin{center}
    $\xymatrix{
     FX  \ar@/^3pc/@<6pt>[rrrr]^{Ff}
        \ar@/^3pc/[rrrr]_{Fg}
        \ar@<2pt>[rr]^{Fd_0}
        \ar@<-3pt>[rr]_{Fd_1}
        \ar[dr]_{Fx} && FW 
        \ar[rr]^{Fh} \ar[dl]^{Fs}  && FY, \\
        &  FZ \\  }$
\end{center}

que es un diagrama conmutativo en $\mathscr{D}$.

Como $Fs$ es una equivalencia, hay una \'unica 2-celda $\widehat{FC}: Fd_0 \Longrightarrow Fd_1$ que satisface $Fs\widehat{FC}=Fx$.

\end{observation}

\bigskip

\begin{definition} \label{FH}
    Dada una homotop\'ia $H=(C, h)$ en $\mathscr{C}$ y dado un 2-funtor $F: \mathscr{C} \longrightarrow \mathscr{D}$ que manda la clase $\Sigma$ en equivalencias, definimos una 2-celda $\widehat{FH}$ en $\mathscr{D}$ como $\widehat{FH}:=Fh\widehat{FC} : Ff \Longrightarrow Fg$, donde $\widehat{FC}: Fd_0 \Longrightarrow Fd_1$ es la \'unica tal que  $Fs\widehat{FC}=Fx$.
\end{definition}
    
\begin{definition} \label{rel_adhoc}
    Sean $f, g: X \longrightarrow Y$ morfismos en $\mathscr{C}$ y $H, H': \xymatrix{f \ar@2{~>}[r] & g}$ dos homotopías. Decimos que $H \sim H'$ si y sólo si $\widehat{FH}=\widehat{FH'}$ para todo 2-funtor $F: \mathscr{C} \longrightarrow \mathscr{D}$ tal que $F(\Sigma) \subseteq Equiv(\mathscr{D})$, para toda 2-categoría $\mathscr{D}$. 
    
    Más generalmente, dadas homotopías $\xymatrix{ f \ar@2{~>}[r]^{H} & g \ar@2{~>}[r]^{K} & l}$ y $\xymatrix{ f \ar@2{~>}[r]^{H'} & g' \ar@2{~>}[r]^{K'} & l}$, decimos que $(K, H) \sim (K', H')$ si y sólo si $\widehat{FK}\circ\widehat{FH} = \widehat{FK'}\circ\widehat{FH'}$, para todo 2-funtor $F:\mathscr{C} \longrightarrow \mathscr{D}$ que manda la clase $\Sigma$ en equivalencias de $\mathscr{D}.$
\end{definition}

\vspace{6mm}

Es claro c\'omo de esta manera podemos establecer una relaci\'on de equivalencia entre secuencias de homotop\'ias componibles. Definimos una 2-celda en $\catH$ como la clase $[H_n,..., H_1]$ de una secuencia finita de homotop\'ias $\xymatrix{ f_0 \ar@2{~>}[r]^(.3){H_1} & f_{1} \ldots f_{n-1} \ar@2{~>}[r]^(.7){H_n} & f_n }$, donde 

\begin{center}
    $(H_n, ..., H_1) \sim (K_m, ..., K_1)$ si y s\'olo si $\widehat{FH_n} \circ ... \circ \widehat{FH_1} = \widehat{FK_m}\circ ... \circ \widehat{FK_1}$.
\end{center}

\bigskip
Veamos que, así, $\catH$ es efectivamente una 2-categoría.

\subsubsection{Composición vertical}

Definimos la composición vertical de homotopías como la yuxtaposición $[K]\circ [H] = [K, H]$.
\begin{center}
    \begin{tikzcd}
    X \arrow[bend left=50, "f"]{rr}[name=LUU, below]{}
    \arrow["\hspace{-6mm}g"]{rr}[name=LUD]{}
    \arrow[swap]{rr}[name=LDU]{}
    \arrow[bend right=50, "l"']{rr}[name=LDD]{}
    \arrow[Rightarrow,to path=(LUU) -- (LUD)\tikztonodes]{r}{[H]}
    \arrow[Rightarrow,to path=(LDU) -- (LDD)\tikztonodes]{r}{[K]}
    && 
    Y
  \end{tikzcd}
  =
      \begin{tikzcd}
    X \arrow[bend left=50, "f"]{rr}[name=LUU, below]{}
    \arrow[bend right=50, "l"']{rr}[name=LDD]{}
    \arrow[Rightarrow,to path=(LUU) -- (LDD)\tikztonodes]{r}{[K, H]}
    && 
    Y
  \end{tikzcd}
\end{center}

\bigskip

La asociatividad es una consecuencia de la asociatividad de la composición vertical en $\mathscr{D}$. En efecto, para todo $F: \mathscr{C} \longrightarrow \mathscr{D}$ tal que $F(\Sigma)\subseteq Equiv(\mathscr{D})$

\begin{center}

$[L, K]\circ [H]=[L]\circ [K, H]$ si y sólo si $(\widehat{FL}\circ \widehat{FK})\circ \widehat{FH}=\widehat{FL}\circ (\widehat{FK}\circ\widehat{FH})$.

\end{center}

\begin{observation} \label{secuencias}
    Las clases de secuencias de una sola homotop\'ia generan las 2-celdas en $\catH$.
\end{observation}

\bigskip
\begin{id_vertical}

Sea $f \in \mathscr{C}[X,Y]$ y sea $H: \xymatrix{f \ar@2{~>}[r] & f}$ una homotop\'ia con cilindro $C$.

Por c\'omo est\'a definida la relaci\'on de equivalencia entre secuencias de homotop\'ias, es claro que
    $[H]= Id_f$ en $\catH$ si y s\'olo si $\widehat{FH}=Id_{Ff}$ en $\mathscr{D}$ para todo $F \in Hom_+(\mathscr{C}, \mathscr{D}).$

As\'i, por ejemplo, la homotop\'ia $H$ dada por

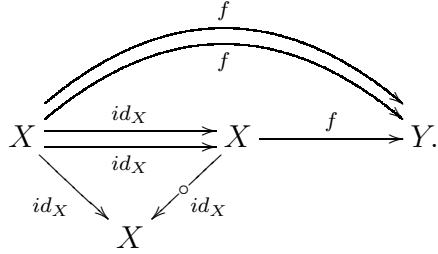
\begin{center}
      $\xymatrix{
     X  \ar@/^3pc/@<6pt>[rrrr]^{f}
        \ar@/^3pc/[rrrr]_{f}
        \ar@<3pt>[rr]^{id_X}
        \ar@<-3pt>[rr]_{id_X}
        \ar[dr]_{id_X} && X 
        \ar[rr]^{f} \ar[dl]^{id_X}|\circ  && Y. \\
        &  X \\  }$
\end{center}
es tal que $[H]=Id_f$. En efecto, por definici\'on tenemos que $\widehat{FH}=Ff\widehat{FC}$, donde $\widehat{FC}:id_{FX} \Longrightarrow id_{FX}$ es como en la definici\'on \ref{FH}, y la 2-celda $Id_{id_{FX}}$ satisface $id_{FX}*Id_{id_{FX}}=id_{FX}$, entonces $\widehat{FC}=Id_{id_{FX}}$, y luego $\widehat{FId_f}=Id_{Ff}$.

An\'alogamente puede verse que si $H$ es la homotop\'ia \hspace{3mm}
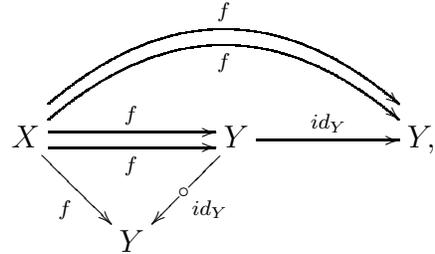
\begin{center}
$\xymatrix{
     X  \ar@/^3pc/@<6pt>[rrrr]^{f}
        \ar@/^3pc/[rrrr]_{f}
        \ar@<3pt>[rr]^{f}
        \ar@<-3pt>[rr]_{f}
        \ar[dr]_{f} && Y 
        \ar[rr]^{id_Y} \ar[dl]^{id_Y}|\circ  && Y, \\
        &  Y \\  }$
\end{center}
entonces $[H]=Id_f.$

Denotamos $I_f$ a cualquier homotop\'ia $H$ tal que $\widehat{FH}=Id_{Ff}.$

\end{id_vertical}

\bigskip

De esta forma, $\mathscr{C}[X,Y]$ es una categoría para cada par de objetos $X$, $Y$ en $\mathscr{C}$.

\bigskip

\subsubsection{Composición horizontal} \label{composicion_horiz_def}
Dado que ya tenemos composición vertical, si definimos la composiciones horizontales con las identidades $l*[H]=[Id_l]*[H]$ y $[H]*r=[H]*[Id_r]$, la composición horizontal de 2-celdas se obtiene de la siguiente manera:

\begin{center}
    \begin{tikzcd}
    X \arrow[bend left=70, "f"]{r}[name=LUU, below]{}
    \arrow[bend right=70, "g"']{r}[name=LDD]{}
    \arrow[Rightarrow,to path=(LUU) -- (LDD)\tikztonodes]{r}{H}
    & \hspace{1mm} Y \arrow[bend left=70, "f'"]{r}[name=RUU, below]{}
    \arrow[bend right=70, "g'"']{r}[name=RDD]{}
    \arrow[Rightarrow,to path=(RUU) -- (RDD)\tikztonodes]{r}{H'}
    & \hspace{1mm} Y'
  \end{tikzcd}
  $\underset{\text{def}}{=}$
  \begin{tikzcd}
    X \arrow[bend left=50, "f'f"]{rr}[name=LUU, below]{}
    \arrow{rr}[name=LUD]{}
    \arrow[swap]{rr}[name=LDU]{}
    \arrow[bend right=50, "g'g"']{rr}[name=LDD]{}
    \arrow[Rightarrow,to path=(LUU) -- (LUD)\tikztonodes]{rr}{f'*H}
    \arrow[Rightarrow,to path=(LDU) -- (LDD)\tikztonodes]{rr}{H'*g}
    &&\hspace{1mm} Y'
  \end{tikzcd}
  $\underset{(1)}{=}$
   \begin{tikzcd}
    X \arrow[bend left=50, "f'f"]{rr}[name=LUU, below]{}
    \arrow{rr}[name=LUD]{}
    \arrow[swap]{rr}[name=LDU]{}
    \arrow[bend right=50, "g'g"']{rr}[name=LDD]{}
    \arrow[Rightarrow,to path=(LUU) -- (LUD)\tikztonodes]{rr}{H'*f}
    \arrow[Rightarrow,to path=(LDU) -- (LDD)\tikztonodes]{rr}{g'*H}
    &&\hspace{1mm} Y',
  \end{tikzcd}
\end{center}
siempre y cuando se verifique la igualdad $(1)$, de acuerdo a la observaci\'on \ref{horizontal_comp}.

\vspace{8mm}
    
Sean $\xymatrix{X'\ar[r]^{l} & X \ar@<4pt>[r]^{f} \ar@<-3pt>[r]_{g} & Y \ar[r]^{r} & Y' }$  y $H= (C, h):\xymatrix{f \ar@2{~>}[r] & g}$ una homotopía con cilindro $C=(W, Z, d_0, d_1, s, x)$.

Consideramos $Hl=(Cl, h): \xymatrix{fl \ar@2{~>}[r] & gl}$ y $rH=(C, rh): \xymatrix{rf \ar@2{~>}[r] & rg}$, donde $Cl=(W, Z, d_0l, d_1l, s, xl)$ es un cilindro para $X'$.

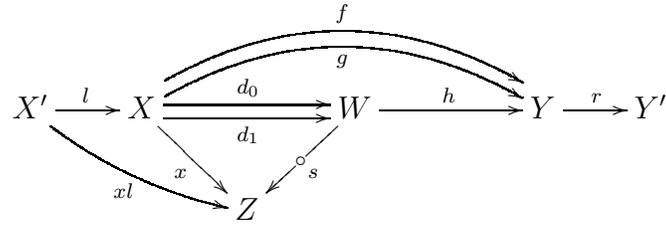
\begin{center}
    $\xymatrix{
     X'\ar[r]^{l} \ar@/_0.5pc/[drr]_{xl}
     & X  \ar@/^2pc/@<6pt>[rrrr]^{f}
        \ar@/^2pc/[rrrr]_{g}
        \ar@<2pt>[rr]^{d_0}
        \ar@<-3pt>[rr]_{d_1}
        \ar[dr]_{x} && W 
        \ar[rr]^{h} \ar[dl]^{s}|\circ  && Y \ar[r]^{r}& Y'\\
        &&  Z \\  }$
\end{center}

\bigskip

Es claro que tanto $Hl$ como $rH$ resultan homotopías. Adem\'as, se tienen las ecuaciones
\begin{equation} \label{eq}
    \widehat{F(Hl)}=\widehat{FH}Fl \text{  y  } \widehat{F(rH)}=Fr\widehat{FH},
\end{equation}
ya que, como $\widehat{F(Hl)}=Fh\widehat{F(Cl)}$ y $Fs(\widehat{FC}Fl)=(Fs\widehat{FC})Fl=xl$, por la unicidad de $\widehat{F(Cl)}$ en la definici\'on \ref{FH} se tiene que $\widehat{F(Cl)}=\widehat{FC}Fl$ y, luego, $\widehat{F(Hl)}=\widehat{FH}Fl$. La segunda ecuaci\'on es evidente, dado que $H$ y $rH$ tienen el mismo cilindro.

\bigskip
Definimos $[H]*l=[Hl]$ y $r*[H]=[rH]$.
Más generalmente,
\begin{center}
    $[K, H]*l:=[Kl, Hl]$ \hspace{2mm} y \hspace{2mm} $r*[K, H]:=[rK, rH]$.
\end{center}

Ahora, si $H \sim H'$, entonces
\begin{center}
$\widehat{F(Hl)}=\widehat{FHFl}=\widehat{FH}Fl=\widehat{FH'}Fl=\widehat{FH'Fl}=\widehat{F(H'l)}$.
\end{center}
Esto nos dice que la composición con $l$ está bien definida, y de la misma manera se ve la buena definici\'on de la composici\'on con $r$.

\bigskip
De las ecuaciones $\widehat{F(Hl)}=\widehat{FH}Fl$ y $\widehat{F(rH)}=Fr\widehat{FH}$, y de la compatibilidad entre las composiciones vertical y horizontal en $\mathscr{D}$, se deduce la igualdad $(1)$ requerida, correspondiente al \'ultimo axioma en \ref{horizontal_comp}.

Adem\'as, por definici\'on tambi\'en se tiene que
\begin{center}
    $([K]*l)\circ ([H]*l)=([Kl])\circ ([Hl])=[Kl,Hl]=[K, H]*l=([K]\circ [H])*l$,
\end{center}

\begin{center}
    $[I_f]*l=[I_fl]=[I_{fl}].$
\end{center}

Luego, los axiomas de la observaci\'on \ref{horizontal_comp} se verifican, por lo que la composici\'on horizontal de 2-celdas en $\catH$ queda determinada y es compatible con la composici\'on vertical.
\\

En virtud de las mismas ecuaciones (\ref{eq}) se tiene que la composición horizontal también es asociativa, y las identidades en este caso son como en $\mathscr{C}$:
para cada objeto $X$, tenemos la 2-celda $[I_{id_X}]=[Id_{id_X}]$, que escribimos $[Id_X]$ para simplificar la notación.

\vspace{6mm}
\begin{comentario}
    En el caso en el que $\Sigma$ es la clase $\mathcal{W}$ de equivalencias d\'ebiles de una categor\'ia de modelos, la composición $r[H]=[rH]$ está bien definida incluso si $H$ es una $q$-homotopía, ya que $rH$ también lo es, pero no ocurre lo mismo con $Hl$. Sin embargo, veremos m\'as adelante que cuando nos restringimos a la subcategor\'ia $\mathscr{C}_{fc}$, dada cualquier homotop\'ia $H$ existe una homotopía de Quillen que está en la misma clase, de manera que ser\'a posible definir la composición con $l$ cuando los objetos sean fibrantes y cofibrantes (ver \ref{lema_3_pasos}).
\end{comentario}

\subsection{Propiedades de \texorpdfstring{$\catH$}{} y 2-localizaci\'on de \texorpdfstring{$\mathscr{C}_{fc}$}{}}

Junto con la 2-categor\'ia $\catH$ se obtiene un 2-funtor $i: \mathscr{C} \longrightarrow \catH $ dado por la inclusi\'on, y si bien no manda equivalencias d\'ebiles en equivalencias de $\catH$, tiene la siguiente propiedad universal.

\begin{proposition} \label{pu_i}
    Sean $i: \mathscr{C} \longrightarrow \catH $ la inclusión, $\mathscr{D}$ una 2-categoría y $F: \mathscr{C} \longrightarrow \mathscr{D}$ un 2-funtor que manda los elementos de $\Sigma$ en equivalencias. Entonces, existe un único 2-funtor $\widetilde{F}: \catH \longrightarrow \mathscr{D}$ tal que $\widetilde{F}X=FX$ y $\widetilde{F}f=Ff$:
    
\begin{center}
     $\xymatrix{
        \ar @{} [drr] |(.45){} 
        \mathscr{C} \hspace{1mm} \ar@{^{(}->}[rr]^{i} \ar[dr]_{F} && \catH \ar@{-->}[dl]^(.5){\exists! \widetilde{F}} \\
        & \mathscr{D} & \\
    }$
\end{center}

\end{proposition} 

\begin{proof}
    Definimos $\widetilde{F}$ en las clases de homotop\'ias como
    \begin{equation}
        \widetilde{F}([H])= \widehat{FH}, \label{functor_hpy}   
    \end{equation}
     y por la observaci\'on \ref{secuencias} podemos extender funtorialmente esta definici\'on a cualquier 2-celda en $\catH$, siendo $$\widetilde{F}([H_n, ..., H_1]) = \widehat{FH_n} \circ ... \circ \widehat{FH_1}.$$ 
    
    De este modo, $\widetilde{F}$ resulta un funtor para la composici\'on vertical.
    
    Como $\widetilde{F}l\widetilde{F}([H])=Fl\widehat{FH}=\widehat{FlFH}=\widehat{F(lH)}= \widetilde{F}([lH])$ y, análogamente, $\widetilde{F}([H])\widetilde{F}r=\widetilde{F}([Hr])$, entonces de la definici\'on \ref{composicion_horiz_def} se sigue que $\widetilde{F}$ es funtorial respecto a la composición horizontal.
    
    \vspace{5mm}
    
    Queremos ver ahora la unicidad de $\widetilde{F}$.
    
    Sea $R:\catH \longrightarrow \mathscr{D}$ un 2-funtor tal que $Ri=F$.
    Si $H=(C, h)$ es una homotopía en $\mathscr{C}$ con cilindro $C=(W, Z, d_0, d_1, s, x)$, escribimos $H=hH_0$ donde $H_0=(C, id_W)$.
    \begin{center}
       $\xymatrix{
        X
        \ar@<3pt>[rr]^{d_0}
        \ar@<-3pt>[rr]_{d_1}
        \ar[dr]_{x} && W 
        \ar[r]^{id_W} \ar[dl]^{s}|\circ  & W \ar[r]^{h} & Y\\
        &  Z\\  }$
    \end{center}
    
    Como $R([H])=RhR([H_0])=FhR([H_0])$ y $\widetilde{F}(H)=\widehat{FH}=Fh\widehat{Fc}$, para probar que $R$ coincide con $\widetilde{F}$ en $[H]$, alcanza con ver que $R([H_0])=\widehat{Fc}$. 
    
    Sabemos que $\widehat{Fc}: Fd_0 \Longrightarrow Fd_1$ es la única tal que $Fs\widehat{Fc}=id_x$.
    Por otro lado $FsR([H_0])=RsR([H_0])=R(s[H_0])=R([sH_0])$, y $[sH_0]=[id_x]$ ya que $Fid_xid_{Fx}=id_{Fx}=Fs\widehat{Fc}$.
    \begin{center}
        $\xymatrix{
        X
        \ar@<3pt>[rr]^{d_0}
        \ar@<-3pt>[rr]_{d_1}
        \ar[dr]_{x} && W 
        \ar[rr]^{s} \ar[dl]^{s}|\circ && Z\\
        &  Z\\  }$
        \hspace{3mm} $\sim$ \hspace{3mm}
        $\xymatrix{
        X
        \ar@<3pt>[rr]^{x}
        \ar@<-3pt>[rr]_{x}
        \ar[dr]_{x} && Z 
        \ar[rr]^{id_Z} \ar[dl]^{id_Z}|\circ && Z\\
        &  Z\\  }$
    \end{center}
    
    Luego, $FsR([H_0])=R([id_{Fx}])=id_{Rx}=id_{Fx}$ y, por unicidad de $\widehat{Fc}$, es $R([H_0])=\widehat{Fc}$.
    
\end{proof}

\bigskip

Denotaremos $Hom_{p_+}(\mathscr{C},\mathscr{D})$ y $Hom_{s_+}(\mathscr{C},\mathscr{D})$ a las subcategoría de $Hom_p(\mathscr{C}, \mathscr{D})$ y $Hom_s(\mathscr{C}, \mathscr{D})$, respectivamente, cuyos objetos son los funtores $F$ tales que $F(\Sigma)\subseteq Equiv(\mathscr{D})$.

La última proposición nos dice que la precomposición
\begin{center}
    $i^*: Hom_{s_+}(\catH, \mathscr{D})\longrightarrow Hom_{s_+}(\mathscr{C},\mathscr{D})$ 
\end{center}
es un funtor biyectivo en los objetos; en particular, es esencialmente suryectivo. Veamos que además es plenamente fiel.

\begin{lemma} \label{clave}
    Sean $\mathscr{D}$ una 2-categoría, $F, G \in Hom_+(\mathscr{C},\mathscr{D})$ y un cilindro $C=(W, Z, d_0, d_1, s, x)$ para un objeto $X \in \mathscr{C}$. Si $\theta: F \Longrightarrow G$ es una transformación natural, entonces vale $\theta_W\widehat{FC}=\widehat{GC}\theta_X$.
\end{lemma}
\begin{center}
    $\xymatrix{
        FX \ar[rr]^{\theta_X} \ar@<-10pt>[dd]_{Fd_0} \ar@{}[dd]|{\overset{\widehat{FC}}{\Rightarrow}} \ar@<10pt>[dd]^{Fd_1} && GX \ar@<-10pt>[dd]_{Gd_0} \ar@{}[dd]|{\overset{\widehat{GC}}{\Rightarrow}} \ar@<10pt>[dd]^{Gd_1} \\
        \\
        FW \ar[rr]_{\theta_W} && GW }$
\end{center}

\begin{proof}
    Como $\widehat{GC}:Gd_0 \Longrightarrow Gd_1$ es la \'unica 2-celda que satisface $Gs\widehat{GC}=Gx$, entonces $\widehat{GC}\theta_X$ es la única tal que $Gs(\widehat{GC}\theta_X)=Gx\theta_X$. Luego, basta ver que $Gs(\theta_W\widehat{FC})=Gx\theta_X$.
    
    De la naturalidad de $\eta$ se obtiene $Gs\theta_W=\theta_ZFs$ y $\theta_ZFx=Gx\theta_X$.
    
    \begin{center}
        $\xymatrix{
            FW \ar[r]^{Fs}\ar[d]_{\theta_W} & FZ\ar[d]^{\theta_Z}\\
            GW\ar[r]_{Gs}& GZ\\
        }$
        \hspace{25mm}
        $\xymatrix{
            FX \ar[r]^{Fx}\ar[d]_{\theta_X} & FX\ar[d]^{\theta_Z}\\
            GX\ar[r]_{Gx}& GZ\\
        }$
    \end{center}
    
    \bigskip
    
    Por lo tanto,
        $Gs(\theta_W\widehat{FC})=(Gs\theta_W)\widehat{FC}=(\theta_ZFs)\widehat{FC}=\theta_ZFx=Gx\theta_X$.

\end{proof}

\begin{proposition} \label{prop_clave}
    Si $F, G: \mathscr{C}\longrightarrow\mathscr{D}$ mandan las flechas de $\Sigma$ en equivalencias y $\theta: F \Longrightarrow G$ es una transformación natural, entonces existe una única transformación 2-natural $\widetilde{\theta}:\widetilde{F}\Longrightarrow\widetilde{G}$ tal que $\widetilde{\theta}i=\theta$.
\end{proposition}

\begin{proof}
    Para cada $X$ en $\mathscr{C}$, sabemos que $\widetilde{F}i=F$ y $\widetilde{G}i=G$. Definimos $\widetilde{\theta}_X :\widetilde{F}X \longrightarrow \widetilde{G}X$ como $\widetilde{\theta}_X=\theta_X$ y veamos, entonces, que $\widetilde{\theta}$ satisface las condiciones de 2-naturalidad.
    
    Sean $f, g: X \longrightarrow Y$ y $[H]: f \Longrightarrow g$ una 2-celda en $\catH$. Queremos probar la igualdad $\widetilde{\theta}_Y\widetilde{F}[H]=\widetilde{G}[H]\widetilde{\theta}_X$. Si $H=(C, h)$, de la naturalidad de $\theta$ se tiene $\theta_YFh=Gh\theta_X$, por lo que $\theta_Y\widehat{FH}=\theta_YFh\widehat{FC}=Gh\theta_W\widehat{FC}$. Además, $\widehat{GH}\theta_X=Gh\widehat{GC}\theta_X$. Del lema anterior obtenemos $\theta_W\widehat{FC}=\widehat{GC}\theta_X$ y, luego, $\theta_Y\widehat{FH}=\widehat{GH}\theta_X$, que por definición de $\widetilde{F}, \widetilde{G}$ y de $\widetilde{\theta}$, es exactamente lo que queríamos ver.
    
\end{proof}

\bigskip

Como consecuencia de los resultados anteriores obtenemos que la precomposición con $i$ induce, para toda 2-categoría $\mathscr{D}$, una equivalencia de categorías, que de hecho es un isomorfismo. En cuanto al aspecto 2-categórico, se tiene el siguiente

\begin{lemma} \label{aspecto_2categorico}
     Sean $F, G \in Hom_+(\mathscr{C}, \mathscr{D})$. Si $\eta$, $\theta$: $F \Longrightarrow G$ son transformaciones naturales y $\mu: \eta \longrightarrow \theta$ una modificación, existe una única $\Tilde{\mu}: \Tilde{\eta} \longrightarrow \Tilde{\theta}$ tal que $\Tilde{\mu}i=\mu$.
\end{lemma}

\begin{proof}
    Definimos $\Tilde{\mu}: \Tilde{\eta} \longrightarrow \Tilde{\theta}$ como $\Tilde{\mu}_X=\mu_X$ para cada $X$.
    
    Sean $f, g: X \longrightarrow Y$ y $H$ una homotopía de $f$ a $g$. Queremos ver que $\Tilde{\mu}_Y\Tilde{F}[H]=\Tilde{G}[H]\Tilde{\mu}_X$; es decir, $\mu_Y\widehat{FH}=\widehat{GH}\mu_X$. Como $\mu$ es una modificación, $\theta_YFg=Gg\mu_X$, y por el lema \ref{clave} también tenemos $\eta_W\widehat{FC}=\widehat{GC}\eta_X$, de manera que
    \begin{equation*}
    \begin{split}
        \mu_Y\widehat{FH} & =\mu_YFg \circ \eta_Y \widehat{FH}
                          =Gg\mu_X\circ\eta_YFh\widehat{FC}
                          =Gg\mu_X \circ Gh\eta_W\widehat{FC}\\
                          & = Gg\mu_X \circ Gh\widehat{GC}\eta_X
                          =Gg \mu_X \circ \widehat{GH}\eta_X
                          =\widehat{GH}\mu_X.
    \end{split}
    \end{equation*}
\end{proof}

\begin{corollary} \label{hom_iso}
    El funtor $i^*: Hom_{s_+}(\catH, \mathscr{D})\longrightarrow Hom_{s_+}(\mathscr{C}, \mathscr{D})$ es un isomorfismo de 2-categorías. 
\end{corollary}

\begin{Observation} \label{hom_p_iso}
    El corolario anterior tambi\'en puede obtenerse en t\'erminos de 2-funtores y transformaciones pseudonaturales con una demostraci\'on muy similar a la que exhibimos para los resultados \ref{clave}, \ref{prop_clave} y \ref{aspecto_2categorico}. Es decir, la inclusi\'on $i$ induce un isomorfismo de 2-categor\'ias $$i^*: Hom_{p_+}(\catH, \mathscr{D})\longrightarrow Hom_{p_+}(\mathscr{C}, \mathscr{D}).$$
\end{Observation}

\vspace{6mm}

Fijamos ahora una categor\'ia de modelos $\mathscr{C}$. Para que la 2-categor\'ia $\catH$ sea la localización de $\mathscr{C}$ con respecto a la clase $\mathcal{W}$, s\'olo necesitamos que la inclusión $i: \mathscr{C}\longrightarrow \catH$ mande las equivalencias d\'ebiles en equivalencias. Sin embargo, vamos a poder demostrar esto si nos restringimos a la subcategoría plena $\mathscr{C}_{fc}$ de objetos fibrantes-cofibrantes, obteniendo de esta forma la propiedad universal de la 2-localización para $\mathscr{C}_{fc}$ respecto de $\mathcal{W}$. La demostraci\'on requiere del siguiente lema.

\bigskip

\begin{lemma} \label{2_cells_inv}
    Toda 2-celda en $\catH$ es inversible.
\end{lemma}

\begin{proof}
    Fijemos $H=(C, h)$ una homotopía de $f$ a $g$ con cilindro $C=(W, Z, d_0, d_1, s, x)$. Definimos $H^{-1}=(C^{-1}, h)$, donde $C^{-1}$ es el cilindro que se obtiene de $C$ intercambiando $d_0$ y $d_1$, por lo que $H^{-1}$ es una homotopía de $g$ a $f$. Veamos que $[H^{-1}]\circ[H]=[Id_f]$.
    
    Sea $F: \mathscr{C}\longrightarrow\mathscr{D}$ un 2-funtor. Tenemos $\widehat{FC}: Fd_0 \Longrightarrow Fd_1$ y $\widehat{FC^{-1}}: Fd_1 \Longrightarrow Fd_0$ satisfaciendo las ecuaciones $Fs\widehat{FC}=Fx$ y $Fs\widehat{FC^{-1}}=Fx$. Luego, para la composición se tiene $Fs(\widehat{FC^{-1}}\circ\widehat{FC})=Fx\circ Fx=Fx$, pero como existe una única 2-celda de $Fd_0$ en $Fd_0$ que satisface la igualdad anterior, entonces $\widehat{FC^{-1}}\circ\widehat{FC}=Id_{Fd_0}=Fd_0$. De esta forma,
    \begin{center}
    $\widehat{FH^{-1}}\circ \widehat{FH}=Fh\widehat{FC^{-1}} \circ Fh\widehat{FC}= Fh(\widehat{FC^{-1}} \circ \widehat{FC})=FhFd_0=Ff$;
    \end{center}
    es decir, $[H^{-1}, H]=[Id_f]$. Una cuenta similar muestra que $[H, H^{-1}]=[Id_g]$, por lo tanto $[H]$ es inversible y su inversa es $[H]^{-1}=[H^{-1}]$. 
    
\end{proof}

\begin{definition}
    Un morfismo $f:X \longrightarrow Y$ es una \emph{sección} si admite una inversa a izquierda; es decir, existe $g: Y\longrightarrow X$ que satisface $gf=id_X$.
    
    Dualmente, decimos que $f$ es una \emph{retracción} si tiene una inversa a derecha.
\end{definition}

\begin{theorem}
    Sea $s: X \longrightarrow Y$ una equivalencia débil en una categor\'ia de modelos $\mathscr{C}$. Si $X$ es fibrante e $Y$ es cofibrante, entonces $s$ es una equivalencia en $\catH$.
\end{theorem}

\begin{proof}
    Consideramos una factorización $s=pi$, donde $p$ es una fibración, $i$ una cofibración y alguna de las dos es una equivalencia débil, usando el axioma M4. Como $s$ es una equivalencia débil, por M5 obtenemos que los tres morfismos lo son.
    
    Ahora, dado que $X$ es fibrante, $i$ resulta una sección: como es una cofibración trivial y además $X\longrightarrow 1$ es una fibración, la propiedad de levantamiento garantiza la existencia de la flecha punteada en el diagrama conmutativo $\xymatrix{
            X \ar@{=}[r] \ar@{ >->}[d]_i | \circ & X \ar@{->>}[d]\\
            Z \ar@{-->}[ur] \ar[r] & 1\\},$
    por lo que $i$ tiene una inversa a izquierda.
    
    Dualmente, $p$ es una retracción gracias a que $Y$ es cofibrante.
    Luego, basta ver que si una equivalencia débil es además una sección o una retracción, es una equivalencia en $\catH$, ya que las equivalencias son cerradas por composición.
    
    Supongamos entonces que $s$ es una sección con $r:Y \longrightarrow X$ una inversa a izquierda. Para ver que es una equivalencia en $\catH$, tenemos que mostrar que hay un isomorfismo entre $id_Y$ y la composición $sr$, pero como en $\catH$ toda 2-celda es inversible, entonces sólo necesitamos probar la existencia de una homotopía $\xymatrix{ sr \ar@2{~>}[r] & id_Y}$.
    
    Dado que $rsr=r$, el diagrama
    \begin{center}
        $\xymatrix{
           Y \ar@<3pt>[rr]^{sr}\ar@<-3pt>[rr]_{id_Y}\ar[dr]_{r} &&Y\ar[rr]^{id_Y}\ar[dl]^{r}|\circ  &&Y\\
           &  X \\  
        }$
    \end{center}
    es conmutativo y, efectivamente, nos da una homotopía de $sr$ a $id_Y$. 
    
    \bigskip
    
     Como el caso en el que $s$ es una retracción es completamente análogo, podemos concluir la demostración. 
    
\end{proof}

Denotaremos $\Ho$ a la subcategoría de $\catH$ cuyos objetos y flechas son como en $\mathscr{C}_{fc}$ y cuyas 2-celdas son clases de homotopías en $\mathscr{C}$.

\begin{observation}
    Dado que toda equivalencia débil es una equivalencia en $\Ho$, entonces todo 2-funtor $\Ho \longrightarrow \mathscr{D}$ manda equivalencias débiles en equivalencias; es decir,  $Hom_{s_+}(\Ho, \mathscr{D})=Hom_s(\Ho, \mathscr{D})$ y $Hom_{p_+}(\Ho, \mathscr{D})=Hom_p(\Ho, \mathscr{D})$.
\end{observation}

\bigskip

\begin{theorem} \label{2_loc_fc}
    La inclusión $i:\mathscr{C}_{fc}\longrightarrow \Ho$ es la 2-localización, en el sentido estricto, de la subcategoría $\mathscr{C}_{fc}$ con respecto a la clase $\mathcal{W}$ (ver definici\'on \ref{2-localization}).
    
    Más aún, los 2-funtores $i^*: Hom_s(\Ho, \mathscr{D})\longrightarrow Hom_{s_+}(\mathscr{C}_{fc}, \mathscr{D})$ e $i^*: Hom_p(\Ho, \mathscr{D})\longrightarrow Hom_{p_+}(\mathscr{C}_{fc}, \mathscr{D})$ son ambos isomorfismos de 2-categorías.
\end{theorem}

\subsection{Localizaci\'on de Quillen de \texorpdfstring{$\mathscr{C}_{fc}$}{}}

Fijamos una categor\'ia de modelos $\mathscr{C}$ y consideramos $\mathscr{C}_{fc} \subseteq \mathscr{C}$. Quillen define la categor\'ia homot\'opica de $\mathscr{C}$ como la localizaci\'on (en el sentido estricto) con respecto a la clase de equivalencias d\'ebiles, denotada $\mathbf{Ho}$($\mathscr{C}$). Los objetos de esta categor\'ia coinciden con los objetos de $\mathscr{C}$ y las flechas son clases de equivalencia de morfismos en $\mathscr{C}_{fc}$ que se obtienen mediante reemplazos fibrante y cofibrante, donde las clases de equivalencia se definen por la relaci\'on de homotop\'ia. Quillen demuestra que $\mathbf{Ho}$($\mathscr{C}$) es \emph{equivalente} a la categor\'ia de objetos fibrantes-cofibrantes $\mathscr{C}_{fc}$ cocientada por la misma relaci\'on en los morfismos, denotada $\pi\mathscr{C}_{fc}$. Cuando la categor\'ia que se quiere localizar es $\mathscr{C}_{fc}$, dicha localizaci\'on $\mathbf{Ho}$($\mathscr{C}_{fc}$) es isomorfa a $\pi\mathscr{C}_{fc}$.

En lo que sigue veremos cómo obtener la categor\'ia homot\'opica de $\mathscr{C}_{fc}$ a partir de la 2-categor\'ia $\Ho$.
Comenzamos introduciendo el concepto de \emph{morfismo de cilindros}, y probaremos que si dos homotop\'ias se conectan por un morfismo de estos, entonces ambas definen la misma 2-celda en $\catH$, lo que facilitar\'a, luego, algunas demostraciones.

\bigskip

\begin{definition}
    Sean $C=(W, Z, d_0, d_1, s, x)$, $C'=(W', Z', d'_0, d'_1, s', x')$ dos cilindros para $X$ en $\mathscr{C}$. Un \emph{morfismo de cilindros} $C\longrightarrow C'$ consiste de un par de morfismos $\phi: W\longrightarrow W'$ y $\psi: Z \longrightarrow Z'$ que hacen conmutativo el diagrama
    
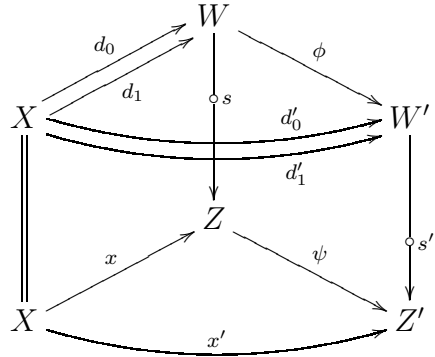
\begin{center}
    $\xymatrix{
        && W \ar[dd]^(.4){s}|(.4)\circ
           \ar[drr]^{\phi} \\      
        X \ar@<3pt>[urr]^{d_0}
          \ar@<-3pt>[urr]_{d_1}
          \ar@/_1pc/@<3pt>[rrrr]^(.7){d'_0}
          \ar@/_1pc/@<-3pt>[rrrr]_(.7){d'_1}
          \ar@{=}[dd] &&&& W'\ar[dd]^(.6){s'}|(.6)\circ \\
        && Z\ar[drr]^{\psi}\\
        X\ar[urr]^{x} \ar@/_1pc/[rrrr]^{x'} &&&& Z'}$
\end{center}
\end{definition}

\vspace{5mm}

\newcommand \germ{\underset{g}{\thicksim}}

\begin{definition} \label{rel_germ}
    La \emph{relación de gérmenes} entre homotopías es la relación de equivalencia generada por los morfismos de cilindros:
    dadas $H=(C, h)$ y $H'=(C', h')$ dos homotopías de $f$ a $g$, decimos que $H\germ H'$ si existe un morfismo de cilindros $\xymatrix {C \ar[r]^{(\phi, \psi)} & C'}$ tal que $h' \circ \phi = h$.
   
\end{definition}

\bigskip

\begin{lemma}
    Si $H$, $H':\xymatrix{f \ar@2{~>}[r] & g}$ son dos homotopías tales que $H \germ H'$, entonces $[H]=[H']$.
\end{lemma}

\begin{proof}
    Escribimos $H=(C, h)$ y $H'=(C', h')$, $C=(W, Z, d_0, d_1, s, x)$ y $C'=(W', Z', d'_0, d'_1, s', x')$. Sea $\xymatrix{C \ar[r]^{(\phi, \psi)} & C'}$ un morfismo de cilindros tal que $h'\phi =h$.
    
    Si $\mathscr{D}$ es una 2-categoría y $F:\mathscr{C}\longrightarrow \mathscr{D}$ es un funtor que manda equivalencias débiles en equivalencias, como $F\psi Fs=Fs'F\phi$, entonces se tiene $Fx=F\psi Fx=F\psi Fs \widehat{FC}=Fs'F\phi\widehat{FC}$, y $F\phi\widehat{FC}: Fd'_0 \Longrightarrow Fd'_1$. Por unicidad de $\widehat{FC'}$, resulta $F\phi\widehat{FC}=\widehat{FC'}$, y luego, de la ecuación $h\phi=h'$ obtenemos $\widehat{FH'}=Fh'F\psi \widehat{FC}=Fh\widehat{FC}=\widehat{FH}.$
    
\end{proof}

\begin{lemma} \label{lema_3_pasos}
    Sean $X$ e $Y$ ambos fibrantes y cofibrantes, $f, g: X \longrightarrow Y$ y $H$ una homotopía de $f$ a $g$. Entonces existe $H'$, una $q$-homotopía en $\mathscr{C}_{fc}$ tal que $[H]=[H']$.
\end{lemma}

\begin{proof}
    Podemos suponer que $H$ es una homotopía fibrante (ver definici\'on \ref{fibrante_hpy}). En efecto, si $H$ es de la forma
    \begin{center}
        $\xymatrix{
           X \ar@<3pt>[rr]^{d_0}\ar@<-3pt>[rr]_{d_1}\ar[dr]_{x} && W\ar[rr]^{h}\ar[dl]^{s}|\circ  && Y\\
           &  Z \\  
        }$
    \end{center}
    consideramos una factorización de $s$ dada por
    $\xymatrix{ W \ar[rr]^{s}| \circ \ar@{>->}[dr]_{j}|\circ && Z\\
                & \widetilde{W} \ar@{->>}[ur]_{\Tilde{s}} |\circ\\ }$,
    y un morfismo $\Tilde{h}:\widetilde{W}\longrightarrow Y$ haciendo conmutar el diagrama 
    $\xymatrix{ W\ar[r]^{h} \ar@{ >->}[d]_(.4){j}|(.4)\circ & Y \ar@{->>}[d] \\
                \widetilde{W} \ar@{-->}[ur]^{\Tilde{h}} \ar[r] & 1 \\ }$.
                
    Tomando $\Tilde{d}_0:=jd_0$ y $\Tilde{d}_1:=jd_1$ queda definida una homotopía de $f$ a $g$ dada por $\xymatrix{
           X \ar@<3pt>[rr]^{\Tilde{d}_0}\ar@<-3pt>[rr]_{\Tilde{d}_0}\ar[dr]_{x} && \widetilde{W}\ar[rr]^{\Tilde{h}}\ar@{->>}[dl]^{\Tilde{s}}|\circ  && Y\\
           &  Z \\  
        }$, que es fibrante y está en la misma clase que $H$ ya que $(j, id_Z)$ es un morfismo de cilindros tal que $\Tilde{h}j=h$.

    \bigskip   
    Suponiendo entonces que $s$ es una fibración trivial, tomando el pullback de $x$ y $s$ obtenemos
    
    \begin{center}
        $
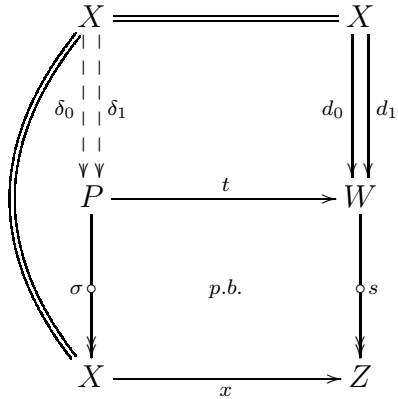
\xymatrix{
            X \ar@{=}[rrr] \ar@{-->}@<3pt>[dd]^{\delta_1} \ar@{-->}@<-3pt>[dd]_{\delta_0} \ar@/_2.5pc/@{=}[dddd] &&& X \ar@<3pt>[dd]^{d_1} \ar@<-3pt>[dd]_{d_0} \\
            \\
            P \ar[rrr]^{t} \ar@{} [ddrrr]|{p.b.} \ar@{->>}[dd]_{\sigma}|\circ &&& W \ar@{->>}[dd]^{s}|\circ \\
            \\
            X \ar[rrr]_{x} &&& Z\\
        }$
    \end{center}
    
    donde $\sigma$ es también una fibración trivial por ser el pullback de $s$ que es un morfismo en la misma clase, y la existencia de $\delta_0$ y $\delta_1$ se debe a la propiedad universal.
    
    Luego, $(P, X, \delta_0, \delta_1, \sigma, id_X)$ es un cilindro para $X$, y con $k:=ht: P \longrightarrow Y$ tenemos una homotopía de $f$ a $g$ \hspace{2mm} $
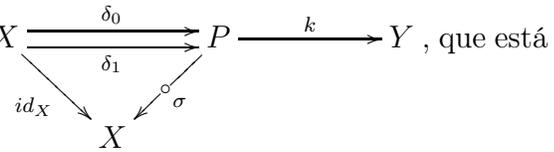
\xymatrix{
           X \ar@<3pt>[rr]^{\delta_0}\ar@<-3pt>[rr]_{\delta_1}\ar[dr]_{id_X} && P\ar[rr]^{k}\ar[dl]^{\sigma}|\circ  && Y\\
           &  X \\  
        }$,
    que está en la clase de $H$ dado que $(t, x)$ es un morfismo de cilindros satisfaciendo $ht=k$. Notamos que $P$ es fibrante porque $X$ lo es y, luego, la composición $\xymatrix{ P \ar@{->>}[r]^(.4){\sigma}|\circ & X \ar@{->>}[r] & 1 }$ es una fibración.
    
    Como $\binom{\sigma_0}{\sigma_1}$ no necesariamente es una cofibración, el diagrama anterior no necesariamente es una $q$-homotopía. Consideramos, entonces, la siguiente factorización $\xymatrix{ X\coprod X \ar[rr]^{\binom{\sigma_0}{\sigma_1}} \ar@{>->}[dr]_{\binom{d'_0}{d'_1}} && P\\
                            & W' \ar@{->>}[ur]_{p} |\circ\\ }$,
    junto con $s':=\sigma p: W'\longrightarrow X$ y $h':=kp: W'\longrightarrow Y$, y definimos $H'$ como
    \begin{center}
        $\xymatrix{
           X \ar@<3pt>[rr]^{d'_0}\ar@<-3pt>[rr]_{d'_1}\ar[dr]_{x} && W'\ar[rr]^{h'}\ar[dl]^{s'}|\circ  && Y\\
           &  X \\  
        }$
    \end{center}
    Dado que $h'd'_0=kpd'_0=k\delta_0=f$ y $h'd'_1=k\delta_1=g$, entonces $H'$ es una $q$-homotopía de $f$ a $g$ que está en la misma clase que $H$.
    
    Además, como $P$ es fibrante y $p$ es una fibración, $W'$ también es fibrante, y de la composición
    $\xymatrix{ 0 \hspace{2mm} \ar@{>->}[r] &X \hspace{1mm} \ar@{>->}[r]^(.4){i_0} & X\coprod X \hspace{1mm} \ar@{>->}[r]^(.6){\binom{d'_0}{d'_1}} & W' }$
    se deduce que $W'$ es cofibrante; por lo que $H'$ es una $q$-homotopía en $\mathscr{C}_{fc}$.
    
\end{proof}

\begin{lemma}
    Dadas $f, g, l: X \longrightarrow Y$ en $\mathscr{C}_{fc}$ y dos homotopías componibles $H: \xymatrix{ f \ar@2{~>}[r] & g}$ y $H': \xymatrix{ g \ar@2{~>}[r] & l}$, existe $H'': \xymatrix{f \ar@2{~>}[r] & l}$ tal que $[H'']=[H', H]$.
\end{lemma}

\begin{proof}
    Como $X$ e $Y$ son fibrantes y cofibrantes, por el lema anterior podemos suponer que tanto $H$ como $H'$ son $q$-homotopías. Si $H=(C, h)$ y $H'=(C', h')$, consideremos la homotopía $H''=(C'', h'')$ del lema \ref{comp_v} y veamos que $\widehat{FH''}=\widehat{FH'}\circ \widehat{FH}$ para todo 2-funtor $F: \mathscr{C} \longrightarrow \mathscr{D}$ que manda equivalencias débiles en equivalencias.

    La homotopía $H''$ queda determinada por el diagrama conmutativo
    
    \begin{center}
    $\xymatrix{
            & W \ar[dr]_{\alpha}
                 \ar@/^/[drrr]^{s}|\circ \ar@/^2pc/[rrrrd]^{h}\\
        X \hspace{2mm} \ar@<3pt>[ur]^{d_0}
          \ar@<-3pt>[ur]_{d_1}
          \ar@<3pt>[dr]^{d'_0}
          \ar@<-3pt>[dr]_{d'_1}  && W'' \ar[rr]^{s''}|\circ \ar@/_1pc/[rrr]_(.6){h''} && X & Y\\
            & W' \ar[ur]^{\beta}
                 \ar@/_/[urrr]_{s'}|\circ \ar@/_2pc/[rrrru]_{h'}\\}$
    
\end{center}

\vspace{3mm}
Como $\widehat{FH'}\circ \widehat{FH}=Fh''F\beta\widehat{FC'} \circ Fh''F\alpha \widehat{FC}= Fh'' (F\beta \widehat{FC'} \circ F\alpha \widehat{FC})$, basta ver que $F\beta \widehat{FC'} \circ F\alpha \widehat{FC}=\widehat{FC''}$. De las ecuaciones $s=s''\alpha$ y $s'=s''\beta$ se deduce que $id_{FX}= id_{FX} \circ id_{FX}= Fs''F\alpha \widehat{FC} \circ Fs'' F\beta \widehat{FC'}=Fs''(F\beta \widehat{FC'} \circ F\alpha \widehat{FC})$ y, luego, $F\alpha \widehat{FC} \circ F\beta \widehat{FC'}= \widehat{FC''}$.

\end{proof}

Con lo visto hasta el momento podemos asegurar que  cuando nos restringimos a la subcategoría $\mathscr{C}_{fc}$ existe una correspondencia entre las clases de $q$-homotopías y las clases de secuencias finitas de homotopías componibles, y se tiene de esta manera la siguiente proposici\'on

\begin{proposition} \label{categoria_Ho}
    La 2-categor\'ia $\Ho$ es aquella cuyos objetos y morfismos son como en $\mathscr{C}_{fc}$ y cuyas 2-celdas son las clases de $q$-homotop\'ias, donde dos $q$-homotop\'ias se identifican conforme a la relaci\'on de equivalencia definida en \ref{rel_adhoc}.
\end{proposition}

\begin{comentario}
    La relaci\'on de g\'ermenes (\ref{rel_germ}) permitir\'ia definir otra 2-categor\'ia homot\'opica. Pueden definirse las composiciones vertical y horizontal, y demostrarse todos los requisitos, salvo el axioma que relaciona ambas composiciones.

    Quillen tambi\'en define una relaci\'on de equivalencia entre homotop\'ias (ver \cite{Qui} Ch.I $\S$2), y demuestra que las clases de homotop\'ias pueden componerse verticalmente y horizontalmente, pero no hace menci\'on de la compatibilidad entre ambas composiciones, lo que sugiere que dicha compatibilidad no puede demostrarse o no ser\'ia v\'alida. De hecho, la relaci\'on entre homotop\'ias definida por Quillen es equivalente a la relaci\'on de g\'ermenes.
\end{comentario}

\vspace{5mm}
Obtendremos la categoría homotópica de $\mathscr{C}_{fc}$ aplicando el funtor de componentes conexas $\pi _0 : 2$-$Cat \longrightarrow Cat$, definido como:

\begin{enumerate}
    \item Por cada 2-categoría $\mathscr{D}$, $\pi_0(\mathscr{D})$  es una categoría cuyos objetos son los de $\mathscr{D}$ y, por cada par de objetos $X$, $Y$ en $\mathscr{D}$, $\pi_0(\mathscr{D})[X, Y]=\small{\faktor{\mathscr{D}[X, Y]}{\equiv}\text{ }}$, donde `` $\equiv$ '' es la clausura transitiva de la relaci\'on
    \begin{center}
    $f \thicksim g \Leftrightarrow$ existe una 2-celda $f \Longrightarrow g$ o existe una 2-celda $g \Longrightarrow f$.
    \end{center}
    \item Dado un 2-funtor $G: \mathscr{D}_1 \longrightarrow \mathscr{D}_2$, $\pi_0(G): \pi_0(\mathscr{D}_1) \longrightarrow \pi_0(\mathscr{D}_2)$ es $\pi_0(G)X=GX$ para todo $X$ en $\mathscr{D}_1$ y $\pi_0(G)[f]=[Gf]$ para toda $f$ en $\mathscr{D}_1$.
\end{enumerate}

\bigskip
\begin{observation} \label{pi_0_functor}
    Sea $d: Cat \longrightarrow 2$-$Cat$ el funtor que asocia cada categoría $\mathscr{X}$ consigo misma vista como 2-categoría discreta. Entonces, $\pi_0$ es adjunto a izquierda de $d$: 
    
    Para cada 2-categoría $\mathscr{D}$, existe un 2-funtor $\theta_{\mathscr{D}}: \mathscr{D} \longrightarrow \pi_0 \mathscr{D}$ tal que para toda categoría $\mathscr{X}$ y para todo 2-funtor $F: \mathscr{D} \longrightarrow \mathscr{X}$ existe un único funtor $\widetilde{F}: \pi_0 \mathscr{D} \longrightarrow \mathscr{X}$ satisfaciendo $\widetilde{F}\theta_{\mathscr{D}}=F$
    
    \begin{center}
        $\xymatrix{
            \mathscr{D} \ar[rr]^{\theta_{\mathscr{D}}} \ar[d]_{\forall F} && \pi_0 \mathscr{D}. \ar@{-->}[dll]^(.4){\exists ! \widetilde{F}} \\
            \mathscr{X}\\
        }$
    \end{center}
    
    El 2-funtor $\theta_{\mathscr{D}}$ se define de la siguiente manera. En los objetos, $\theta_{\mathscr{D}}(X)=X$ para todo $X$ en $\mathscr{D}$; en las flechas, $\theta_{\mathscr{D}}(f)=[f]$ para toda $f$ en $\mathscr{D}$; y, adem\'as, $\theta_{\mathscr{D}}$ manda toda 2-celda de $\mathscr{D}$ en la 2-celda trivial correspondiente. Queda definida, as\'i, una transformaci\'on $\theta: Id_{2-Cat} \Longrightarrow d\pi_0$ natural en la variable $\mathscr{D}$.
    
    \vspace{3mm}
    Dado que esto quiere decir que $\pi_0 : 2$-$Cat \longrightarrow Cat$ es adjunto a izquierda del funtor $d$, ambas categor\'ias $\mathscr{X}$ y $\pi_0\mathscr{D}$ son interpretadas como 2-categor\'ias; es decir, el diagrama anterior es un diagrama en $2$-$Cat$, y tenemos $d(\mathscr{X})$ en lugar de $\mathscr{X}$ y $d(\pi_0\mathscr{D})$ en lugar de $\pi_0\mathscr{D}$, pero hacemos aqu\'i un abuso de notaci\'on.
    
    Adem\'as, en adelante denotaremos $\pi_0$ en lugar de $\theta_{\mathscr{D}}$.
\end{observation}    

    \vspace{3mm}
    Para cada 2-categor\'ia $\mathscr{D}$, el 2-funtor $\pi_0:\mathscr{D} \longrightarrow \pi_0 \mathscr{D}$ induce un isomorfismo en las categor\'ias de funtores:
    
\begin{proposition} \label{pi0_iso}
    El funtor $\pi_0^*: Hom(\pi_0\mathscr{D}, \mathscr{X}) \longrightarrow Hom(\mathscr{D}, \mathscr{X})$ es un isomorfismo de categor\'ias.
\end{proposition} 

\begin{proof}
    S\'olo resta probar que es plenamente fiel.
    
    Dados 2-funtores $F, G: \mathscr{D} \longrightarrow \mathscr{X}$ y $\eta: F \Longrightarrow G$ una transformación 2-natural, definimos $\Tilde{\eta}: \widetilde{F} \Longrightarrow \widetilde{G}$ como $\Tilde{\eta}_X=\eta_X$. La buena definición de $\Tilde{\eta}$ se debe a que $\pi_0$ es la identidad en los objetos, y entonces $FX=\widetilde{F}X$, $GX=\widetilde{G}X$ para todo $X$. Por otro lado, como adem\'as $\widetilde{F}[f]=Ff$ para toda $f$, de la naturalidad de $\eta$ obtenemos que $\Tilde{\eta}$ es una transformación natural, y $\Tilde{\eta}\pi_0 = \eta$.
    
\end{proof}

\bigskip
\begin{proposition} \label{Quillen_loc}
    Dados los 2-funtores $i: \mathscr{C}_{fc} \longrightarrow \Ho$ del teorema \ref{2_loc_fc} y $\pi_0: \Ho \longrightarrow \pi_0(\Ho)$ de la observaci\'on \ref{pi_0_functor}, el funtor definido por la composición $\pi_0i:\mathscr{C}_{fc}\longrightarrow \pi_0(\Ho)$ es la localización de $\mathscr{C}_{fc}$ con respecto a la clase $\mathcal{W}$.
    
    Más aún, la precomposición $$(\pi_0i)^*:Hom_+(\pi_0(\Ho), \mathscr{X})\longrightarrow Hom(\mathscr{C}_{fc}, \mathscr{X})$$ es un isomorfismo de categorías, para toda categoría $\mathscr{X}$.
\end{proposition}

\begin{proof}
    El funtor $\pi_0i$ manda equivalencias débiles en isomorfismos ya que $i(\mathcal{W})\subseteq Equiv(\Ho)$ y, por cómo está definido $\pi_0$ en las 2-celdas, este manda equivalencias en isomorfismos.
    
    Por otro lado, como $i^*$ es un isomorfismo por \ref{hom_iso} y $\pi_0^*$ es un isomorfismo por \ref{pi0_iso}
    \begin{center}
        $\xymatrix{
            \mathscr{C}_{fc} \ar@{^{(}->}[rr]^{i} \ar[ddr]_(.4){F} && \Ho \ar[rr]^{\pi_0} \ar@{-->}[ddl]^(.4){\exists !\widehat{F}} && \pi_0(\Ho), \ar@{-->}[ddlll]^(.4){\exists !\Tilde{F}}\\
            \\
            &\mathscr{X}\\}$
        
    \end{center}
    entonces la composici\'on $i^*\pi_0^*$ es tambi\'en un isomorfismo de categor\'ias.
\end{proof}

\newpage

\section{La 2-localizaci\'on de la categor\'ia \texorpdfstring{$\mathscr{C}$}{}}

Queremos definir ahora un funtor $\mathscr{C} \longrightarrow \mathscr{C}_{fc}$ y tomar la composición con $i: \mathscr{C}_{fc} \longrightarrow \Ho$. Probaremos que este 2-funtor, que llamaremos $q:\mathscr{C} \longrightarrow \Ho $, es la 2-localización de la categoría $\mathscr{C}$ respecto de la clase $\mathcal{W}$, que es nuestro principal objetivo. Para esto vamos a considerar las subcategorías plenas $\mathscr{C}_f$ y $\mathscr{C}_c$ de objetos fibrantes y cofibrantes, respectivamente, y dos asignaciones $R: \mathscr{C} \longrightarrow \mathscr{C}_f$ y $Q: \mathscr{C} \longrightarrow \mathscr{C}_c$ que pueden construirse a partir de los axiomas de categorías de modelos de la definición \ref{model_cat},
pero que \emph{no necesariamente} son funtoriales; en ese caso, la composición $q$ no será exactamente un 2-funtor y en consecuencia no podrá determinar la localizaci\'on buscada. Por esta razón, vamos a pedir que la factorización del axioma M4 sea funtorial, y si bien esto modifica la definición original de categoría de modelos que hemos dado, gran parte de los ejemplos conocidos cumplen esta axiomática (por ejemplo, todas las categor\'ias de modelos \emph{cofibrantemente generadas} - ver \ref{comentario}), que tambi\'en es ampliamente utilizada en la literatura. Asumiendo que la estructura de modelos de $\mathscr{C}$ admite una factorización funtorial, las flechas $\xymatrix{ \mathscr{C} \ar[r]^{Q} &\mathscr{C}_c \ar[r]^{R} & \mathscr{C}_{fc} }$ determinar\'an efectivamente un funtor y, junto con lo que vimos en la secci\'on 3 para la subcategoría $\mathscr{C}_{fc}$, nos permitirá concluir el teorema de la localización.

\subsection{Reemplazos fibrante y cofibrante} \label{fibrant-cofibrant}

\begin{definition}
    Un \emph{sistema de factorización débil} en una categoría $\mathscr{X}$ es un par $(\mathcal{L}, \mathcal{R})$ de clases distinguidas de morfismos tales que 
    \begin{enumerate}
        \item Todo morfismo $h$ en $\mathscr{X}$ puede ser factorizado como $h=gf$, con $f \in \mathcal{L}$ y $g \in \mathcal{R}$.
        \item $\mathcal{L}$ es precisamente la clase de morfismos que tienen la propiedad de levantamiento a izquierda con respecto a todo morfismo de $\mathcal{R}$.
        
            $\mathcal{R}$ es la clase de morfismos que tienen la propiedad de levantamiento a derecha con respecto a todo morfismo en $\mathcal{L}$.
    \end{enumerate}
\end{definition}

\vspace{4mm}

\begin{examples} 
    En las categorías $Set$, $Grp$ y $R$-$Mod$, donde $R$ es un anillo con unidad, las clases $\mathcal{L}= monomorfismos$ y $\mathcal{R}=epimorfismos$ forman un sistema de factorización débil.
 
    Si $\mathscr{C}$ es una categoría de modelos, las clases $(\mathcal{C} \cap \mathcal{W}, \mathcal{F})$ son un sistema de factorización débil, asi como también lo son $(\mathcal{C}, \mathcal{F} \cap \mathcal{W})$.

\end{examples}
\vspace{6mm}

\begin{definition} \label{funtorial}
    Decimos que una factorización débil $(\mathcal{L}, \mathcal{R})$ es \emph{funtorial} si cada vez que tenemos un diagrama conmutativo
    \begin{center}
            $
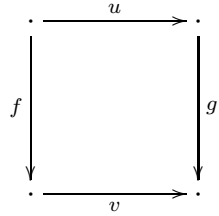
\xymatrix{
                \cdot \ar[rr]^{u} \ar[dd]_{f} && \cdot \ar[dd]^{g}\\
                \\
                \cdot \ar[rr]_{v} && \cdot \\
            }$
    \end{center}
    existen morfismos $(\lambda_f, \rho_f)$, $(\lambda_g, \rho_g)$ y un morfismo $F(u, v)$ haciendo conmutar el diagrama
    \begin{center}
            $
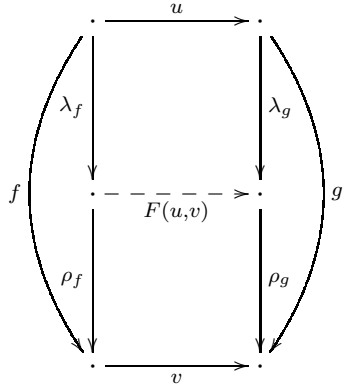
\xymatrix{
                \cdot \ar[rr]^{u} \ar[dd]_{\lambda_f} \ar@/_2pc/[dddd]_{f} && \cdot \ar@/^2pc/[dddd]^{g} \ar[dd]^{\lambda_g}\\
                \\
                \cdot \ar@{-->}[rr]_{F(u, v)} \ar[dd]_{\rho_f} && \cdot \ar[dd]^{\rho_g} \\
                \\
                \cdot \ar[rr]_{v} && \cdot \\
            }$
    \end{center}
    donde $\lambda_f$, $\lambda_g \in \mathcal{L}$, $\rho_f$, $\rho_g \in \mathcal{R}$,
    y $F(u, v)$ depende funtorialmente de $u$ y $v$; es decir, $F(u\circ u', v\circ v')= F(u, v) \circ F(u', v')$ y si $f=g$ entonces $F(id, id)=id$.
\end{definition}

\begin{observation}
    Denotamos $\overrightarrow{\mathscr{C}}$ a la categoría cuyos objetos son los morfismos de $\mathscr{C}$ y una flecha de $f$ a $g$ en $\overrightarrow{\mathscr{C}}$ es un par $(u, v)$ de morfismos de $\mathscr{C}$ tales que $gu=vf$. Sean $dom, codom: \overrightarrow{\mathscr{C}} \longrightarrow \mathscr{C}$ los funtores que proyectan a dominio y codominio, respectivamente. La definición anterior nos dice precisamente que un sistema $(\mathcal{L}, \mathcal{R})$ es funtorial si existen un funtor $F:\overrightarrow{\mathscr{C}} \longrightarrow \mathscr{C}$ y transformaciones naturales $\lambda: dom \longrightarrow F$ y $\rho: F \longrightarrow codom$ tales que para toda $f$ en $\mathscr{C}$ se tiene 
    
    \begin{center}
        $\xymatrix{
            dom(f) \ar[rr]^{f} \ar[dr]_{\lambda_f} && codom(f)\\
            & F(f) \ar[ur]_{\rho_f}
        }$
    \end{center}
    con $\lambda_f \in \mathcal{L}$, $\rho_f \in \mathcal{R}$.
    
    Decimos que $(F, \lambda, \rho)$ es una \emph{realización funtorial (functorial realization)} para la factorización débil $(\mathcal{L}, \mathcal{R})$. Ver \cite{Riehl}.
\end{observation}

\bigskip
En adelante trabajaremos sobre una categoría de modelos $\mathscr{C}$ en la que las factorizaciones del axioma M4 pueden elegirse funtorialmente en el sentido de la definici\'on \ref{funtorial}.

\begin{comentario} \label{comentario}
    La mayoría de las estructuras de modelos conocidas son \emph{cofibrantemente generadas}, es decir que la clase $\mathcal{R}$ de la factorización débil es definida como aquellos morfismos que tienen la propiedad de levantamiento a derecha con respecto a cierto \emph{conjunto} de flechas (y, consecuentemente, la clase $\mathcal{L}$ consiste precisamente de aquellos morfismos que tienen la propiedad de levantamiento a izquierda respecto de la clase $\mathcal{R}$), donde además dicho conjunto permite el \emph{argumento del objeto pequeño (small object argument)}, que es un proceso de factorización asociado a este conjunto, y es la principal herramienta para producir factorizaciones funtoriales. Las categor\'ias $\mathcal{T}op$, $\mathcal{SS}et$, $Ch_A$ y $\mathbf{Cat}$ que mencionamos en la secci\'on 3 son ejemplos de estructuras cofibrantemente generadas. Podemos encontrar en \cite{Hov} una descripci\'on m\'as precisa de este tipo de estructuras junto con algunos ejemplos que tambi\'en son presentados con todo detalle.

\end{comentario}

\vspace{1mm}
\begin{definition} 
    Sea $\mathscr{C}_c \subseteq \mathscr{C}$ la subcategoría de objetos cofibrantes. Definimos un funtor $Q: \mathscr{C} \longrightarrow \mathscr{C}_c$ de la siguiente forma:
    \begin{enumerate}
        \item Para cada $X$ en $\mathscr{C}$, factorizamos $0 \longrightarrow X$ obteniendo un objeto cofibrante $QX$ y una fibración trivial $p_X: QX \longrightarrow X$.
        
        Si $F:\overrightarrow{\mathscr{C}} \longrightarrow \mathscr{C}$ es el funtor de la realización funtorial asocida a esta factorización, entonces $QX=F(0 \rightarrow X)$.
        \item Dada $f:X \longrightarrow Y$, se define $Qf:QX \longrightarrow QY$ cumpliendo la ecuación $p_{Y}Qf=fp_{X}$ como vemos en el diagrama 
        \begin{center}
            $\xymatrix{
            0 \ar@{ >->}[rr] \ar@{ >->}[dd] && QY \ar@{->>}[dd]^{p_Y}|\circ \\
            \\
            QX \ar@{-->}[uurr]|{Qf} \ar@{->>}[r]_{p_X}|\circ & X \ar[r]_{f} & Y\\
            }$
        \end{center}
       
       Notamos que $Qf= F(id_0, f)$. Además, por el axioma M5, si $f \in \mathcal{W}$ entonces $Qf \in \mathcal{W}$.
    \end{enumerate}
    
    \bigskip
    De la funtorialidad de $F$ se deduce que $Q$ también es un funtor: 
    
    \begin{center}
        $\xymatrix{
            0 \ar[rr] \ar[d] && 0\ar[d]\\
                QX \ar[rr]|{id_{QX}} \ar@{->>}[d]_{p_X}|\circ && QX \ar@{->>}[d]^{p_X}|\circ  & \hspace{3mm}\text{y}\\
                X \ar[rr]_{id_X} && Y\\
        }$
        \hspace{6mm}
        $\xymatrix{
            0\ar[rr] \ar[d] && 0 \ar[rr] \ar[d] && 0 \ar[d] \\
            QX \ar@/^1.5pc/[rrrr]|(.4){Q(gf)} \ar[rr]_{Qf} \ar@{->>}[d]_{p_X}|\circ && QY \ar[rr]_{Qg} \ar@{->>}[d]_{p_Y}|\circ && QZ \ar@{->>}[d]^{p_Z}|\circ \\
            X \ar[rr]_{f} && Y\ar[rr]_{g} && Z\\
        }$
    \end{center}

\vspace{5mm}

Dualmente, obtenemos también un funtor $R:\mathscr{C} \longrightarrow \mathscr{C}_f$ con $RX$ un objeto fibrante y una cofibración trivial $i_X: X \longrightarrow RX$ factorizando el morfismo $X \longrightarrow 1$, para todo $X$; y una flecha $Rf: RX \longrightarrow RY$ satisfaciendo $Rfi_X=i_Yf$, para cada $f:X \longrightarrow Y$ en $\mathscr{C}$. Además, si $f$ es una equivalencia débil, entonces $Rf$ también.

Los funtores $Q$ y $R$ se conocen como \emph{reemplazo cofibrante (cofibrant replacement)} y \emph{reemplazo fibrante (fibrant replacement)}, respectivamente.
\end{definition}

\begin{observation}
    Pensando a $Q$ y a $R$ como funtores de $\mathscr{C}$ en $\mathscr{C}$ se obtienen transformaciones naturales $p:Q \Longrightarrow Id$ e $i: Id \Longrightarrow R$ definidas por $p_X$ e $i_X$, respectivamente, para cada $X$ en $\mathscr{C}$:
    de las definiciones de $Q$ y $R$ en los morfismos se tienen las ecuaciones que demuestran la naturalidad de $p$ e $i$. 
    \\
    
    Notamos que, por definici\'on, $p_X$ e $i_X$ son equivalencias d\'ebiles.
\end{observation}

\bigskip

\begin{definition} \label{funtor_q}
    Restringiendo $R$ a la subcategoría $\mathscr{C}_c$ y precomponiendo con $Q$ tenemos un funtor $RQ: \mathscr{C} \longrightarrow \mathscr{C}_{fc}$. Definimos $q:\mathscr{C} \longrightarrow \Ho$ como la composición

    \begin{center}
        $\xymatrix{
        \mathscr{C} \ar[rr]^{Q}\ar@/^2.5pc/[rrrrrr]^{q} && \mathscr{C}_c \ar[rr]^{R} && \mathscr{C}_{fc} \ar@{^{(}->}[rr]^{i} && \Ho
        }$
    \end{center}
\end{definition}

\vspace{3mm}
\begin{Observation}
    Si $s$ es una equivalencia débil en $\mathscr{C}$, $RQs$ es una equivalencia débil en $\mathscr{C}_{fc}$ y, luego, resulta una equivalencia en $\Ho$, por lo que $q$ manda la clase $\mathcal{W}$ en equivalencias.

\end{Observation}
    
\subsection{El teorema de 2-localizaci\'on}

Queremos ver que tenemos una pseudoequivalencia de 2-categorías
\begin{center}
    $\xymatrix{ Hom_p(\Ho, \mathscr{D}) \ar[rr]^{q^*} && Hom_{p_+}(\mathscr{C}, \mathscr{D}), }$
\end{center}
para toda 2-categoría $\mathscr{D}$.

Recordemos que el funtor $i: \mathscr{C}_{fc} \longrightarrow \Ho$ se obtiene restringiendo $i:\mathscr{C} \longrightarrow \catH$ como es indicado en el siguiente diagrama

\begin{center}
    $\xymatrix{
        \mathscr{C}_{fc} \ar[r]^{i} \ar@{^{(}->}[d] & \Ho \ar@{^{(}->}[d]\\
        \mathscr{C} \ar[r]^{i} & \catH. \\
    }$
\end{center}

Si bien esta última inclusión no manda equivalencias débiles en equivalencias, sí cumple la propiedad universal de la proposicion \ref{pu_i}:

\begin{center}
    $\xymatrix{
        \ar @{} [drr] |(.45){\equiv} 
        \mathscr{C} \hspace{1mm} \ar@{^{(}->}[rr]^{i} \ar[dr]_{F(\mathcal{W})\subseteq Equiv(\mathscr{D}) \hspace{8mm} F} && \catH \ar@{-->}[dl]^(.5){\exists !\hspace{1mm}\bar{F}} & \hspace{2mm}\catH \ar@{-->}[dl]^{\bar{q}}\\
        & \hspace{2mm}\mathscr{D} &\Ho \\
    }$
\end{center}.

Tomando $\mathscr{D}=\Ho$ y $F=q$, existe entonces un único 2-funtor $\bar{q}$ tal que $\bar{q}i=q$.

\begin{proposition}
    Dado $F: \Ho \longrightarrow \mathscr{D}$, el 2-funtor $F\bar{q}$ manda la clase $\mathcal{W}$ en equivalencias.
\end{proposition}

\begin{proof}
    Como $\bar{q}i=q$ y $q$ manda equivalencias débiles en equivalencias, entonces $\bar{q}$ también. Luego, necesariamente $F\bar{q}(\mathcal{W}) \subseteq Equiv(\mathscr{D}).$
    
\end{proof}

Fijemos una 2-categoría $\mathscr{D}$. Tenemos el 2-funtor dado por la precomposición
\begin{center}
    $\bar{q}^*: Hom_p(\Ho, \mathscr{D}) \longrightarrow Hom_{p_+}(\catH, \mathscr{D})$.
\end{center}

Como $q^*=i^*\bar{q}^*$, para probar que $q^*$ es una pseudoequivalencia será suficiente ver que tanto $\bar{q}^*$ como $i^*: Hom_{p_+}(\catH, \mathscr{D}) \longrightarrow Hom_{p_+}(\mathscr{C}, \mathscr{D})$ son ambos pseudoequivalencias de 2-categorías. En cuanto a $i^*$, ya vimos (\ref{hom_p_iso}) que, de hecho, es un isomorfismo.

Consideremos la inclusión $j: \Ho \longrightarrow \catH$ y el 2-funtor inducido $j^*: Hom_{p_+}(\catH, \mathscr{D}) \longrightarrow Hom_p(\Ho, \mathscr{D})$.

\begin{theorem} \label{pseudoequivalencia}
    Los 2-funtores $$\xymatrix{ Hom_p(\Ho, \mathscr{D}) \ar@<5pt>[rr]^{\bar{q}^*} \ar@<-5pt>@{<-}[rr]_{j^*} && Hom_{p_+}(\catH, \mathscr{D}) }$$ determinan una pseudoequivalencia de 2-categorías.
\end{theorem}

\begin{proof}
    Sea $\bar{q}^*j^*: Hom_{p_+}(\catH, \mathscr{D})\longrightarrow Hom_{p_+}(\catH, \mathscr{D})$. Definimos una transformación pseudonatural $\eta: id_{Hom_{p_+}(\catH, \mathscr{D})} \Longrightarrow \bar{q}^*j^*$ de la siguiente forma:
    
    Sea $F \in Hom_+(\catH, \mathscr{D})$. Para cada objeto $X$ en $\mathscr{C}$, sabemos que $i_X$ y $p_X$ son equivalencias débiles, y aplicando $F$ obtenemos equivalencias $\xymatrix{ FRQX \ar@{<-}[r]^{Fi_{QX}} & FQX \ar[r]^{Fp_X} & FX }$. Tomando la inversa de $Fp_X$, se tiene una equivalencia
    \begin{center}
        $\xymatrix{ FX \ar@/^2pc/[rrrr]^{(\eta_F)_X} \ar@{<-}[rr]^{\sim} && FQX \ar[rr]^{\sim} && FRQX }$.
    \end{center}
    
     Definimos $\eta_F: F \Longrightarrow Fj\bar{q}$ asociando a cada $X \in Ob(\catH)$ el morfismo $(\eta_F)_X:FX \longrightarrow FRQX$. Notamos que $\eta_F$ es una transformación natural porque $p$ e $i$ lo son; es decir que para toda flecha $f: X \longrightarrow Y$ en $\catH$ se tiene $Fj\bar{q}(f)(\eta_F)_X=(\eta_F)_YFf$. Usando el lema \ref{clave}
     y con una demostración análoga a la de la proposición \ref{prop_clave}
     puede verse que $\eta_F$ es 2-natural. En particular, $\eta_F$ es una flecha en $Hom_{p_+}(\catH, \mathscr{D})$, y es adem\'as una equivalencia, ya que lo es punto a punto (ver proposici\'on \ref{prop_infinita}).
     \\
     
     A su vez, si $\eta$ es pseudonatural, entonces ser\'a una equivalencia en la correspondiente 2-categor\'ia de 2-funtores, ya que por lo anterior cada componente $\eta_F$ lo es.
     
     En los objetos de $Hom_{p_+}(\Ho, \mathscr{D})$, tenemos a $\eta$ ya definida. Ahora, dada $\sigma: F \Longrightarrow G$ una flecha en $Hom_{p_+}(\catH, \mathscr{D})$, queremos definir una modificaci\'on inversible
     $\eta_{\sigma}:q^*j^*(\sigma)\circ \eta_F \longrightarrow \eta_G \circ \sigma$
     \begin{center}
         $\xymatrix@R=1.5pc@C=1.5pc{
            F \ar@{=>}[rr]^{\eta_F} \ar@{=>}[dd]_{\sigma} && Fqj \ar@{=>}[dd]^{q^*j^*(\sigma)} \ar@{}[ddll]|{\rotatebox[origin=c]{45}{$\leftarrow$} \eta_{\sigma}}\\
            \\
            G \ar@{=>}[rr]_{\eta_G} && Gqj. \\
         }$
     \end{center}
     Definimos $\eta_{\sigma}$ punto a punto de la siguiente manera:
     
     Para cada $X$ en $\mathscr{C}$ tenemos $(\eta_F)_X=F(i_{QX})F(p_X)^{-1}$, $(\eta_G)_X=G(i_{QX})G(p_X)^{-1}$ y el siguiente diagrama 
     \begin{center}
         $\xymatrix@R=3pc@C=3pc{
            FX \ar@{-->}@/^2.3pc/[rr]^{(\eta_F)_X} \ar[d]_{\sigma_X} \ar@{}[dr]|{\rotatebox[origin=c]{-45}{$\Rightarrow$} \sigma_{p_X}} & FQX \ar[l]_{F(p_X)} \ar[r]^{F(i_{QX})} \ar[d]^(.4){\sigma_{QX}} & FRQX \ar[d]^{\sigma_{RQX}} \ar@{}[dl]|{\rotatebox[origin=c]{45}{$\Leftarrow$} \sigma_{i_{QX}}}\\
            GX \ar@{-->}@/_2.3pc/[rr]_{(\eta_G)_X}  & GQX \ar[l]^{G(p_X)} \ar[r]_{G(i_{QX})} & GRQX, \\
         }$
     \end{center}
     donde $F(p_{X})^{-1}$ y $G(p_{X})^{-1}$ denotan las respectivas cuasi-inversas de $F(p_{X})$ y $G(p_{X})$, y $\sigma_{p_X}$, $\sigma_{i_{QX}}$ son 2-celdas inversibles en $\mathscr{D}$.
     
     Tenemos $\alpha_G: id_{GQX} \Longrightarrow G(p_X)^{-1}G(p_X)$ y $\beta_F: F(p_X)F(p_X)^{-1} \Longrightarrow id_{FX} $ 2-celdas, tambi\'en inversibles, y podemos considerar las siguientes composiciones
    \begin{center}
     $\xymatrix{
        \sigma_{RQX}(\eta_F)_X \ar@{=>}[r]^(.4){\simeq}_(.4){(1)} & G(i_{QX})\sigma_{QX}F(p_X)^{-1} \ar@{=>}[r]^(.45){\simeq}_(.45){(2)} & (\eta_G)_XG(p_X)\sigma_{QX}F(p_X)^{-1}
     }$
     \end{center}
     y
     \begin{center}
         $\xymatrix{
            (\eta_G)_XG(p_X)\sigma_{QX}F(p_X)^{-1} \ar@{=>}[r]^(.5){\simeq}_(.5){(3)} & (\eta_G)_X\sigma_X F(p_X)F(p_X)^{-1} \ar@{=>}[r]^(.65){\simeq}_(.65){(4)} & (\eta_G)_X\sigma_X,
         }$
     \end{center}
     donde $(1)$ es la 2-celda $\sigma_{i_{QX}}F(p_X)^{-1},$ $(2)$ es la 2-celda $G(i_{QX})\alpha_G\sigma_{QX}F(p_X)^{-1}$,
    $(3)$ es $(\eta_G)_X\sigma_{p_X}^{-1}F(p_X)^{-1}$ y $(4)$ es $(\eta_G)_X\sigma_X\beta_F$. Notamos que $\alpha_G$ y $\beta_F$ dependen de $X$. Tomamos $\eta_{\sigma_X}$ como la composici\'on de ambas secuencias:
    \begin{equation*}
    \begin{split}
     (\eta_{\sigma})_X &=[(\eta_G)_X\sigma_X\beta_F] \circ [(\eta_G)_X\sigma_{p_X}^{-1}F(p_X)^{-1}] \circ [G(i_{QX})\alpha_G\sigma_{QX}F(p_X)^{-1}]\circ \\
     & \hspace{100mm}[(\sigma_{i_{QX}}F(p_X)^{-1}].
    \end{split}
    \end{equation*}
    Se puede ver que $\eta_{\sigma}$ es efectivamente una modificaci\'on y, con esta definici\'on,
     $\eta$ satisface los axiomas de pseudonaturalidad. Luego, $\eta$ es una equivalencia en la 2-categor\'ia de 2-funtores de $Hom_{p_+}(\catH, \mathscr{D})$ en $Hom_{p_+}(\catH, \mathscr{D}).$
    \\
    
    De la misma forma puede definirse $\theta: j^*\bar{q}^* \Longrightarrow id_{Hom(\Ho, \mathscr{D})}$ tomando $(\theta_F)_X$ como la composición 
     \begin{center}
        $\xymatrix{ FRQX \ar@/^2pc/[rrrr]^{(\theta_F)_X} \ar@{<-}[rr]^{Fi_{QX}} && FQX \ar[rr]^{Fp_X} && FX }$,
    \end{center}
    para cada $F \in Hom_p(\Ho, \mathscr{D})$, $X \in Ob(\Ho)$. Además, $\theta$ también es una equivalencia en la correspondiente 2-categor\'ia de 2-funtores y, por lo tanto, podemos concluir la demostración. 
    
\end{proof}

\begin{observation}
    El teorema anterior nos dice que la composici\'on de los reemplazos fibrante y cofibrante $QR:\mathscr{C}\longrightarrow \mathscr{C}_{fc}$ induce una pseudoequivalencia de 2-categor\'ias
    \begin{center}
        $\xymatrix@R=3pc@C=1.5pc{
        \mathscr{C} \ar[d]_{i} \ar[rr]^{QR} && \mathscr{C}_{fc} \ar[d]^{i}\\
        \catH \ar@{-->}[rr]_{\overline{QR}} && \Ho.\\
        }$
    \end{center}
\end{observation}

\vspace{6mm}
Como consecuencia del teorema \ref{pseudoequivalencia} obtenemos el resultado principal de esta tesis:

\begin{theorem} \label{2-localization}
    Dada una categoría de modelos $\mathscr{C}$, se tiene que el funtor $q: \mathscr{C} \longrightarrow \Ho$ de la definición \ref{funtor_q} es la 2-localización de $\mathscr{C}$ con respecto a la clase $\mathcal{W}$, en el sentido que manda los elementos de $\mathcal{W}$ en equivalencias y, además, el 2-funtor $q^*: Hom_p(\Ho, \mathscr{D}) \longrightarrow Hom_{p_+}(\mathscr{D}, \mathscr{D})$ es una pseudoequivalencia de 2-categorías para toda 2-categoría $\mathscr{D}$.
\end{theorem}

\bigskip
Aplicando el funtor de componentes conexas $\pi_0$ concluimos un resultado an\'alogo a la proposici\'on \ref{Quillen_loc}. M\'as precisamente, lo que se tiene es la localizaci\'on de la categor\'ia $\mathscr{C}$ con respecto a la clase de equivalencias d\'ebiles en el sentido de la definici\'on \ref{localization}. A diferencia de lo hecho en la secci\'on 4, donde pudimos obtener la categor\'ia homot\'opica de la subcategor\'ia $\mathscr{C}_{fc}$, en este caso obtenemos un funtor $\pi_0q:\mathscr{C}\longrightarrow \pi_0(\Ho)$ con la propiedad universal de la localizaci\'on pero no en el sentido estricto, por lo que no es exactamente la localizaci\'on de Quillen, sino que es equivalente a ella.

\begin{observation}
Consideremos una categor\'ia $\mathscr{X}$. El 2-funtor $q^*$ del teorema \ref{2-localization}, tomando $\mathscr{D}= \mathscr{X}$, ahora es un funtor entre las categor\'ias de 2-funtores $$q^*: Hom(\Ho, \mathscr{X}) \longrightarrow Hom_+(\mathscr{C}, \mathscr{X}),$$
ya que $Hom_p(\Ho, \mathscr{X})=Hom_s(\Ho, \mathscr{X})=Hom(\Ho, \mathscr{X})$, $Hom_{p_+}(\mathscr{C}, \mathscr{X})=Hom_+(\mathscr{C}, \mathscr{X})$, y todas las modificaciones son identidades.

Adem\'as, $q^*$ es una equivalencia de categor\'ias. Notamos que esto quiere decir que existe una cuasi-inversa para $q^*$ tal que la unidad y la counidad, digamos $\eta$ y $\theta$, son isomorfismos naturales. En principio, el teorema \ref{2-localization} afirma que $\eta$ y $\theta$ son equivalencias como transformaciones pseudonaturales, pero, de nuevo, una transformacion pseudonatural entre funtores es una transformaci\'on natural y las modificaciones son necesariamente identidades, de manera que tanto la unidad como la counidad son isomorfismos. Tenemos entonces:
\end{observation}

\begin{proposition}
    Dados los 2-funtores $q: \mathscr{C} \longrightarrow \Ho$ de la definici\'on \ref{funtor_q} y $\pi_0: \Ho \longrightarrow \pi_0(\Ho)$ de \ref{pi_0_functor}, se tiene que la composición $\pi_0q:\mathscr{C}\longrightarrow \pi_0(\Ho)$ es la localización de $\mathscr{C}$ con respecto a la clase $\mathcal{W}$ en el sentido de la definici\'on \ref{localization}.
\end{proposition}

\begin{proof}
    Como $q$ manda equivalencias d\'ebiles en equivalencias y $\pi_0$ manda equivalencias en isomorfismos, entonces la composici\'on manda la clase $\mathcal{W}$ en $Isos(\mathscr{\pi_0(\Ho)})$.
    
    Por otro lado, hay que ver que $$q^*\pi_0^*: Hom(\pi_0(\Ho), \mathscr{D}) \longrightarrow Hom_+(\mathscr{\mathscr{C}, \mathscr{D}})$$ es una equivalencia de categor\'ias para toda categor\'ia $\mathscr{D}$.
    Pero como se tiene un isomorfismo $\pi_0^*:Hom(\pi_0(\Ho), \mathscr{X})\longrightarrow Hom(\Ho, \mathscr{X})$ y adem\'as $q^*: Hom(\Ho, \mathscr{X}) \longrightarrow Hom_+(\mathscr{C}, \mathscr{X})$ es una equivalencia por la observaci\'on anterior, entonces la composici\'on define una equivalencia de categor\'ias.
\end{proof}

\subsection{Homotopías a derecha}

De lo hecho en \ref{fibrant-cofibrant} vemos que, en presencia de funtorialidad en las factorizaciones, las homotop\'ias a derecha no son necesarias, a diferencia del caso con las factorizaciones no funtoriales de los axiomas de Quillen.

La categoría $\Ho$ también puede obtenerse tomando como 2-celdas a las clases de secuencias finitas de homotopías a derecha. Generalizando el concepto de path-object como lo hicimos para los cilindros, obtenemos una versión de homotopías a derecha también más general que la de Quillen. Decimos que dos homotopías a derecha $K$ y $K'$ están en la misma clase si y sólo si $\widehat{FK}=\widehat{FK'}$ para todo 2-funtor $F: \mathscr{C} \longrightarrow \mathscr{D}$ tal que $F(\mathcal{W}) \subseteq Equiv(\mathscr{D})$, para toda 2-categoría $\mathscr{D}$. Definimos las composiciones vertical y horizontal como lo hicimos antes, construyendo de esta forma una 2-categoría que llamamos $\catH ^r$ para distinguirla de $\catH ^l := \catH$.

Considerando la inclusión $j: \mathscr{C} \longrightarrow \catH ^r$, la precomposición también nos da un isomorfismo como en el corolario \ref{hom_iso}, pero esto no quiere decir que ambas categorías sean isomorfas, ya que ni $i$ ni $j$ mandan equivalencias débiles en equivalencias.

Por otro lado, una homotopía a derecha $K=(P, k)$ y una homotopía a izquierda $H=(C, h)$ están relacionadas si $\widehat{FK}=\widehat{FH}$ para todo $F: \mathscr{C} \longrightarrow \mathscr{D}$ que manda equivalencias d\'ebiles en equivalencias. Si bien $\catH ^r$ y $\catH ^l$ no tienen por qué coincidir, el siguiente resultado nos permitirá establecer cierta correspondencia cuando los objetos sean fibrantes y cofibrantes.

\bigskip
\begin{observation}
    Por el axioma M4 en la definici\'on \ref{model_cat}, dado un objeto $Y$ existe al menos un path-object $P=(V, \delta_0,  \delta_1, \sigma)$ para $Y$ que se obtiene factorizando la diagonal $\Delta_Y$ de la siguiente forma
    $$ \xymatrix{ Y \ar[rr]^{\Delta_Y} \ar@{>->}[dr]|\circ && Y\times Y\\
                                            & V \ar@{->>}[ur]_{(\delta_0, \delta_1)} }.$$
    
\end{observation}

\begin{proposition}
    Sean $f,g: X \longrightarrow Y$, $H: \xymatrix{ f \ar@2{~>}[r]^{l} & g }$ una $q$-homotopía y $P=(V, \delta_0,  \delta_1, \sigma)$ un path-object de $Y$. Si $X$ es cofibrante, entonces existe una $q$-homotopía a derecha $K: \xymatrix{ f \ar@2{~>}[r]^{r} & g }$ con path-object $P$ tal que $[K]=[H]$.
\end{proposition}

\begin{proof}
    Si $H$ es de la forma  $\xymatrix{ X    \ar@<3pt>[rr]^{d_0}
                                            \ar@<-3pt>[rr]_{d_1}
                                            \ar[dr]_{id} && W 
                                            \ar[rr]^{h} \ar[dl]^{s}  && Y, \\
                                            &  X \\  }$
   por el lema \ref{lemma2_quillen} tanto $d_0$ como $d_1$ son cofibraciones triviales, de manera que existe un morfismo $k': W \longrightarrow V$ que hace conmutar el diagrama siguiente
    \begin{center}
        $\xymatrix{ X \ar[rr]^{\sigma f} \ar@{ >->}[dd]_{d_0}|\circ && V \ar@{->>}[dd]^{(\delta_0, \delta_1)} \\
                    \\
                    W \ar@{-->}[uurr]^{\exists \hspace{.5mm} k'} \ar[rr]^{(fs, h)} && Y \times Y \\ }$
    \end{center}
    
    Definiendo $k=K'd_1$ se obtiene una $q$-homotopía a derecha $K$ dada por 
    \begin{center}
        $\xymatrix{
        X \ar[rr]^{k} && V \ar@<3pt>[rr]^{\delta_o} \ar@<-3pt>[rr]_{\delta_1} && Y. \\
        &&& Y \ar[ul]^{\sigma}|\circ \ar[ur]_{id}\\
        }$
    \end{center}
    
    Si $F: \mathscr{C} \longrightarrow \mathscr{D}$, veamos que $\widehat{FH}=\widehat{FK}$. Escribimos $\widehat{FH}=Fh\widehat{Fc}$ y $\widehat{FK}=\widehat{Fp}Fk$, donde además $Fh=F\delta_1 Fk'$ y $Fk=Fk'Fd_1$. Luego,\vspace{5mm}\\
    \begin{tikzcd}
    FX \arrow[bend left=40, "Ff"]{rr}[name=LUU, below]{}
    \arrow[bend right=40, "Fg"']{rr}[name=LDD]{}
    \arrow[Rightarrow,to path=(LUU) -- (LDD)\tikztonodes]{rr}{\widehat{FH}}
    && \hspace{1mm} FY 
    \end{tikzcd}
    $=$
    \begin{tikzcd}
     FX \arrow[bend left=40, "Ff"]{rr}[name=LUU, below]{}
    \arrow[bend right=40, "Fg"']{rr}[name=LDD]{}
    \arrow[Rightarrow,to path=(LUU) -- (LDD)\tikztonodes]{rr}{\widehat{FH}}
    && \hspace{1mm} FY \arrow["F\sigma"]{r}
    & \hspace{1mm} FV \arrow[bend left=40, "F\delta_0"]{rr}[name=RUU, below]{}
    \arrow[bend right=40, "F\delta_1"']{rr}[name=RDD]{}
    \arrow[Rightarrow,to path=(RUU) -- (RDD)\tikztonodes]{rr}{\widehat{Fp}}
    && \hspace{1mm} FY
    \end{tikzcd}\\
    \vspace{3mm}
    $=$
    \begin{tikzcd}
     FX \arrow[bend left=40, "F\sigma f"]{rr}[name=LUU, below]{}
    \arrow[bend right=40, "F\sigma g"']{rr}[name=LDD]{}
    \arrow[Rightarrow,to path=(LUU) -- (LDD)\tikztonodes]{rr}{F\sigma \widehat{FH}}
    && \hspace{1mm} FV \arrow[bend left=40, "F\delta_0"]{rr}[name=RUU, below]{}
    \arrow[bend right=40, "F\delta_1"']{rr}[name=RDD]{}
    \arrow[Rightarrow,to path=(RUU) -- (RDD)\tikztonodes]{rr}{\widehat{Fp}}
    && FY
    \end{tikzcd}\\
    \vspace{3mm}
    $=$
    \begin{tikzcd}
    FX \arrow[bend left=50, "F\sigma f"]{rr}[name=LUU, below]{}
    \arrow{rr}[name=LUD]{}
    \arrow[swap]{rr}[name=LDU]{}
    \arrow[bend right=50, "F\sigma g"']{rr}[name=LDD]{}
    \arrow[Rightarrow,to path=(LUU) -- (LUD)\tikztonodes]{rr}{Id_{F\sigma f}}
    \arrow[Rightarrow,to path=(LDU) -- (LDD)\tikztonodes]{rr}{F\sigma \widehat{FH}}
    &&\hspace{1mm} FV
    \arrow[bend left=50, "F \sigma_0"]{rr}[name=RUU, below]{}
    \arrow{rr}[name=RUD]{}
    \arrow[swap]{rr}[name=RDU]{}
    \arrow[bend right=50, "F\sigma_1"']{rr}[name=RDD]{}
    \arrow[Rightarrow,to path=(RUU) -- (RUD)\tikztonodes]{r}{\widehat{Fp}}
    \arrow[Rightarrow,to path=(RDU) -- (RDD)\tikztonodes]{r}{Id_{F\sigma_1}}
    &&\hspace{1mm} FY
  \end{tikzcd}\\
    \vspace{3mm}
    $=$
    \begin{tikzcd}
    FX \arrow[bend left=50, "Fd_0"]{rr}[name=LUU, below]{}
    \arrow{rr}[name=LUD]{}
    \arrow[swap]{rr}[name=LDU]{}
    \arrow[bend right=50, "Fd_1"']{rr}[name=LDD]{}
    \arrow[Rightarrow,to path=(LUU) -- (LUD)\tikztonodes]{rr}{d_0}
    \arrow[Rightarrow,to path=(LDU) -- (LDD)\tikztonodes]{rr}{\widehat{Fc}}
    &&\hspace{1mm} FW \arrow["Fk'"]{rr}
    &&\hspace{1mm} FV \arrow[bend left=50, "F \sigma_0"]{rr}[name=RUU, below]{}
    \arrow{rr}[name=RUD]{}
    \arrow[swap]{rr}[name=RDU]{}
    \arrow[bend right=50, "F\sigma_1"']{rr}[name=RDD]{}
    \arrow[Rightarrow,to path=(RUU) -- (RUD)\tikztonodes]{r}{\widehat{Fp}}
    \arrow[Rightarrow,to path=(RDU) -- (RDD)\tikztonodes]{r}{Id_{F\sigma_1}}
    &&\hspace{1mm} FY
  \end{tikzcd}\\
  \vspace{3mm}
  $=$
  \begin{tikzcd}
    FX \arrow[bend left=50, "Fd_0"]{rr}[name=LUU, below]{}
    \arrow{rr}[name=LUD]{}
    \arrow[swap]{rr}[name=LDU]{}
    \arrow[bend right=50, "Fd_1"']{rr}[name=LDD]{}
    \arrow[Rightarrow,to path=(LUU) -- (LUD)\tikztonodes]{rr}{\widehat{Fc}}
    \arrow[Rightarrow,to path=(LDU) -- (LDD)\tikztonodes]{rr}{d_1}
    &&\hspace{1mm} FW \arrow["Fk'"]{rr}
    &&\hspace{1mm} FV \arrow[bend left=50, "F \sigma_0"]{rr}[name=RUU, below]{}
    \arrow{rr}[name=RUD]{}
    \arrow[swap]{rr}[name=RDU]{}
    \arrow[bend right=50, "F\sigma_1"']{rr}[name=RDD]{}
    \arrow[Rightarrow,to path=(RUU) -- (RUD)\tikztonodes]{r}{Id_{F\sigma_0}}
    \arrow[Rightarrow,to path=(RDU) -- (RDD)\tikztonodes]{r}{\widehat{Fp}}
    &&\hspace{1mm} FY
  \end{tikzcd} \\
  $=$
  \begin{tikzcd}
    FX \arrow[bend left=50, "Ff"]{rrr}[name=LUU, below]{}
    \arrow{rrr}[name=LUD]{}
    \arrow[swap, "\hspace{-10mm} Ff"']{rrr}[name=LDU]{}
    \arrow[bend right=50, "Fg"']{rrr}[name=LDD]{}
    \arrow[Rightarrow,to path=(LUU) -- (LUD)\tikztonodes]{rrr}{Ffs\widehat{Fc}}
    \arrow[Rightarrow,to path=(LDU) -- (LDD)\tikztonodes]{rrr}{\widehat{Fp}Fk}
    &&&\hspace{1mm} FY
  \end{tikzcd}
  $=$
  \begin{tikzcd}
    FX \arrow[bend left=40, "Ff"]{rr}[name=LUU, below]{}
    \arrow[bend right=40, "Fg"']{rr}[name=LDD]{}
    \arrow[Rightarrow,to path=(LUU) -- (LDD)\tikztonodes]{rr}{\widehat{FK}}
    && \hspace{1mm} FY. 
    \end{tikzcd}

  \vspace{6mm}
  Esto nos dice entonces que $K \thicksim H$, concluyendo la demostración.
 
\end{proof}

De la proposición anterior junto con su versión dual se deduce que cuando los objetos son fibrantes y cofibrantes las clases de $q$-homotopías a derecha se corresponden con las clases de $q$-homotopías a izquierda. Dado que, asi como ocurre con las homotopías a izquierda, en $\mathscr{C}_{fc}$ no distinguimos entre clases de secuencias de homotopías y clases de $q$-homotopías a derecha, entonces $\Ho$ es la misma 2-categoría para cualquiera de las dos versiones.

\bigskip
Adem\'as, de este \'ultimo resultado se deduce tambi\'en que las hom-categor\'ias de la 2-categor\'ia $\Ho$ son localmente pequeñas: 

\begin{corollary}
    Sean $X$, $Y$ fibrantes-cofibrantes, y sean $f, g: X \longrightarrow Y$ morfismos en $\mathscr{C}$. Entonces $\Ho[X,Y][f, g]$ es un conjunto.
\end{corollary}

\begin{proof}
    Dado un path-object fijo, las clases de $q$-homotop\'ias a derecha de $f$ a $g$ correspondientes a este path-object forman un conjunto y, como $X$ e $Y$ son cofibrantes y fibrantes, por la proposici\'on anterior junto con su versi\'on dual dicho conjunto est\'a en biyecci\'on con las clases de $q$-homotop\'ias a izquierda de $f$ a $g$.
\end{proof}


\newpage

\begin{thebibliography}{}
\bibitem{e.d.}
M.E. Descotte, E.J. Dubuc, M. Szyld, \emph{Model bicategories and their homotopy bicategories}, 	arXiv:1805.07749 (2018).
\bibitem{Dwy}
W. G. Dwyer, J. Spalinski, \emph{Homotopy theories and model categories}, Handbook of Algebraic Topology (I. M. James, ed.), Elsevier Science B.V., 1995.
\bibitem{GZ}
P. Gabriel, M. Zisman: \emph{Calculus of Fractions and Homotopy Theory}, Springer-Verlag, New York, 1967.
\bibitem{JG}
P. G. Goerss, J. F. Jardine, \emph{Simplicial Homotopy Theory}, Progress in Math., vol. 174, Birkhauser Verlag, Basel, 1999.
\bibitem{Bak}
J. W. Gray, \emph{Formal Category theory: Adjointness for 2-categories}, Lecture Notes in Mathematics, Vol. 391, Springer-Verlag, Berlin-New York, 1974.
\bibitem{Hir}
P. Hirschhorn, \emph{Model Categories and Their Localizations}, Mathematical Surveys and Monographs Volume 99 (2003).
\bibitem{Hov}
M. Hovey, \emph{Model categories}, Mathematical Surveys and Monographs, vol. 63, Amer. Math. Soc., Providence, RI, 1999.
\bibitem{Mac}
S. Mac Lane, \emph{Categories for the Working Mathematician}, Graduate Texts in Mathematics Volume 5 (1971).
\bibitem{Mil}
J. Milnor, \emph{The geometric realization of a semi–simplicial complex}, Ann. of Math., vol. 65 (1957),
357–362.
\bibitem{Pro}
D. A. Pronk, \emph{Etendues and stacks as bicategories of fractions},  Compositio Mathematica, Volume 102 (1996) no. 3, p. 243-303.
\bibitem{Qui}
D. Quillen, \emph{Homotopical Algebra}, Springer Lecture Notes in Mathematics 43 (1967).
\bibitem{Riehl}
E. Riehl, \emph{Higher Dimensional Categories Model Categories and Weak Factorization Systems}, \url{http://www.math.jhu.edu/~eriehl/essay.pdf}
\bibitem{Tho}
R. W. Thomason, \emph{Cat as a Closed Model Category}, Cahiers de Topologie et Geometrie Differentielle XXI-3. (1980), 305-324.

\end{thebibliography}
\end{document}